\documentclass[3p, sort&compress]{elsarticle}
\usepackage[adobe-utopia]{mathdesign}
\usepackage[T1]{fontenc}
\usepackage[UKenglish]{babel}
\usepackage{amsmath,enumerate,mathtools,xfrac,enumitem,etoolbox,amsthm,geometry}
\usepackage{verbatimbox,pbox,framed,setspace,tikz,caption,wrapfig}

\def\mylinespacing{1.12}
{\bfseries}{\rmfamily}
\newtheorem{theorem}{Theorem}{\bfseries}
\newtheorem{proposition}[theorem]{Proposition}{\bfseries}
\newtheorem{corollary}[theorem]{Corollary}{\bfseries}
\newtheorem{lemma}[theorem]{Lemma}{\bfseries}

\newtheorem{example}{Example}{\itshape}{\rmfamily}

\renewenvironment{proof}{\vspace*{\parsep}\par\pushQED{\qed}\begin{spacing}{1.12}
\par\nointerlineskip\noindent{\itshape Proof. }}{\popQED\end{spacing}\vspace*{\topsep}}

\newenvironment{proofof}[1]{\vspace*{\parsep}\begin{spacing}{1.12}
\par\noindent{\itshape Proof of #1\ }}{\qed\end{spacing}\vspace*{\topsep}}

\usepackage{hyperref,url}
\hypersetup{  
    bookmarksnumbered
  }
  \hypersetup{  
    breaklinks=false,
    pdfborderstyle={/S/U/W 0},  
    citebordercolor=.235 .702 .443,
    urlbordercolor=.255 .412 .882,
    linkbordercolor=.804 .149 .19,
  }

\newcommand{\reals}{\mathbb{R}}

\newcommand{\nats}{\mathbb{N}}
\newcommand{\natz}{\mathbb{N}_{0}}

\newcommand{\indica}[1]{\mathbb{I}_{#1}}

\newcommand{\prev}[1]{\mathrm{E}_{#1}}
\newcommand{\upprev}[1][]{\overline{\mathrm{E}}_{#1}}

\newcommand{\eiupprev}[1]{
\makebox{$
\pbox[b][1em]{1.5em}{$\overline{\mathrm{E}}$}
_{\raisebox{1pt}[\height][0pt]{\scriptsize$#1$}}
^{\, \raisebox{-1pt}{\scriptsize\textnormal{ei}}}$}} 

\newcommand{\ciupprev}[1]{
\makebox{$
\pbox[b][1em]{1.5em}{$\overline{\mathrm{E}}$}
_{\raisebox{1pt}[\height][0pt]{\scriptsize$#1$}}
^{\, \raisebox{-1pt}{\scriptsize\textnormal{ci}}}$}} 

\newcommand{\riupprev}[1]{
\makebox{$
\pbox[b][1em]{1.5em}{$\overline{\mathrm{E}}$}
_{\raisebox{1pt}[\height][0pt]{\makebox[1.19\width][l]{\scriptsize$#1$}}}
^{\, \raisebox{-1pt}{\scriptsize\textnormal{ri}}}$} 
}

\newcommand{\avriupprev}[1]{
\makebox{$
\pbox[b][1em]{1.5em}{$\overline{\mathrm{E}}$}
_{\raisebox{1pt}[\height][0pt]{\makebox[0.9\width][l]{\scriptsize$\mathrm{av}$,$#1$}}}
^{\, \raisebox{-1pt}{\scriptsize\textnormal{ri}}}$} \,
}

\newcommand{\lowprev}[1]{\underline{\mathrm{E}}_{#1}}
\newcommand{\eilowprev}[1]{
\makebox{$
\pbox[b][1em]{1em}{$\underline{\mathrm{E}}$}
_{\, \raisebox{1pt}[\height][0pt]{\makebox[1.17\width][l]{\scriptsize$#1$}}}
^{\, \raisebox{-1pt}{\text{\scriptsize{\textnormal{ei}}}}}$} 
}
\newcommand{\cilowprev}[1]{
\makebox{$
\pbox[b][1em]{1em}{$\underline{\mathrm{E}}$}
_{\, \raisebox{1pt}[\height][0pt]{\makebox[1.17\width][l]{\scriptsize$#1$}}}
^{\, \raisebox{-1pt}{\text{\scriptsize{\textnormal{ci}}}}}$} 
}
\newcommand{\rilowprev}[1]{
\makebox{$
\pbox[b][1em]{1em}{$\underline{\mathrm{E}}$}
_{\, \raisebox{1pt}[\height][0pt]{\makebox[1.17\width][l]{\scriptsize$#1$}}}
^{\, \raisebox{-1pt}{\text{\scriptsize{\textnormal{ri}}}}}$}
}

\newcommand{\eiupprob}[1]{
\makebox{$
\pbox[b][1em]{1.5em}{$\overline{\mathrm{P}}$}
_{\raisebox{1pt}[\height][0pt]{\scriptsize$#1$}}
^{\, \raisebox{-1pt}{\scriptsize\textnormal{ei}}}$}}

\newcommand{\eilowprob}[1]{
\makebox{$
\pbox[b][1em]{1.5em}{$\underline{\mathrm{P}}$}
_{\raisebox{1pt}[\height][0pt]{\scriptsize$#1$}}
^{\, \raisebox{-1pt}{\scriptsize\textnormal{ei}}}$}}

\newcommand{\statespace}{\mathscr{X}}

\newcommand{\setofgambles}{\mathscr{L}}

\newcommand{\settrans}[1][\,\,]{\mathscr{T}_{#1}}

\newcommand{\uptrans}{\overline{T}}

\newcommand{\setofprocesses}[2][]{\mathscr{P}_{#2}^{\hspace*{-1pt}#1}}
\newcommand{\eimarkov}[1]{\smash{\setofprocesses[\, \mathrm{ei}]{#1}}}
\newcommand{\cimarkov}[1]{\smash{\setofprocesses[\, \mathrm{ci}]{#1}}}
\newcommand{\rimarkov}[1]{\smash{\setofprocesses[\, \mathrm{ri}]{#1}}}

\newcommand{\avuptrans}[2]{\smash{\overline{T}{}_{\hspace*{-2pt} #1}^{#2}}}

\DeclarePairedDelimiter{\hnorm}{\lVert}{\rVert_{\mathrm{H}}}
\DeclarePairedDelimiter{\supnorm}{\lVert}{\rVert_{\infty}}

\begin{document}

\title{Average Behaviour in Discrete-Time Imprecise Markov Chains: \\ A Study of Weak Ergodicity}
\author{Natan T'Joens}
\ead{natan.tjoens@ugent.be}
\author{Jasper De Bock}
\address{FLip, Ghent University, Belgium}

\begin{abstract}
We study the limit behaviour of upper and lower bounds on expected time averages in imprecise Markov chains; a generalised type of Markov chain where the local dynamics, traditionally characterised by transition probabilities, are now represented by sets of `plausible' transition probabilities.
Our first main result is a necessary and sufficient condition under which these upper and lower bounds, called upper and lower expected time averages, will converge as time progresses towards infinity to limit values that do not depend on the process' initial state.
Our condition is considerably weaker than that needed for ergodic behaviour;
a similar notion which demands that marginal upper and lower expectations of functions at a single time instant converge to so-called limit---or steady state---upper and lower expectations.
For this reason, we refer to our notion as `weak ergodicity'.
Our second main result shows that, as far as this weakly ergodic behaviour is concerned, one should not worry about which type of independence assumption to adopt---epistemic irrelevance, complete independence or repetition independence.
The characterisation of weak ergodicity as well as the limit values of upper and lower expected time averages do not depend on such a choice. 
Notably, this type of robustness is not exhibited by the notion of ergodicity and the related inferences of limit upper and lower expectations.
Finally, though limit upper and lower expectations are often used to provide approximate information about the limit behaviour of time averages, we show that such an approximation is sub-optimal and that it can be significantly improved by directly using upper and lower expected time averages.
\end{abstract}


\maketitle

\section{Introduction}
Markov chains \cite{kemeny1960finite,shiryaev1995probability} are probabilistic models that are used to describe the uncertain dynamics of a large variety of stochastic processes.
One of the key results in the field is the point-wise ergodic theorem.
It establishes a relation between the long-term time average $f_{\mathrm{av}}(X_{1:k}) = \tfrac{1}{k}\sum_{i=1}^k f(X_i)$ of a real-valued function $f$ and its limit expectation $\mathrm{E}_\infty(f) = \lim_{k\to+\infty}\mathrm{E}(f(X_k))$, which is guaranteed to exist if the Markov chain is ergodic.\footnote{
The term ergodicity has various meanings; sometimes it refers to properties of an invariant measure, sometimes it refers to properties such as irreducibility (with or without aperiodicity), regularity, ... 
Our usage of the term follows conventions introduced in earlier work \cite{DECOOMAN201618,Hermans:2012ie} on imprecise Markov chains; see Sections~\ref{section: precise Markov chains} and~\ref{section: trans operators and ergodicity}.}\label{footnote} 
For this reason, limit expectations and limit distributions have become central objects of interest.
Of course, if one is interested in the long-term behaviour of time averages, one could also study the expected values $\mathrm{E}(f_{\mathrm{av}}(X_{1:k}))$ of these averages directly.
This is not often done though, because, if the Markov chain is ergodic, the limit of these expected time averages coincides with the aforementioned limit expectations, which can straightforwardly be obtained by solving a linear eigenproblem \cite{kemeny1960finite}. 
However, for a Markov chain that is not ergodic, 
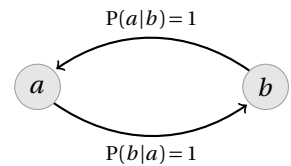
\begin{wrapfigure}{r}{8cm}
\begin{center}
\vspace*{-0.40cm}
\begin{tikzpicture}
\tikzset{pil/.style={->,thick,shorten >=1pt}}
\tikzstyle{place}=[circle,inner sep=2pt,draw=gray!75,fill=gray!20,minimum size=6mm]
\node [place] (a) at (0,0) {$a$};
\node [place] (b) at (3,0) {$b$};

\draw[pil, bend right = 35] (b.north west) to node[above]{\footnotesize $\mathrm{P}(a \vert b) = 1$} (a.north east);
\draw[pil, bend right = 35] (a.south east) to node[below]{\footnotesize $\mathrm{P}(b \vert a) = 1$} (b.south west);
\end{tikzpicture}
\caption{The graph above represents a cyclic Markov chain with two states $a$ and $b$.
Suppose that we want to have an idea about the average number of times that the process will be in state $b$.
So, we are interested in the limit behaviour of $f_{\mathrm{av}}(X_{1:k})$ where $f(a) \coloneqq 0$ and $f(b) \coloneqq 1$.
Though the expectation $\mathrm{E}_\infty(f)$ does not exist (if the initial distribution differs from the uniform distribution), the expectation $\mathrm{E}(f_{\mathrm{av}}(X_{1:k}))$ converges to $1/2$ irrespectively of the process' initial state or distribution.
This value is clearly representative for the long-term average of $f$.
}\label{Figure 1}
\end{center}
\vspace*{-0.50cm}
\end{wrapfigure}
the limit expectation $\mathrm{E}_\infty(f)$ does not necessarily exist and can therefore not be used to provide us with information about the average behaviour of $f$.
The expected time average $\mathrm{E}(f_{\mathrm{av}}(X_{1:k}))$---or its limit for $k\to+\infty$, if it exists---then serves as a seemingly suitable alternative; 
the figure on the right depicts a basic example where this is the case.
So we see that even in the context of traditional ``precise'' Markov chains, expected time averages have the potential to be more informative about the long-term average of $f$ compared to the limit expectation $\mathrm{E}_\infty(f)$.

In this work, we consider a generalisation of Markov chains, called imprecise Markov chains \cite{deCooman:2009jz,HermansITIP,DECOOMAN201618}, for which the study of long-term average behaviour becomes somewhat more complex.
Imprecise Markov chains are sets of traditional ``precise'' probabilistic models, where the Markov property (history independence) and the time-homogeneity property apply to this set of precise models as a whole, but not necessarily to the individual models themselves.
In fact, one distinguishes between three different types of imprecise Markov chains (IMC's):\,\footnote{A fourth type of imprecise Markov chain that is often encountered in the literature, especially in the more general context where imprecise Markov chains are simply regarded as special credal networks, are IMC's under strong independence \cite{HermansITIP,Cozman:2000ug,Cozman:2012fc,Antonucci:2014ty}; convex hulls of IMC's under complete independence.
However, as one of us argues in \cite[Section~3]{DEBOCK2017107} for the case of credal networks, we are of the opinion that such models lack a clear and sensible meaning.
Moreover, the resulting upper and lower expectations---the inferences that we will be interested in here---are identical to those for IMC's under complete independence.
We will therefore not consider them in our study.}
\begin{itemize}
\item \emph{IMC under epistemic irrelevance:} the individual models do not (necessarily) satisfy the Markov property, nor the time-homogeneity property. 
\item \emph{IMC under complete independence:} the individual models satisfy the Markov property, but not (necessarily) the time-homogeneity property. 
\item \emph{IMC under repetition independence:} the individual models satisfy both the Markov property and the time-homogeneity property.
\end{itemize}
So an imprecise Markov chain under repetition independence only allows one to incorporate model uncertainty about the numerical values of the transition probabilities that make up a Markov chain, while an imprecise Markov chain under epistemic irrelevance also allows one to take into account uncertainty about the structural assumptions of being time-homogeneous and satisfying the Markov property.
Regardless of the type of imprecise Markov chain that is used, one is typically interested in obtaining tight upper and lower bounds on inferences for the individual constituting models.
The operators that represent these upper and lower bounds are respectively called upper and lower expectations and we will, for the time being, denote them by $\upprev{}(\cdot)$ and~$\lowprev{}(\cdot)$ respectively.

Just like ergodicity in traditional Markov chains,
an imprecise Markov chain is said to be ergodic if the limit upper expectation $\upprev{}_{\infty}(f) = \lim_{k\to+\infty}\upprev{}(f(X_k))$ and the limit lower expectation $\lowprev{}{}_{\infty}(f)=\lim_{k\to+\infty}\lowprev{}(f(X_k))$ exist and do not depend on the process' initial state or distribution.
There are necessary and sufficient conditions for when this is the case \cite{Hermans:2012ie} as well as an imprecise variant of the point-wise ergodic theorem \cite{DECOOMAN201618}.\footnote{These results only hold for imprecise Markov chains under epistemic irrelevance and under complete independence.}
An important difference with traditional Markov chains, however, is that even if an imprecise Markov chain is ergodic and the limit upper expectation $\upprev{}_{\infty}(f)$ and the limit lower expectation $\lowprev{}{}_{\infty}(f)$ exist, the upper and lower expected time averages $\upprev{}(f_{\mathrm{av}}(X_{1:k}))$ and $\lowprev{}(f_{\mathrm{av}}(X_{1:k}))$ may not converge to them---that is, to $\upprev{}_{\infty}(f)$ and $\lowprev{}{}_{\infty}(f)$, respectively.
Nevertheless, because they (i) give conservative bounds \cite[Lemma~57]{8535240}, (ii) are fairly easy to compute \cite{deCooman:2009jz} and (iii) satisfy a point-wise ergodic theorem \cite{DECOOMAN201618}, the inferences $\upprev{}_{\infty}(f)$ and $\lowprev{}{}_{\infty}(f)$ are often used as descriptors of the long-term behaviour of imprecise Markov chains, even if one is actually interested in time averages.
This comes at a cost though: as we will show in Section~\ref{section: trans operators and ergodicity}, both types of inferences can differ greatly, with limit upper and lower expectations sometimes providing far too conservative bounds.

Unfortunately, apart from some experiments in \cite{8535240}, little is known about the long-term behaviour of the upper and lower expected time averages $\upprev{}(f_{\mathrm{av}}(X_{1:k}))$ and $\lowprev{}(f_{\mathrm{av}}(X_{1:k}))$.
The aim of this paper is to remedy this situation.
Our main result is an accessibility condition that is necessary and sufficient for these upper and lower expected time averages to (each) converge to a limit value that does not depend on the process' initial state (or distribution).
Remarkably, this condition is considerably weaker than the one required for ergodicity.
This explains why we call this type of behaviour `weak ergodicity' (or `weakly ergodic behaviour').
Moreover, we also show that this notion of weak ergodicity does not depend on the adopted type of imprecise Markov chain; whether one considers an imprecise Markov chain under epistemic irrelevance, complete independence or repetition independence is not relevant for the weakly ergodic behaviour of the Markov chain. 
More precisely, given sufficient model parameters, both the accessibility condition that characterises weak ergodicity, as well as---if this condition is satisfied---the limit values of the inferences $\upprev{}(f_{\mathrm{av}}(X_{1:k}))$ and $\lowprev{}(f_{\mathrm{av}}(X_{1:k}))$ are the same, no matter what kind of IMC we consider.
Conventional ergodicity does not exhibit this kind of robustness; we illustrate this in Example~\ref{example 2}.
This provides yet another argument for why (limits of) upper and lower expected time averages---$\upprev{}(f_{\mathrm{av}}(X_{1:k}))$ and $\lowprev{}(f_{\mathrm{av}}(X_{1:k}))$---candidate as the objects of interest when looking at the long-term average behaviour of imprecise Markov chains.

The outline of the paper is as follows. 
We start by introducing ``precise'' Markov chains in Section~\ref{section: precise Markov chains} and subsequently generalise towards the case of imprecise Markov chains in Section~\ref{Sect: imprecise  Markov chains}.
We then focus, as a first step, on average behaviour in imprecise Markov chains under epistemic irrelevance and complete independence, temporarily leaving imprecise Markov chains under repetition independence out of the picture.
As mentioned before, we will study two types of inferences: (limits of) upper and lower expectations of a function evaluated at a single time instant, and (limits of) upper and lower expected time averages of a function.
In Section~\ref{section: trans operators and ergodicity}, we give recursive expressions for how these inferences evolve through time, introduce the notions of ergodicity and weak ergodicity, and moreover illustrate, using two basic examples, that weak ergodicity has some considerable advantages over ``conventional'' ergodicity when it comes to characterising average behaviour.
Section~\ref{Sect: accessibility and topical maps} introduces essential mathematical machinery needed in order to arrive to our results in Sections~\ref{Sect: A Sufficient Condition for Weak Ergodicity}--\ref{sect: weak ergo in repetition independence}; we explain what it means for a map to be topical and introduce some graph-theoretic notions. 
In the subsequent section, Section~\ref{Sect: A Sufficient Condition for Weak Ergodicity}, we derive a sufficient condition for weak ergodicity by borrowing an eigenproblem result from the theory of topical maps.
Section~\ref{section: Result} then shows that this condition can be replaced by a weaker one that is not only sufficient, but also necessary.
Finally, we consider the case of imprecise Markov chains under repetition independence and relate their weak ergodicity to that of imprecise Markov chains under epistemic irrelevance or complete independence.
This will be the subject of Section~\ref{sect: weak ergo in repetition independence}.
As mentioned before, it will turn out that for all three types of imprecise Markov chains, weak ergodicity is characterised by the same condition and the limit upper (lower) expected time averages are all equal.

This paper extends upon an earlier conference paper \cite{TJoens_IPMU2020_weak_ergodicity}; 
we provide proofs for the results in \cite{TJoens_IPMU2020_weak_ergodicity}, and extend our study of weak ergodicity to also include imprecise Markov chains under repetition independence.
In order not to lose the reader's focus, we have chosen to relegate some of the more technical proofs to an appendix at the end of the paper. 
This is particularly true for the results in Sections~\ref{section: Result} and~\ref{sect: weak ergo in repetition independence}, where the main text provides, for the most part, an informal argument that aims to provide intuition. 

\section{Markov Chains}\label{section: precise Markov chains}
  
We consider an infinite sequence $X_{1} X_{2} X_{3} \cdots$ of uncertain states, where each state $X_k$ at time $k \in \nats$ takes values in some finite set $\statespace{}$, called the \emph{state space}.
Such a sequence $X_{1} X_{2} X_{3} \cdots$ will be called a \emph{(discrete-time) stochastic process}.
For any $k,\ell \in \nats{}$ such that $k \leq \ell$, we use $X_{k:\ell}$ to denote the finite subsequence $X_{k} \cdots X_{\ell}$ of states that takes values in $\statespace{}^{\ell-k+1}$.
Moreover, for any $k,\ell \in \nats{}$ such that $k \leq \ell$ and any $x_{k:\ell} \in \statespace{}^{\ell-k+1}$, we use $X_{k:\ell} = x_{k:\ell}$ to denote the event that $X_{k} = x_k \cdots X_{\ell} = x_\ell$.
The uncertain dynamics of a stochastic process are then typically described by probabilities of the form $\mathrm{P}(X_{k+1} = x_{k+1} \vert X_{1:k} = x_{1:k})$, for any $k \in \nats{}$ and any $x_{1:k+1} \in \statespace{}^{k+1}$.
They represent beliefs about which state the process will be in at time $k+1$ given that we know that it was in the states $x_{1} \cdots x_{k}$ at time instances $1$ through $k$.
Additionally, our beliefs about the value of the initial state $X_1$ can be represented by probabilities $\mathrm{P}(X_1=x_1)$ for all $x_{1} \in \statespace{}$.
The local probability assessments $\mathrm{P}(X_{k+1} = x_{k+1} \vert X_{1:k} = x_{1:k})$ and $\mathrm{P}(X_1=x_1)$ can now be combined to construct a global probability model $\mathrm{P}$ that describes the dynamics of the process on a more general level.
This can be done in various ways; one of the most common ones being a measure-theoretic approach where countable additivity plays a central role.
For our purposes however, we will only require finite additivity.
Regardless, once you have such a global probability model $\mathrm{P}$, it can then be used to define expectations and make inferences about the uncertain behaviour of the process.

For any set $A$, let us write $\setofgambles{}(A)$ to denote the set of all real-valued functions on $A$.
Throughout, for any $B \subseteq A$, we use $\indica{B}$ to denote the \emph{indicator} of $B$: the function in $\setofgambles{}(A)$ that takes the value $1$ in $B$ and $0$ otherwise.
We will only be concerned with (upper and lower) expectations of \emph{finitary functions}:
functions that depend on the state of the process at a finite number of time instances.
So if $f$ is finitary, we can write $f = g(X_{1:k})$ for some $k \in \nats{}$ and some $g \in \setofgambles{}(\statespace{}^{k})$.
Note that finitary functions are bounded; this follows from their real-valuedness and the fact that $\statespace{}$ is finite.
The expectation of a finitary function $f(X_{1:k})$ conditional on some event $X_{1:\ell} = x_{1:\ell}$, with $\ell<k$, simply reduces to a finite weighted sum:
\begin{align*}
\prev{\mathrm{P}}(f(X_{1:k}) \vert X_{1:\ell} = x_{1:\ell})
= 
\sum_{x_{\ell+1:k} \in \statespace{}^{k-\ell}} f(x_{1:k})\prod_{i=\ell}^{k-1} \mathrm{P}(X_{i+1} = x_{i+1} \vert X_{1:i} = x_{1:i}).
\end{align*} 
A particularly interesting case arises when studying stochastic processes that are described by a probability model $\mathrm{P}$ that satisfies
\begin{align*}
\mathrm{P}(X_{k+1} = y \, \vert \,  X_{1:k} = x_{1:k}) = \mathrm{P}(X_{k+1} = y \, \vert \,  X_{k} =  x_{k}),
\end{align*}
for all $k \in \nats{}$, all $y \in \statespace{}$ and all $x_{1:k} \in \statespace{}^{k}$.
This property, known as the \emph{Markov property}, states that given the present state of the process the future behaviour of the process does not depend on its  history. 
A process of this type is called a \emph{Markov chain}.
We moreover call it \emph{(time) homogeneous} if additionally
$\mathrm{P}(X_{k+1} = y \, \vert \, X_{k} = x) = \mathrm{P}(X_{2} = y \, \vert \, X_{1} = x)$,
for all $k \in \nats{}$ and all $x,y \in \statespace{}$.
Hence, together with the assessments $\mathrm{P}(X_1 = x_1)$, the dynamics of a homogeneous Markov chain are fully characterised by the probabilities $\mathrm{P}(X_{2} = y \, \vert \, X_{1} = x)$.
These probabilities are typically gathered in a \emph{transition matrix} $T$; a row-stochastic $\vert \statespace{} \vert  \times \vert \statespace{} \vert$ matrix $T$ that is defined by $T(x,y) \coloneqq \mathrm{P}(X_{2} = y \, \vert \, X_{1} = x)$ for all $x,y \in \statespace{}$.
This matrix representation~$T$ can be regarded as a linear operator from $\setofgambles{}(\statespace{})$ to $\setofgambles{}(\statespace{})$, defined for any $f \in \setofgambles{}(\statespace{})$ and any $x \in \statespace{}$ by 
\begin{align*}
Tf (x) \coloneqq 
\sum_{y \in \statespace{}} f(y) \mathrm{P}(X_{2} = y \, \vert X_{1} = x) 
= \prev{\mathrm{P}}(f(X_{2}) \, \vert \, X_{1} = x ).
\end{align*}
Conveniently, for any $k\in\nats$, we also have that 
\begin{align*}
\prev{\mathrm{P}}(f(X_{k+1}) \, \vert \, X_{k} = x )
= \sum_{y \in \statespace{}} f(y) \mathrm{P}(X_{k+1} = y \, \vert X_{k} = x)
= \sum_{y \in \statespace{}} f(y) \mathrm{P}(X_{2} = y \, \vert X_{1} = x) 
= Tf (x).
\end{align*}
More generally, it holds that $\prev{\mathrm{P}}(f(X_{k+\ell}) \, \vert \, X_{k} = x) = T^{\ell} f(x)$ for all $k \in \nats{}$, all $\ell \in \natz{}\coloneqq \nats \cup \{0\}$ and all $x \in \statespace{}$.
Then, under some well-known accessibility conditions \cite[Proposition~3]{Hermans:2012ie}, the expectation $T^{\ell} f(x)$ converges for increasing $\ell$ towards a constant $\prev{\infty}(f)$ independently of the initial state $x$.
If this is the case for all $f \in \setofgambles{}(\statespace{})$, the homogeneous Markov chain will have a steady-state distribution, represented by the limit expectation $\prev{\infty}$, and we call the Markov chain \emph{ergodic}. 
The expectation $\prev{\infty}$ is in particular also useful if we are interested in the limit behaviour of expected time averages.
Indeed, let $f_{\mathrm{av}}(X_{1:k}) \coloneqq \frac{1}{k} \sum_{i = 1}^{k} f(X_i)$ be the time average of some function $f \in \setofgambles{}(\statespace{})$ evaluated at the time instances $1$ through $k$.
Then, according to \cite[Theorem~38]{8535240}, the limit of the expected average $\lim_{k \to +\infty} \prev{\mathrm{P}}(f_{\mathrm{av}}(X_{1:k}))$ coincides with the limit expectation $\prev{\infty}(f)$.
One of the aims of this paper is to explore to which extent this remains true for imprecise Markov chains.

\section{Imprecise Markov Chains}\label{Sect: imprecise  Markov chains}

If the basic probabilities $\mathrm{P}(X_{k+1} \vert X_{1:k} = x_{1:k})$ that describe a stochastic process are imprecise, in the sense that we only have partial information about them, then we can still model the process' dynamics by considering a set $\settrans[x_{1:k}]$ of such probabilities, for all $k \in \nats{}$ and all $x_{1:k} \in \statespace{}^{k}$.
This set $\settrans[x_{1:k}]$ is then interpreted as the set of all probability mass functions $\mathrm{P}(X_{k+1} \vert X_{1:k} = x_{1:k})$ that we deem ``plausible''.
We here consider the special case where the sets $\settrans[x_{1:k}]$ satisfy a Markov property, meaning that $\settrans[x_{1:k}] = \settrans[x_{k}]$ for all $k \in \nats{}$ and all $x_{1:k} \in \statespace{}^{k}$.
Similarly to the precise case, the sets $\settrans[x]$, for all $x \in \statespace{}$, can be gathered into a single object: the set $\settrans{}$ of all row stochastic $\vert \statespace{} \vert \times \vert \statespace{} \vert$ matrices $T$ such that, for all $x \in \statespace{}$, the probability mass function $T(x, \cdot )$ is an element of $\settrans[x]$. 
A set $\settrans{}$ of transition matrices defined in this way is called \emph{separately specified} \cite[Definition~11.6]{HermansITIP}; this property asserts that, for any two transition matrices $T_1, T_2 \in\settrans$ and any subset $A\subseteq \statespace{}$, there is a third transition matrix $T_3\in\settrans$ such that $T_3(x,\cdot) = T_1(x,\cdot)$ for all $x\in A$ and $T_3(y,\cdot) = T_2(y,\cdot)$ for all $y\in \statespace{}\setminus A$.
For any such set $\settrans{}$, the corresponding \emph{imprecise Markov chain under epistemic irrelevance}\/ $\eimarkov{\settrans{}}$ \cite{deCooman:2010gd,deCooman:2009jz} is the set of all (precise) probability models $\mathrm{P}$ such that $\smash{\mathrm{P}(X_{k+1} \vert X_{1:k} = x_{1:k}) \in \settrans[x_k]}$ for all $k \in \nats{}$ and all $\smash{x_{1:k} \in \statespace{}^{k}}$.
The values of the probabilities $\mathrm{P}(X_1 = x_1)$ will be of no importance to us, because we will focus solely on (upper and lower) expectations conditional on the value of the initial state $X_1$.

Clearly, an imprecise Markov chain $\eimarkov{\settrans{}}$ also contains non-homogeneous, and even non-Markovian processes.
So the Markov property does in this case not apply to the individual probability assessments, but rather to the sets $\settrans[x_{1:k}]$.
The model $\eimarkov{\settrans{}}$ is therefore a generalisation of a traditional Markov chain where we allow for model uncertainty about, on the one hand, the mass functions $\mathrm{P}(X_{k+1} \vert X_{1:k} = x_{1:k})$ and, on the other hand, about structural assumptions such as the Markov and time-homogeneity property.
In order to make inferences that are robust with respect to this model uncertainty, we will use \emph{upper} and \emph{lower expectations} \cite{troffaes2014,Walley:1991vk,Augustin:2014di}.
These operators are respectively defined as the tightest upper and lower bound on the expectation $\mathrm{E}_\mathrm{P}$ associated with any probability model $\mathrm{P}$ in $\eimarkov{\settrans}$: 
\begin{align*}
\eiupprev{\settrans{}} (f \/ \vert A) \coloneqq \sup_{\mathrm{P} \in \eimarkov{\settrans{}}} \prev{\mathrm{P}}(f \/ \vert  A) \quad 
\text{ and } \quad \eilowprev{\settrans{}} (f  \vert  A) \coloneqq \inf_{\mathrm{P} \in \eimarkov{\settrans{}}} \prev{\mathrm{P}}(f  \vert  A),
\end{align*}
for any finitary function $f$ and any event $A$ of the form $X_{1:k} = x_{1:k}$. 
The operators $\smash{\eiupprev{\settrans{}}}$ and $\smash{\eilowprev{\settrans{}}}$ are related by conjugacy, meaning that $\smash{\eilowprev{\settrans{} }(\cdot \vert \cdot) = - \eiupprev{\settrans{}}(- \cdot \vert \cdot)}$, which allows us to focus on only one of them; upper expectations in our case.
The lower expectation $\smash{\eilowprev{\settrans{}}(f \/ \vert A)}$ of a finitary function $f$ can then simply be obtained by considering the upper expectation $\smash{- \eiupprev{\settrans{}}(- f \/ \vert A)}$. 
Moreover, note that \emph{upper} and \emph{lower probabilities} can simply be regarded as special cases of upper and lower expectations: for any event $B$ such that $\indica{B}$ is finitary and any event $A$ of the form $X_{1:k} = x_{1:k}$, we can define them as respectively $\eiupprob{\settrans}(B \vert A) \coloneqq \eiupprev{\settrans{}} (\indica{B} \/ \vert A)$ and $\eilowprob{\settrans}(B \vert A) \coloneqq \eilowprev{\settrans{}} (\indica{B} \/ \vert A)$.

Apart from epistemic irrelevance, there are also other types of independence assumptions for imprecise Markov chains that impose more stringent conditions on the individual composing probability models.
For a given set $\settrans{}$, the imprecise Markov chain under \emph{complete independence} $\smash{\cimarkov{\settrans{}}}$ is the subset of $\smash{\eimarkov{\settrans{}}}$ that contains all---possibly non-homogeneous---Markov chains in $\smash{\eimarkov{\settrans{}}}$~\cite{Lecture_Seidenfeld_2007,Cozman:2012fc,8535240}.
The models $\cimarkov{\settrans}$, also known as `Markov set-chains' \cite{Hartfiel1998}, were the first types of imprecise Markov chains to be thoroughly studied.
They can be motivated in a rather straightforward way, using a `sensitivity analysis interpretation'; the set $\settrans$ is then regarded as a result of our ignorance about some ``true'' transition matrix $T_k$ that may depend on the time $k$.
A third type of imprecise Markov chain that we will associate with $\settrans$ is the corresponding imprecise Markov chain under \emph{repetition independence} $\smash{\rimarkov{\settrans{}}}$, which is the subset of $\smash{\eimarkov{\settrans{}}}$ containing all homogeneous Markov chains \cite{couso_moral_walley_2000,8535240}.
Similarly as for $\cimarkov{\settrans}$, the model $\smash{\rimarkov{\settrans{}}}$ can be motivated using a sensitivity analysis interpretation, where the unknown matrix $T_k$ is now assumed to be fixed in time. 
Observe that the models $\cimarkov{\settrans}$ and $\smash{\rimarkov{\settrans{}}}$ do not allow us to incorporate uncertainty about the Markov assumption, a feature that only imprecise Markov chains under epistemic irrelevance have.
Moreover, though imprecise Markov chains under epistemic irrelevance can also be justified starting from a sensitivity analysis interpretation---the underlying `true' probability model is in that case not assumed to be Markov---they are especially suitable when we regard the sets $\settrans[x]$ as arising from the---subjective---beliefs of a subject that is uncertain about the process' next state value, like Walley does~\cite{Walley:1991vk}.
The fact that these sets $\settrans[x]$ satisfy a Markov property, then simply means that our subject's beliefs are solely based on the current state $x$ of the process; we refer to \cite{DEBOCK2017107,deCooman:2010gd} for further details.

Similarly to how we defined upper and lower expectations for imprecise Markov chains under epistemic irrelevance, we can define the upper expectations $\smash{\ciupprev{\settrans{}}}$ and $\smash{\riupprev{\settrans{}}}$ (and the lower expectations $\smash{\cilowprev{\settrans{}}}$ and $\smash{\rilowprev{\settrans{}}}$) as the tightest upper (and lower) bounds on the expectations corresponding to the models in $\cimarkov{\settrans{}}$ and $\rimarkov{\settrans{}}$, respectively.
Upper and lower probabilities can also be defined in the same way as before; as upper and lower expectations of indicators.
Furthermore, note that, since $\rimarkov{\settrans{}} \subseteq \cimarkov{\settrans{}} \subseteq \eimarkov{\settrans{}}$, we have that 
\begin{align*}
\smash{\riupprev{\settrans{}} (f \/ \vert A) 
\leq \ciupprev{\settrans{}} (f \/ \vert A)
\leq \eiupprev{\settrans{}} (f \/ \vert A)},
\end{align*}
for any finitary function~$f$ and any event~$A$ of the form $X_{1:k} = x_{1:k}$.
Henceforth, we let $\settrans{}$ be some generic set of transition matrices that is separately specified.

In this paper, we will be specifically concerned with two types of inferences: the conditional upper (and lower) expectation of a function $f\in\setofgambles{}(\statespace{})$ evaluated at a single time instant $k$, and the conditional upper (and lower) expectation of the time average $f_{\mathrm{av}}(X_{1:k})$ of a function $f\in\setofgambles{}(\statespace{})$, given that we start in some $x\in\statespace{}$.
For imprecise Markov chains under epistemic irrelevance and under complete independence, both of these inferences coincide \cite[Theorem~51 \& Theorem~52]{8535240}.
For any $f \in \setofgambles{}(\statespace{})$, any $x \in \statespace{}$ and any $k\in\nats$, we will denote them by
\vspace*{-4pt}
\begin{align*}
\upprev[k](f \vert x)
= \, &\eiupprev{\settrans{}}(f(X_{k}) \vert X_1 = x) 
= \ciupprev{\settrans{}}(f(X_{k}) \vert X_1 = x) \\
\text{ and } \  
\upprev[\mathrm{av},k](f \vert x)
= \, &\eiupprev{\settrans{}}(f_{\mathrm{av}}(X_{1:k}) \vert X_1 = x) 
= \ciupprev{\settrans{}}(f_{\mathrm{av}}(X_{1:k}) \vert X_1 = x),
\end{align*}
where the dependency on $\settrans{}$ is implicit.
The corresponding lower expectations can be obtained through conjugacy: 
$\lowprev{k}(f \vert x) = - \upprev[k]( -f \vert x)$ and $\lowprev{\mathrm{av},k}(f \vert x) = - \upprev[\mathrm{av},k]( -f \vert x)$ for all $f \in \setofgambles{}(\statespace{})$, all $x \in \statespace{}$ and all $k\in\nats$.
As we will discuss shortly, the behaviour (or evolution) of both of these inferences can be recursively expressed in terms of a single so-called upper transition operator $\uptrans{}$.
These relations will form the starting point for our further study of the limit behaviour of these inferences.
However, similar expressions seem not to exist for the upper expectations $\riupprev{\settrans}(f(X_k)\vert X_1 = x)$ and $\riupprev{\settrans}(f_{\mathrm{av}}(X_{1:k})\vert X_1 = x)$ corresponding to an imprecise Markov chain $\rimarkov{\settrans}$ under repetition independence.
As a consequence, such inferences demand a somewhat different approach.
For the moment, we therefore omit them from our discussion.
We will come back to them in Section~\ref{sect: weak ergo in repetition independence}.

\section{Transition Operators, Ergodicity and Weak Ergodicity}\label{section: trans operators and ergodicity}
 
Inferences of the form $\upprev[k](f \vert x)$---and, more specifically, upper (and lower) probabilities of events of the form $X_k \in A$, with $A \subseteq \statespace{}$---were among the first ones to be thoroughly studied in imprecise Markov chains \cite{Hartfiel1998,deCooman:2009jz,SKULJ20091314}.
Recent work on the topic \cite{deCooman:2009jz,Hermans:2012ie,DECOOMAN201618} is crucially based on the observation that $\upprev[k+1](f \vert x)$ can be elegantly rewritten as the $k$-th iteration of the map $\uptrans{} \colon \setofgambles{}(\statespace{}) \to \setofgambles{}(\statespace{})$ defined by
\begin{align*}
\uptrans{} h (x) 
\coloneqq \sup_{T \in \, \settrans{}} \, T h (x)
= \sup_{T(x , \cdot) \, \in \, \settrans[x]} \, \sum_{y \in \statespace{}} T(x,y) h(y),
\end{align*}
for all $x \in \statespace{}$ and all $h \in \setofgambles{}(\statespace{})$.
Concretely, $\upprev[k](f \vert x) = [\uptrans{}^{k-1} f](x)$ for all $x \in \statespace{}$ and all $k \in \nats{}$ \cite[Theorem~3.1]{deCooman:2009jz}.
The map $\uptrans{}$ therefore plays a similar role as the transition matrix $T$ in traditional Markov chains, which is why it is called the \emph{upper transition operator} corresponding to the set $\settrans{}$.
Moreover, observe that (the value of) the upper transition operator $\uptrans$, for any $h\in\setofgambles{}(\statespace{})$, can always be approximated arbitrarily closely by an element of $\settrans$:
\begin{enumerate}[leftmargin=*,ref={\upshape{}S\arabic*},label={\upshape{}S\arabic*}.,itemsep=3pt]
\item\label{sep specified} 
$(\forall\epsilon>0)\,(\forall h\in\setofgambles{}(\statespace{}))\,(\exists T\in\settrans)\ \uptrans{}h - \epsilon \leq Th \leq \uptrans h$.
\end{enumerate}
This property is a direct consequence of the fact that $\settrans$ is separately specified---see for example \cite[Lemma 1]{extended8627473} where we assume the set $\settrans$ to be closed---and will be used later on in Section~\ref{sect: weak ergo in repetition independence}.

In an analogous way, inferences of the form $\upprev[\mathrm{av},k](f \vert x)$ can be obtained as the $k$-th iteration of the map $\smash{\avuptrans{f}{} \colon \setofgambles{}(\statespace{}) \to \setofgambles{}(\statespace{})}$ defined by $\smash{\avuptrans{f}{} h \coloneqq f + \uptrans{} h}$ for all $h \in \setofgambles{}(\statespace{})$.
In particular, if we let $\tilde{m}_{f,1} \coloneqq f = \avuptrans{f}{} (0)$ and 
\begin{align}\label{Eq: recursive expression}
\tilde{m}_{f,k} \coloneqq f + \uptrans{} \tilde{m}_{f,k-1} = \avuptrans{f}{} \tilde{m}_{f,k-1} \text{ for all } k > 1,
\end{align} 
then it follows from \cite[Lemma 41]{8535240} that $\upprev[\mathrm{av},k](f \vert x) = \tfrac{1}{k} \tilde{m}_{f,k}(x)$ for all $x \in \statespace{}$ and all $k \in \nats{}$.
Applying Equation~\eqref{Eq: recursive expression} repeatedly, we find that for all $x \in \statespace{}$ and all $k\in\nats$:
\begin{align}\label{Eq: recursive expression 2}
\upprev[\mathrm{av},k](f \vert x) = 
\tfrac{1}{k}\tilde{m}_{f,k}(x)
=\tfrac{1}{k} [\avuptrans{f}{k-1} \tilde{m}_{f,1}](x) 
= \tfrac{1}{k} [\avuptrans{f}{k}(0)](x).
\end{align}
The same formula can also be obtained as a special case of the results in \cite{10.1007/978-3-030-29765-7_38}.

These expressions for $\upprev[k](f \vert x)$ and $\upprev[\mathrm{av},k](f \vert x)$ in terms of the respective operators $\uptrans{}$ and $\avuptrans{f}{}$ are particularly useful when we aim to characterise the limit behaviour of these inferences.
As will be elaborated on in the next section, there are conditions on $\uptrans{}$ that are necessary and sufficient for $\upprev[k](f \vert x)$ to converge to a limit value that does not depend on the process' initial state $x \in \statespace{}$.
If this is the case for all $f \in \setofgambles{}(\statespace{})$, the imprecise Markov chain (under epistemic irrelevance or complete independence) is called \emph{ergodic} \cite{Hermans:2012ie,DECOOMAN201618} and we then denote the constant limit value by $\upprev[\infty](f) \coloneqq \lim_{k \to +\infty} \upprev[k](f \vert x)$.
By analogy, we call an imprecise Markov chain (under epistemic irrelevance or complete independence) \emph{weakly ergodic}\/ if, for all $f \in \setofgambles{}(\statespace{})$, $\lim_{k \to +\infty} \upprev[\mathrm{av},k](f \vert x)$ exists and does not depend on the initial state $x$.
For a weakly ergodic imprecise Markov chain, we denote the common limit value by $\upprev[\mathrm{av},\infty](f) \coloneqq \lim_{k \to +\infty} \upprev[\mathrm{av},k](f \vert x)$.
In contrast with conventional ergodicity, weak ergodicity and other properties of the long-term behaviour of $\upprev[\mathrm{av},k](f \vert x)$ are almost entirely unexplored.
The main result of the first part of this paper, which focuses on imprecise Markov chains under epistemic irrelevance and complete independence, is a necessary and sufficient condition for weak ergodicity.
As we will see, this condition is weaker than that needed for conventional ergodicity, hence our choice of terminology.
The following example shows that this difference already becomes apparent in the precise case.

\begin{example}\label{example 1}
\begin{spacing}{\mylinespacing}
Recall the situation sketched in Figure~\ref{Figure 1}, where $\statespace{} = \{a,b\}$ and where $\settrans{}$ consists of a single matrix 
$T = 
\big[\begin{smallmatrix}
0 & 1 \\ 1 & 0
\end{smallmatrix}\big]$.
Fix any function $\smash{f = \big[\begin{smallmatrix}
f_a \\ f_b
\end{smallmatrix}\big] \in \setofgambles{}(\statespace{})}$.
Clearly, $\uptrans{}$ is not ergodic because $\smash{\uptrans{}^{(2\ell + 1)} f
= T^{(2\ell + 1)} f
= \big[\begin{smallmatrix}
0 & 1 \\ 1 & 0
\end{smallmatrix}\big] f
= \big[\begin{smallmatrix}
f_b \\ f_a
\end{smallmatrix}\big]}
$ and 
$\smash{\uptrans{}^{(2\ell)} f 
= \big[\begin{smallmatrix}
1 & 0 \\ 0 & 1
\end{smallmatrix}\big]f
= \big[\begin{smallmatrix}
f_a \\ f_b
\end{smallmatrix}\big]}$
for all $\ell \in \natz{}$.
$\uptrans{}$ is weakly ergodic though, because
\begin{align*}
\avuptrans{f}{(2\ell)}(0) = \ell \big[\begin{smallmatrix}
f_a + f_b \\ f_a + f_b
\end{smallmatrix}\big]
 \, \text{ and } \, \avuptrans{f}{(2\ell +1)}(0) = f + \uptrans{} \, \avuptrans{f}{(2\ell)}(0) = f + \ell \big[\begin{smallmatrix}
f_a + f_b \\ f_a + f_b
\end{smallmatrix}\big],
\end{align*} 
for all $\ell \in \natz{}$, which implies that $\upprev[\mathrm{av},\infty](f) \coloneqq \lim_{k \to +\infty} \avuptrans{f}{k}(0)/k = ( f_a + f_b )/2$ exists.
\hfill $\Diamond$
\end{spacing}
\vspace*{-0.06cm}
\end{example}

Notably, even if an imprecise Markov chain is ergodic (and hence also weakly ergodic) and therefore both $\upprev[\infty](f)$ and $\upprev[\mathrm{av},\infty](f)$ exist, these inferences will not necessarily coincide.
This was first observed in an experimental setting \cite[Section~7.6]{8535240}, but the differences that were observed there were marginal.
The following example shows that these differences can in fact be very substantial.

\begin{example}\label{example 2}
\begin{spacing}{\mylinespacing}
Let $\statespace{} = \{a,b\}$, let\/ $\settrans[a]$ be the set of all probability mass functions on\/ $\statespace{}$ and let\/ $\settrans[b]$ be the set that consists of the single probability mass function 
that puts all mass in $a$; see Figure~\ref{Figure 2}.
Then, for any $f = \big[\begin{smallmatrix}
f_a \\ f_b
\end{smallmatrix}\big] \in \setofgambles{}(\statespace{})$, we have that
\begin{align*}
\uptrans{}f (x) = 
\begin{aligned}
\begin{cases}
\max f &\text{ if } x=a; \\
f_a &\text{ if } x=b,
\end{cases}
\end{aligned}
\quad \text{ and } \quad
\uptrans{}^{ \, 2} f (x) = 
\begin{aligned}
\begin{cases}
\max \uptrans{} f = \max f &\text{ if } x=a; \\
\uptrans{}f (a) = \max f &\text{ if } x=b.
\end{cases}
\end{aligned}
\end{align*}
It follows that $\uptrans{}^k f = \max f$ for all $k \geq 2$, so the limit upper expectation $\upprev[\infty](f)$ exists and is equal to $\max f$ for all $f \in \setofgambles{}(\statespace{})$.
In particular, we have that $\upprev[\infty](\indica{b}) = 1$.  
On the other hand, we find that 
$\smash{\avuptrans{\indica{b}}{(2\ell)}(0) = \ell}$ and $\smash{\avuptrans{\indica{b}}{(2\ell + 1)}(0)} = \smash{\indica{b} + \uptrans{} \, \avuptrans{\indica{b}}{(2\ell)}(0)} = \smash{\big[\begin{smallmatrix}
\ell \\ \ell+1
\end{smallmatrix}\big]}$ for all $\ell \in \natz{}$.
This implies that the upper expectation $\smash{\upprev[\mathrm{av},\infty](\indica{b})} \coloneqq \smash{\lim_{k \to +\infty} \avuptrans{\indica{b}}{k}(0)/k}$ exists and is equal to $1/2$.
This value differs significantly from the limit upper expectation $\upprev[\infty](\indica{b}) = 1$.

In fact, this result could have been expected simply by taking a closer look at the dynamics that correspond to $\settrans{}$.
Indeed, it follows directly from $\settrans{}$  that, if the system is in state $b$ at some instant, then it will surely be in $a$ at the next time instant.
Hence, the system can only reside in state $b$ for maximally half of the time, resulting in an upper expected average that converges to $1/2$.
These underlying dynamics have little effect on the limit upper expectation $\upprev[\infty](\indica{b})$ though, because it is only concerned with the upper expectation of\/ $\indica{b}$ evaluated at a single time instant.

Finally, we want to draw attention to the fact that the upper expectation\/ $\smash{\upprev[\mathrm{av},\infty](\indica{b})} = 1/2$ is actually reached by a compatible homogeneous Markov chain.
Specifically, it is reached by the Markov chain from Example~\ref{example 1}, which is indeed compatible because its transition matrix
$\smash{T = 
\big[\begin{smallmatrix}
0 & 1 \\ 1 & 0
\end{smallmatrix}\big]}$ is in $\settrans$.
This already illustrates what will be established later on in Section~\ref{sect: weak ergo in repetition independence}:
when we are interested in the limit behaviour of the inferences $\upprev[\mathrm{av},k](f \vert x)$, we can simply treat them as upper envelopes of the expectations $\smash{\prev{\mathrm{P}}(f_{\mathrm{av}}(X_{1:k}) \vert X_1 = x)}$ that correspond to the compatible homogeneous Markov chains\/ $\mathrm{P}$.
This is not the case for the limit behaviour of the inferences\/ $\upprev[k](f \vert x)$ though; 
for instance, in the current example, where\/ $\upprev[\infty](\indica{b})=1$, the expectation $\prev{\mathrm{P}}(\indica{b}(X_k) \vert X_1 = b)$, for any $\mathrm{P}\in\rimarkov{\settrans}$, is lower or equal than $1/2$ for $k$ even.
This is left as an exercise for the reader---Hint: 
for any $\smash{T = 
\big[\begin{smallmatrix}
p & 1-p \\ 1 & 0
\end{smallmatrix}\big]}\in\settrans$, find $c_1,c_2 \in\reals$ such that $T^2\indica{a} = c_1 + c_2 \indica{a}$, and then use this observation to find an expression for $T^{2\ell}\indica{a}$ and $T^{2\ell+1}\indica{b} = (1-p) T^{2\ell}\indica{a}$ for all $\ell\in\nats$.
\hfill $\Diamond$
\end{spacing}
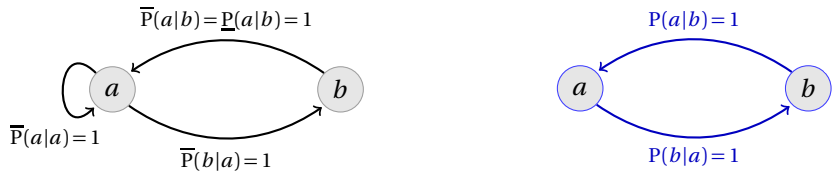
\begin{figure}[ht]
\begin{center}
\vspace*{-0.2cm}
\begin{tikzpicture}[scale=1]
\node[anchor = east] (firstpic) at (0,0) {
\begin{tikzpicture}
\tikzset{pil/.style={->,thick,shorten >=1pt}}
\tikzstyle{place}=[circle,inner sep=2pt,draw=gray!75,fill=gray!20,minimum size=6mm]
\node [place] (a) at (0,0) {$a$};
\node [place] (b) at (3,0) {$b$};
\draw[pil] (a) to [out=135,in=225,loop,looseness=4.8] node[left=1mm,below=3.5mm]{\footnotesize $\overline{\mathrm{P}}(a \vert a) = 1$} (a);
\draw[pil, bend right = 35] (b.north west) to node[above]{\footnotesize $\overline{\mathrm{P}}(a \vert b) = \underline{\mathrm{P}}(a \vert b) = 1$} (a.north east);
\draw[pil, bend right = 35] (a.south east) to node[below]{\footnotesize $\overline{\mathrm{P}}(b \vert a) = 1$} (b.south west);
\end{tikzpicture}
};




\node[anchor = west] (thirdpic) at (2.3,0) {
\begin{tikzpicture}
\tikzset{pil/.style={->,thick,shorten >=1pt,color=blue!75!black}}
\tikzstyle{place}=[circle,inner sep=2pt,draw=blue!75,fill=gray!20,minimum size=6mm]
\node [place] (a) at (0,0) {$a$};
\node [place] (b) at (3,0) {$b$};

\draw[->,thick,shorten >=1pt,color=white] (b) to [out=45,in=-45,loop,looseness=4.8,white] node[right=1mm,below=3.5mm]{\footnotesize $\overline{\mathrm{P}}(a \vert a) = 1$} (b);
\draw[pil, bend right = 35] (b.north west) to node[above]{\footnotesize $\mathrm{P}(a \vert b) = 1$} (a.north east);
\draw[pil, bend right = 35] (a.south east) to node[below]{\footnotesize $\mathrm{P}(b \vert a) = 1$} (b.south west);
\end{tikzpicture}
};
\end{tikzpicture}
\caption{The graph on the left hand side illustrates, in terms of upper and lower transition probabilities, how the state of the process can change from one time instant to the next.
The blue graph on the right hand side depicts the compatible precise Markov chain $\mathrm{P}$ for which the upper bound $\smash{\upprev[\mathrm{av},\infty](\indica{b})} = 1/2$ is reached.
}\label{Figure 2}
\end{center}
\end{figure}
\vspace*{-0.6cm}
\end{example}

Although we have used sets $\settrans{}$ of transition matrices to define imprecise Markov chains, it should at this point be clear that, if we are interested in the inferences $\upprev[k](f \vert x)$ and $\upprev[\mathrm{av},k](f \vert x)$ and their limit values, then it suffices to specify $\uptrans{}$.
In fact, we will temporarily forget about $\settrans{}$ and simply assume that $\uptrans{}$ is a general upper transition operator on $\setofgambles{}(\statespace{})$. 
That is, we assume $\uptrans{}$ to be any operator from $\setofgambles{}(\statespace{})$ to $\setofgambles{}(\statespace{})$ that satisfies
\begin{enumerate}[leftmargin=*,ref={\upshape{}U\arabic*},label={\upshape{}U\arabic*}.,itemsep=3pt, series=transcoherence]
\begin{samepage}
\item\label{transcoherence: upper bound} $\uptrans{} h \leq \max h$ \hfill [upper bounds];
\item\label{transcoherence: subadditivity} $\uptrans{} (h+g) \leq \uptrans{}h + \uptrans{}g$ \hfill [sub-additivity];
\item\label{transcoherence: homogeneity} $\uptrans{}(\lambda h) = \lambda \uptrans{}h$ \hfill [non-negative homogeneity],
\end{samepage}
\end{enumerate}
for all $h,g \in \setofgambles{}(\statespace)$ and all real $\lambda \geq 0$ \cite{DECOOMAN201618,Hermans:2012ie,Krak_ICTMC_2017}.
This can be done without loss of generality because it is well-established that any operator $\smash{\uptrans{}}$ that is defined as an upper envelope of a set $\settrans{}$ of transition matrices---as we did in Section~\ref{section: trans operators and ergodicity}---always satisfies \ref{transcoherence: upper bound}--\ref{transcoherence: homogeneity} \cite[Theorem 2.6.3]{Walley:1991vk}.
Then again, any upper transition operator can uniquely be represented by a closed, convex set of transition matrices that is seperately specified \cite[Theorem~3.3.3]{Walley:1991vk}, so there is no gain in generality either.
Apart from the axioms above, our results and proofs will also rely on the following three properties that are implied by \ref{transcoherence: upper bound}--\ref{transcoherence: homogeneity} \cite[Section~2.6.1]{Walley:1991vk}:
\begin{enumerate}[leftmargin=*,ref={\upshape{}U\arabic*},label={\upshape{}U\arabic*}.,itemsep=3pt, resume=transcoherence]
\item\label{transcoherence: bounds} $\min h  \leq \uptrans{} h \leq \max h$ \hfill [boundedness];
\item\label{transcoherence: constant addivity} $\uptrans{}(\mu + h) = \mu + \uptrans{}h$ \hfill [constant additivity];
\item\label{transcoherence: monotonicity} if $h \leq g$ then $\uptrans{} h \leq \uptrans{} g$ \hfill [monotonicity];
\item\label{transcoherence: mixed additivity} $\uptrans{}h - \uptrans{}g \leq \uptrans{} (h-g)$ \hfill [mixed sub-additivity],
\end{enumerate}
for all $h,g \in \setofgambles{}(\statespace)$ and all real $\mu$.
Henceforth, we will simply say that the upper transition operator $\uptrans{}$ is ergodic if $[\uptrans{}^k f]$ converges to a constant for all $f\in\setofgambles{}(\statespace{})$ and, analogously, we will say that it is weakly ergodic if $\tfrac{1}{k} [\avuptrans{f}{k}(0)]$ converges to a constant for all $f\in\setofgambles{}(\statespace{})$.
So ergodicity and weak ergodicity of $\uptrans{}$ is equivalent to the respective notions for an imprecise Markov chain under epistemic irrelevance or complete independence with upper transition operator $\uptrans{}$.

\section{Accessibility Relations and Topical Maps}\label{Sect: accessibility and topical maps}

To characterise ergodicity and weak ergodicity, we will make use of some well-known graph-theoretic concepts, suitably adapted to the imprecise Markov chain setting; we recall the following from \cite{deCooman:2009jz} and \cite{Hermans:2012ie}.
The \emph{upper accessibility graph} $\mathscr{G}(\uptrans{})$ corresponding to $\uptrans{}$ is defined as the directed graph with vertices $x_1 \cdots x_n \in \statespace{}$, where $n \coloneqq \vert \statespace{} \vert$, with an edge from $x_i$ to $x_j$ if $\uptrans{}\indica{x_j}(x_i) > 0$.
For any two vertices $x_i$ and $x_j$, we say that $x_j$ is \emph{accessible} from $x_i$, denoted by $x_i \to x_j$, if $x_i = x_j$ or if there is a directed path from $x_i$ to $x_j$, which means that there is a sequence $x_i = x'_0, x'_1 , \cdots, x'_k = x_j$ of vertices, with $k\in\nats{}$, such that there is an edge from $x'_{\ell-1}$ to $x'_{\ell}$ for all $\ell \in \{1,\cdots,k\}$.
The following result shows that such a directed path exists if and only if there is some $k\in\nats$ such that the $k$-step upper probability $\uptrans{}^k \indica{x_j}(x_i)$ to transition from $x_i$ to $x_j$ is positive. 
\begin{lemma}\label{lemma: directed path}
\emph{\cite[Proposition~4]{Hermans:2012ie}}
For any two vertices $x$ and $y$, there is a directed path of length $k \in \nats{}$\/ from $x$ to $y$ if and only if\/ $\uptrans{}^k \indica{y}(x) > 0$.
\end{lemma}
We say that two vertices $x_i$ and $x_j$ \emph{communicate} and write $x_i \leftrightarrow x_j$ if both $x_i \to x_j$ and $x_j \to x_i$.
The relation $\rightarrow$ is a preorder (reflexive and transitive), and therefore, $\leftrightarrow$ is an equivalence relation (reflexive, symmetric and transitive) for which the equivalence classes are called \emph{communication classes}.
Moreover, it is well-known that these communication classes then form a partition $\mathscr{C}$ of $\statespace{}$.
Sometimes, we will allow ourselves a slight abuse of terminology, and call any non-empty set $\mathcal{S}\subseteq\statespace{}$ a \emph{class} in $\mathscr{G}(\uptrans)$.
We call the graph $\mathscr{G}(\uptrans{})$ \emph{strongly connected} if any two vertices $x_i$ and $x_j$ in $\mathscr{G}(\uptrans{})$ communicate, or equivalently, if $\statespace{}$ itself is a communication class. 

We also extend the domain of the relation $\to$ to include all communication classes by saying that $A \to B$, for any two $A,B\in\mathscr{C}$, if $x \to y$ for at least one (and hence---since $A$ and $B$ are communication classes---all) $x\in A$ and $y\in B$.
Then it can easily be seen that $\to$ induces a partial order (reflexive, antisymmetric and transitive) on the set $\mathscr{C}$ [Lemma~\ref{lemma: -> induces partial order}, \ref{Sect: app: accessibility relations}].
$\uptrans{}$ (or $\mathscr{G}(\uptrans{})$) is then said to have a \emph{top class} $\mathcal{R}$ if there is a communication class $\mathcal{R}$ that dominates every other communication class in this partial order, or equivalently [Lemma~\ref{lemma: if single communication class then top class}, \ref{Sect: app: accessibility relations}], since $\statespace{}$---and hence also $\mathscr{C}$---is finite, if $\mathcal{R}$ is the only maximal (undominated) communication class of this ordering.
One can also show [Lemma~\ref{lemma: formula top class}, \ref{Sect: app: accessibility relations}] that the top class $\mathcal{R}$ exists if and only if the set of all vertices that can be reached from anywhere in $\mathscr{G}(\uptrans{})$ is non-empty, in which case the top class is equal to this set:
\begin{align}\label{Eq: top class}
\mathscr{G}(\uptrans{})\text{ has a top class }\mathcal{R}\Leftrightarrow
\mathcal{R} = \{ x \in \statespace{} \colon y \to x \text{ for all } y \in \statespace{} \} \not= \emptyset.
\end{align}
Finally, we also say that a class $\mathcal{S}$ is \emph{closed} if $x \not\to y$ for all $x \in \mathcal{S}$ and all $y \in \mathcal{S}^c$.
Since the notions of closedness and maximality coincide for communication classes [Lemma~\ref{lemma: inescapable is maximal for communication classes}, \ref{Sect: app: accessibility relations}], it follows that the top class $\mathcal{R}$, if it exists, is the only closed communication class in $\mathscr{G}(\uptrans{})$.

Having a top class is necessary for $\uptrans{}$ to be ergodic, but it is not sufficient.
Sufficiency additionally requires that the top class $\mathcal{R}$ is regular and absorbing \cite[Proposition~3]{Hermans:2012ie}.
These properties are defined, for any closed class $\mathcal{S}$, as
\begin{enumerate}[leftmargin=*,ref={\upshape{}A\arabic*},label={\upshape{}A\arabic*}.,itemsep=3pt, series=accessibility]
\item\label{property: regularity} $(\forall x \in \mathcal{S})\,(\exists k^\ast \in \nats{})\,(\forall k \geq k^\ast) \ \min \uptrans{}^k \indica{x} > 0$ \hfill [Regularity];
\item\label{property: absorbing} $(\forall x \in \mathcal{S}^c)\,(\exists k \in \nats{}) \  \uptrans{}^k \indica{\mathcal{S}^c}(x) < 1$ \hfill [Absorbing].
\end{enumerate} 
We will say that $\uptrans{}$ is \emph{top class regular} (TCR) if it has a top class that is regular, and analogously for \emph{top class absorbing} (TCA).\footnote{
	The definitions for (TCR) and (TCA) in \cite{Hermans:2012ie} differ slightly from ours; they are equivalent though, as can be seen from Lemmas~\ref{lemma: formula top class} and~\ref{lemma: top class regularity formula}.
}
Top class regularity represents aperiodic behaviour: it demands that there is some time instant $k^\ast \in \nats{}$ such that all of the elements in the top class $\mathcal{R}$ are accessible from each other in $k$ steps, for any $k \geq k^\ast$.
In the case of traditional Markov chains, top class regularity suffices as a necessary and sufficient condition for ergodicity \cite{kemeny1960finite,deCooman:2009jz}.
However, in the imprecise case, we need the additional condition of being top class absorbing, which ensures that the top class will eventually be reached.  
It requires that, if the process starts from any state $x \in \mathcal{R}^c$, the lower probability that it will ever transition to $\mathcal{R}$ is strictly positive; see \cite{deCooman:2009jz} for a more detailed discussion. 
This notion of an absorbing (top) class, however, is not to be confused with what is called an absorbing state in standard literature on Markov chains.
The latter simply is a state such that, once entered, it can never be left anymore \cite[Section~3.1]{kemeny1960finite}; in our terminology, this is the same as a closed class that consists of a single state.

From a practical point of view, an important feature of both of these accessibility conditions is that they can be easily checked in practice, as is shown in~\cite[Section 5]{Hermans:2012ie}.
Strictly speaking, though, the method for checking \ref{property: absorbing} that is presented in \cite[Proposition~6]{Hermans:2012ie} only applies to regular top classes.
However, a closer look at the proof---of \cite[Proposition~6]{Hermans:2012ie}---shows that it does not rely on the regularity of the top class, and that the method can therefore be applied to any top class.
Hence, the condition of (TCA)---the central condition of this paper that will turn out to be necessary and sufficient for weak ergodicity---can be easily verified in practice by first checking the existence of a top class $\mathcal{R}$---for instance, using \eqref{Eq: top class}---and then checking \ref{property: absorbing} using the method described in \cite[Proposition~6]{Hermans:2012ie}.
A more explicit treatment of this subject, however, would lead us too far, and we therefore leave it to this informal argument.

The characterisation of ergodicity using (TCR) and (TCA) was strongly inspired by the observation that upper transition operators are part of a specific collection of order-preserving maps, called \emph{topical maps}.
These are maps $F \colon \reals{}^n \to \reals{}^n$ that satisfy, for all $h,g \in \reals{}^n$ and all $\mu \in \reals{}$,
\begin{enumerate}[leftmargin=*,ref={\upshape{}T\arabic*},label={\upshape{}T\arabic*}.,itemsep=3pt, series=topical]
\item\label{topical: constant addivity} $F(\mu + h) = \mu + Fh$ \hfill [constant additivity];
\item\label{topical: monotonicity} if $h \leq g$ then $F(h) \leq F(g)$ \hfill [monotonicity].
\end{enumerate} 
To show this, we identify $\setofgambles{}(\statespace{})$ with the finite-dimensional linear space $\reals{}^n$, with $n = \vert \statespace{} \vert$; this is clearly possible because both are isomorphic.
That every upper transition operator is topical now follows trivially from \ref{transcoherence: constant addivity} and \ref{transcoherence: monotonicity}.
What is perhaps less obvious, but can be derived in an equally trivial way, is that the operator $\avuptrans{f}{}$ is also topical.
This allows us to apply results for topical maps to $\smash{\avuptrans{f}{}}$ in order to find necessary and sufficient conditions for weak ergodicity.

\section{A Sufficient Condition for Weak Ergodicity}\label{Sect: A Sufficient Condition for Weak Ergodicity}

As a first step, we aim to find sufficient conditions for the existence of $\upprev[\mathrm{av},\infty](f)$.
To that end, recall from Section~\ref{section: trans operators and ergodicity} that $\upprev[\mathrm{av},\infty](f)$ exists if---and only if---the limit $\smash{\lim_{k \to +\infty} \avuptrans{f}{k} (0) / k}$ exists, and in that case $\upprev[\mathrm{av},\infty](f) = \smash{\lim_{k \to +\infty} \avuptrans{f}{k} (0) / k}$.
Then, since $\avuptrans{f}{}$ is topical, the following lemma implies that it is also equal to $\smash{\lim_{k \to +\infty} \avuptrans{f}{k} h / k}$ for any $h \in \setofgambles{}(\statespace{})$.
\begin{lemma}\label{lemma: cycle time exists independently of starting point}\emph{\cite[Lemma 3.1]{GUNAWARDENA2003141}}
Consider any topical map $F \colon \reals{}^n \to \reals{}^n$. 
If the limit\/ $\lim_{k \to +\infty} F^k h / k$ exists for some $h \in \reals{}^n$, then the limit exists for all $h \in \reals{}^n$ and they are all equal.
\end{lemma}
Hence, if $\smash{\lim_{k \to +\infty} \avuptrans{f}{k} h / k}$ converges to a constant vector $\mu$ for some $h \in \setofgambles{}(\statespace{})$, then $\upprev[\mathrm{av},\infty](f)$ exists and is equal to $\mu$.
This condition is clearly satisfied if the map $\smash{\avuptrans{f}{}}$ has an (additive) eigenvector $h \in \setofgambles{}(\statespace{})$, meaning that $\smash{\avuptrans{f}{k} h = h + k \mu}$ for some $\mu \in \reals{}$ and all $k \in \natz{}$.
In that case, we have that $\upprev[\mathrm{av},\infty](f) = \mu$, where $\mu$ is called the eigenvalue corresponding to $h$.
Moreover, there can then only be one such eigenvalue $\mu$. 

\begin{corollary}\label{corollary:if eigenvalue then it is the only eigenvalue}
Consider any topical map $F \colon \reals{}^n \to \reals{}^n$.
If $F$ has an (additive) eigenvalue $\mu$, then it is the only eigenvalue of $F$. 
\end{corollary}
\begin{proof}
Suppose that $F$ has two eigenvalues $\mu_1$ and $\mu_2$, and let $h_1$ and $h_2$ be the corresponding eigenvectors.
Then we have that $F^k h_1 = h_1 + k \mu_1$ for all $k \in \natz{}$, which immediately implies that $\lim_{k \to +\infty} F^k h_1 / k = \mu_1$.
In a similar way, we obtain that $\lim_{k \to +\infty} F^k h_2 / k = \mu_2$.
Then, due to Lemma~\ref{lemma: cycle time exists independently of starting point}, we have that $\mu_1=\mu_2$.
\end{proof}

To find conditions that guarantee the existence of an eigenvector of $\avuptrans{f}{}$, we will make use of results from \cite{10.2307/3845053} and \cite{GUNAWARDENA2003141}.
There, accessibility graphs are defined in a slightly different way:
for any topical map $F \colon \reals{}^n \to \reals{}^n$, they let $\mathscr{G}'(F)$ be the graph with vertices $v_{1} , \cdots , v_{n}$ and an edge from $v_i$ to $v_j$ if $\smash{\lim_{\alpha \to +\infty} [F(\alpha \indica{v_j})](v_i) = +\infty}$.
Subsequently, for such a graph $\mathscr{G}'(F)$, the accessibility relation $\cdot \to \cdot$ and corresponding notions (e.g. `strongly connected', `top class', \dots) are defined as in Section~\ref{Sect: accessibility and topical maps}.
If we identify the vertices $v_{1} , \cdots , v_{n}$ in $\mathscr{G}'(\uptrans{})$ and $\smash{\mathscr{G}'(\avuptrans{f}{})}$ with the different states $x_{1} , \cdots , x_{n}$ in $\statespace{}$, this can in particular be done for the topical maps $\uptrans{}$ and $\avuptrans{f}{}$.
The following results show that the resulting graphs coincide with the one defined in Section~\ref{Sect: accessibility and topical maps}.

\begin{lemma}\label{lemma: edge}
For any two vertices $x$ and $y$ in $\mathscr{G}'(\uptrans{})$, there is an edge from $x$ to $y$ in $\mathscr{G}'(\uptrans{})$ if and only if there is an edge from $x$ to $y$ in $\mathscr{G}(\uptrans{})$. 
\end{lemma}
\begin{proof}
Consider any two vertices $x$ and $y$ in the graph $\mathscr{G}'(\uptrans{})$.
By definition, there is an edge from $x$ to $y$ if $\lim_{\alpha \to +\infty} [\uptrans{}(\alpha \indica{y})](x) = +\infty$. 
Due to \ref{transcoherence: homogeneity}, this is equivalent to the condition that $\lim_{\alpha \to +\infty} \alpha [\uptrans{}\indica{y}](x) = +\infty$.
Since moreover $0 \leq \uptrans{}\indica{y} \leq 1$ by \ref{transcoherence: bounds}, this condition reduces to $\uptrans{}\indica{y}(x) > 0$.
\end{proof}

\begin{corollary}\label{corollary: graphs are identical}
The graphs $\mathscr{G}'(\avuptrans{f}{})$, $\mathscr{G}'(\uptrans{})$ and $\mathscr{G}(\uptrans{})$ are identical.
\end{corollary}
\begin{proof}
Lemma~\ref{lemma: edge} implies that $\mathscr{G}'(\uptrans{})$ and $\mathscr{G}(\uptrans{})$ are identical.
Moreover, that $\mathscr{G}'(\avuptrans{f}{})$ is equal to $\mathscr{G}'(\uptrans{})$, follows straightforwardly from the definition of $\avuptrans{f}{}$.
\end{proof}
 
In principle, we could use this result to directly obtain the desired condition for the existence of an eigenvector from \cite[Theorem~2]{10.2307/3845053}.
However, \cite[Theorem~2]{10.2307/3845053} is given in a multiplicative framework and would need to be reformulated in an additive framework in order to be applicable to the map $\avuptrans{f}{}$; see \cite[Section~2.1]{10.2307/3845053}.
This can be achieved with a bijective transformation, but we prefer to not do so because it would require too much extra terminology and notation.
Instead, we will derive an additive variant of \cite[Theorem~2]{10.2307/3845053} directly from \cite[Theorem~9]{10.2307/3845053} and \cite[Theorem~10]{10.2307/3845053}.

The first result establishes that the existence of an eigenvector is equivalent to the fact that trajectories are bounded with respect to the Hilbert semi-norm $\hnorm{\cdot}$, defined by $\hnorm{h} \coloneqq \max h - \min h$ for all $h \in \reals{}^n$.

\begin{theorem}\label{theorem: eigenvector iff bounded}\emph{\cite[Theorem~9]{10.2307/3845053}}
Let $F \colon \reals{}^n \to \reals{}^n$ be a topical map.
Then $F$ has an eigenvector in $\reals{}^n$ if and only if $\left\{\hnorm{F^k h} \colon k \in \nats{} \right\}$ is bounded for some (and hence all) $h \in \reals{}^n$.
\end{theorem}
That the boundedness of a single trajectory indeed implies the boundedness of all trajectories follows from the non-expansiveness of a topical map with respect to the Hilbert semi-norm \cite{10.2307/3845053}.
The second result that we need uses the notion of a \emph{super-eigenspace}, defined for any topical map $F$ and any $\mu \in \reals{}$ as the set $S^\mu(F) \coloneqq \{h \in \reals{}^n \colon F h \leq h + \mu\}$.

\begin{theorem}\label{theorem: strongly connected then supereigenspace bounded}\emph{\cite[Theorem~10]{10.2307/3845053}}
Let $F \colon \reals{}^n \to \reals{}^n$ be a topical map such that the associated graph $\mathscr{G}'(F)$ is strongly connected.
Then all of the super-eigenspaces are bounded in the Hilbert semi-norm.
\end{theorem}
Together, these theorems imply that any topical map $F \colon \reals{}^n \to \reals{}^n$ for which the graph $\mathscr{G}'(F)$ is strongly connected, has an eigenvector.
The connection between both is provided by the fact that trajectories cannot leave an eigenspace.
The following result formalises this.

\begin{theorem}\label{theorem: strongly connected then eigenvector}
Let $F \colon \reals{}^n \to \reals{}^n$ be a topical map such that the associated graph $\mathscr{G}'(F)$ is strongly connected.
Then $F$ has an eigenvector in $\reals{}^n$.
\end{theorem}
\begin{proof}
Consider any $h\in\reals{}^n$ and any $\mu\in\reals$ such that $\max(Fh - h) \leq \mu$. Then $Fh\leq h+\mu$, so $h\in S^\mu(F)$.
Now notice that $F(Fh) \leq F(h +\mu) = Fh +\mu$ because of \ref{topical: constant addivity} and \ref{topical: monotonicity}, which implies that also $Fh \in S^\mu(F)$.
In the same way, we can also deduce that $F^{2} h \in S^\mu(F)$ and, by repeating this argument, that the whole trajectory corresponding to $h$ remains in $S^\mu(F)$.
This trajectory is bounded because of Theorem~\ref{theorem: strongly connected then supereigenspace bounded}, which by Theorem~\ref{theorem: eigenvector iff bounded} guarantees the existence of an eigenvector. 
\end{proof}
In particular, if $\mathscr{G}'(\avuptrans{f}{})$ is strongly connected then $\avuptrans{f}{}$ has an eigenvector, which on its turn implies the existence of $\smash{\upprev[\mathrm{av},\infty](f)}$ as explained earlier.
If we combine this observation with Corollary~\ref{corollary: graphs are identical}, we obtain the following result.

\begin{proposition}\label{proposition: strongly connected then eigenvector}
$\uptrans{}$ is weakly ergodic if the associated graph $\mathscr{G}(\uptrans{})$ is strongly connected.
In that case, for any $f \in \setofgambles{}(\statespace{})$, the limit value $\upprev[\mathrm{av},\infty](f)$ is equal to the unique (additive) eigenvalue of $\avuptrans{f}{}$.
\end{proposition}
\begin{proof}
Suppose that $\mathscr{G}(\uptrans{})$ is strongly connected.
Then, by Corollary~\ref{corollary: graphs are identical}, $\mathscr{G}'(\avuptrans{f}{})$ is also strongly connected.
Hence, for any $f \in \setofgambles{}(\statespace{})$, since $\smash{\avuptrans{f}{}}$ is a topical map, Theorem~\ref{theorem: strongly connected then eigenvector} guarantees the existence of an eigenvector of $\smash{\avuptrans{f}{}}$.
Let $\mu$ be the corresponding eigenvalue.
By Corollary~\ref{corollary:if eigenvalue then it is the only eigenvalue}, this eigenvalue $\mu$ is the only, and therefore unique, eigenvalue corresponding to $\smash{\avuptrans{f}{}}$.
As explained in the beginning of this section, it now follows from Lemma~\ref{lemma: cycle time exists independently of starting point} that $\smash{\upprev[\mathrm{av},\infty](f)}$ exists and is equal to $\mu$, so we indeed find that $\uptrans{}$ is weakly ergodic.
\end{proof}
In the remainder of this paper, we will use the fact that $\uptrans{}$ is an upper transition operator---so not just any topical map---to strengthen this result.
In particular, we will show that the condition of being strongly connected can be replaced by a weaker one: being top class absorbing.
Nonetheless, the result above can already be useful in practice because checking whether a graph is strongly connected can be done rather efficiently; in any case more efficiently than checking for (TCA) since that requires us to first check the existence of a top class anyway. 
Hence, when interested in weakly ergodic behaviour, and when the dimensions of the considered model are large, one may prefer to verify this before checking the weaker condition of being top class absorbing.\footnote{Note that one might want to do so even in case the considered imprecise Markov chain fails to satisfy (TCR)---and therefore fails to be ergodic.
This is illustrated by Example~\ref{example 1} which depicts a situation where it is immediately clear that (TCR) is not satisfied---since $\mathscr{G}(\uptrans{})$ is cyclic---but where the graph $\mathscr{G}(\uptrans{})$ is nevertheless strongly connected.}

\section{A Necessary and Sufficient Condition for Weak Ergodicity}\label{section: Result}

In order to gain some intuition about how to obtain a more general sufficient condition for weak ergodicity, consider the case where $\uptrans{}$ has a closed (or, equivalently, maximal) communication class $\mathcal{S}$ and the process' initial state $x$ is in $\mathcal{S}$.
Since $\mathcal{S}$ is closed, the process surely remains in $\mathcal{S}$ and hence, it is to be expected that the time average of $f$ will not be affected by the dynamics of the process outside $\mathcal{S}$.
Moreover, the communication class $\mathcal{S}$ is a strongly connected component, so one would expect that, due to Proposition~\ref{proposition: strongly connected then eigenvector}, the upper expected time average $\upprev[\mathrm{av},k](f \vert x)$ converges to a constant that does not depend on the state $x \in \mathcal{S}$.
Our intuition is formalised by the following proposition.
Its proof, as well as those of the other statements in this section (apart from Theorem~\ref{theorem: weakly ergodic iff top class absorbing}), can be found in the Appendix.

\begin{proposition}\label{proposition: if top class then eigenvector in top class}
For any closed communication class $\mathcal{S}$ of \/ $\uptrans{}$, any $f\in\setofgambles{}(\statespace{})$ and any $x\in\mathcal{S}$, the inference $\upprev[\mathrm{av},k](f \vert x)$ is equal to $\upprev[\mathrm{av},k](f \indica{\mathcal{S}} \vert x)$ and converges to a limit value as $k$ recedes to infinity. This limit value is furthermore the same for all $x \in \mathcal{S}$.
\end{proposition}
As a next step, we want to extend the domain of convergence of $\upprev[\mathrm{av},k](f \vert x)$ to all states $x \in \statespace{}$.
To do so, we will impose the additional property of being top class absorbing (TCA), which, as explained in Section~\ref{Sect: accessibility and topical maps}, demands that there is a strictly positive (lower) probability to reach the top class $\mathcal{R}$ in a finite time period.
Once in $\mathcal{R}$, the process can never escape $\mathcal{R}$ though. One would therefore expect that as time progresses---as more of these finite time periods go by---this lower probability increases, implying that the process will eventually be in $\mathcal{R}$ with practical certainty.
Furthermore, if the process transitions from $x \in \mathcal{R}^c$ to a state $y \in \mathcal{R}$, then Proposition~\ref{proposition: if top class then eigenvector in top class} guarantees that $\upprev[\mathrm{av},k](f \vert y)$ converges to a limit and that this limit value does not depend on the state $y$.
Finally, since the average is taken over a growing time interval, the initial finite number of time steps that it took for the process to transition from $x$ to $y$ will not influence the time average of $f$ in the limit. This leads us to suspect that $\upprev[\mathrm{av},k](f \vert x)$ converges to the same limit as $\upprev[\mathrm{av},k](f \vert y)$. Since this argument applies to any $x\in\mathcal{R}^c$, we are led to believe that $\uptrans{}$ is weakly ergodic. The following result confirms this.

\begin{proposition}\label{proposition: if top class absorbing then weakly ergodic}
$\uptrans{}$ is weakly ergodic if it satisfies (TCA). 
\end{proposition}

Conversely, suppose that $\uptrans{}$ does not satisfy (TCA).
Then there are two possibilities: 
either there is no top class or there is a top class but it is not absorbing.
If there is no top class, then it can be easily deduced that there are at least two closed communication classes $\mathcal{S}_1$ and $\mathcal{S}_2$.
As discussed earlier, the process cannot leave the classes $\mathcal{S}_1$ and $\mathcal{S}_2$ once it has reached them.
So if it starts in one of these communication classes, the process' dynamics outside this class are irrelevant for the behaviour of the resulting time average.
In particular, if we let $f$ be the function that takes the constant value $c_1$ in $\mathcal{S}_1$ and $c_2$ in $\mathcal{S}_2$, with $c_1 \not= c_2$, then we would expect that $\upprev[\mathrm{av},k](f \vert x) = c_1$ and $\upprev[\mathrm{av},k](f \vert y) = c_2$ for all $k \in \natz{}$, any $x \in \mathcal{S}_1$ and any $y \in \mathcal{S}_2$.
In fact, this can easily be formalised by means of Proposition~\ref{proposition: if top class then eigenvector in top class}.
Hence, $\upprev[\mathrm{av},\infty](f \vert x)=c_1\neq c_2=\upprev[\mathrm{av},\infty](f \vert y)$, so the upper transition operator $\uptrans{}$ cannot be weakly ergodic.
In other words, if $\uptrans{}$ is weakly ergodic, there must be a top class.

\begin{proposition}\label{prop: top class if weak ergodicity}
$\uptrans{}$ has a top class if it is weakly ergodic.
\end{proposition} 

Finally, suppose that there is a top class $\mathcal{R}$, but that it is not absorbing.
This implies that there is an $x \in \mathcal{R}^c$ and a compatible precise model such that the process is guaranteed to remain in $\mathcal{R}^c$ given that it started in $x$.\footnote{For the sake of this intuitive explanation, we have assumed the set $\settrans$ to be closed. If this is not the case, we can actually not guarantee the existence of such a precise compatible model.
Our proof of Proposition~\ref{prop: TCA if weak ergodicity} uses a more involved---yet less intuitive---argument that does not require $\settrans$ to be closed.}
If we now let $f = \indica{\mathcal{R}^c}$, then conditional on the fact that $X_0 = x$, the expected time average of $f$ corresponding to this precise model is equal to $1$. 
Furthermore, since $f\leq1$, no other process can yield a higher expected time average. 
The upper expected time average $\upprev[\mathrm{av},k](f \vert x)$ is therefore equal to $1$ for all $k \in \nats{}$.
However, using Proposition~\ref{proposition: if top class then eigenvector in top class}, we can also show that $\upprev[\mathrm{av},k](f \vert y) = 0$ for any $y \in \mathcal{R}$ and all $k \in \nats{}$.
Hence, $\upprev[\mathrm{av},\infty](f \vert x)=1\neq 0=\upprev[\mathrm{av},\infty](f \vert y)$, which precludes $\uptrans{}$ from being weakly ergodic.

\begin{proposition}\label{prop: TCA if weak ergodicity}
$\uptrans{}$ satisfies (TCA) if it is weakly ergodic and has a top class.
\end{proposition} 
Together with Propositions~\ref{proposition: if top class absorbing then weakly ergodic} and~\ref{prop: top class if weak ergodicity}, this allows us to conclude that (TCA) is a necessary and sufficient condition for weak ergodicity. 



\begin{theorem}\label{theorem: weakly ergodic iff top class absorbing}
$\uptrans{}$ is weakly ergodic if and only if it satisfies (TCA).
\end{theorem}
\begin{proof}
That (TCA) is a sufficient condition follows from Proposition~\ref{proposition: if top class absorbing then weakly ergodic}.
Necessity follows from Proposition~\ref{prop: top class if weak ergodicity} together with Proposition~\ref{prop: TCA if weak ergodicity}.
\end{proof}

\section{Weak Ergodicity for Imprecise Markov Chains Under Repetition Independence}\label{sect: weak ergo in repetition independence}

So far, we have ignored imprecise Markov chains under repetition independence in our analysis of weakly ergodic behaviour.
Within the field of imprecise probability, these imprecise Markov chains are less studied because (i) they can incorporate fewer types of model uncertainty and (ii) they seem difficult to handle computationally, in the sense that there are (almost) no methods that are able to efficiently solve inference problems for such models; see \cite[Section 5.7]{8535240} for more details.
On the other hand, their relevance should not be underestimated; they model the practical situation where we believe there is a single and fixed transition matrix $T$, but only have partial knowledge about the numerical values that make up this matrix.

The concepts of ergodicity and weak ergodicity can be defined in a similar way as for imprecise Markov chains under epistemic irrelevance and imprecise Markov chains under complete independence.
For any $f\in\setofgambles{}(\statespace{})$, any $x\in\statespace{}$ and any $k\in\nats$, we let 
\begin{align*}
\riupprev{k}(f\vert x) \coloneqq \riupprev{\settrans}(f(X_k) \vert X_1 = x) \quad \text{and} \quad  \avriupprev{k}(f\vert x) \coloneqq \riupprev{\settrans}(f_{\mathrm{av}}(X_{1:k}) \vert X_1 = x).
\end{align*}
Then we say that the imprecise Markov chain $\smash{\rimarkov{\settrans}}$ is ergodic if, for all $f\in\setofgambles{}(\statespace{})$, the upper expectation $\smash{\riupprev{k}(f\vert x)}$ converges to a value $\riupprev{\infty}(f) \coloneqq \lim_{k\to+\infty} \riupprev{k}(f\vert x)$ that does not depend on $x\in\statespace{}$.
Analogously, we call $\smash{\rimarkov{\settrans}}$ weakly ergodic if, for all $f\in\setofgambles{}(\statespace{})$, the upper expectation $\smash{\avriupprev{k}(f\vert x)}$ converges to a value $\smash{\avriupprev{\infty}(f)} \coloneqq \lim_{k\to+\infty} \avriupprev{k}(f\vert x)$ that does not depend on $x\in\statespace{}$.
Unfortunately, however, and unlike what we saw for the inferences $\upprev[k](f\vert x)$ and $\upprev[\mathrm{av},k](f\vert x)$ in imprecise Markov chains under epistemic irrelevance and complete independence, there are no known recursive expressions that describe the behaviour of the inferences $\riupprev{k}(f\vert x)$ and $\avriupprev{k}(f\vert x)$ in terms of a single (topical) map $F\colon\setofgambles{}(\statespace{})\to\setofgambles{}(\statespace{})$.
Hence, in order to study weak ergodicity (and conventional ergodicity) for these models, we will have to rely on a different approach than the one we have used before.
Moreover, because of the absence of any recursive expressions, it is a priori not certain whether the inferences $\avriupprev{k}(f\vert x)$ will only depend on $\settrans$\/ through the upper transition operator $\uptrans{}$.
So, contrary to what we did in Sections~\ref{section: trans operators and ergodicity} to \ref{section: Result}, we cannot simply forget about the set $\settrans$ here.
Instead, we will regard the set $\settrans$---which is always assumed to be separately specified---as the primary object that determines the values of the inferences  $\avriupprev{k}(f\vert x)$.

Before we begin our analysis, recall that the model $\rimarkov{\settrans{}}$ consists of all homogeneous (precise) Markov chains~$\mathrm{P}$ that are compatible with $\settrans{}$, in the sense that $\mathrm{P}(X_{2} \vert X_1 = x) \in \settrans[x]$ for all $x \in \statespace{}$---where we already took into account the homogeneity and the Markov property for each $\mathrm{P}$.
Since each homogeneous Markov chain $\mathrm{P}$ has a transition matrix $T$, we could alternatively write that $\mathrm{P} \in \rimarkov{\settrans{}}$ if and only if $T(x,\cdot) \in \settrans[x]$ for all $x \in \statespace{}$ or, since $\settrans{}$ is separately specified, if and only if $T \in \settrans{}$. 
Furthermore, note that each $\mathrm{P} \in \rimarkov{\settrans{}}$ is itself an imprecise Markov chain, where the upper transition operator is now linear and characterised by the transition matrix $T$ of $\mathrm{P}$.
The independence assumption---epistemic irrelevance, complete independence or repetition independence---obviously does not matter here; they are all equivalent.
Hence, the recursive expressions \eqref{Eq: recursive expression} and \eqref{Eq: recursive expression 2}, and, more generally, all our results for imprecise Markov chains under epistemic irrelevance or complete independence, also hold for a homogeneous Markov chain $\mathrm{P} \in \rimarkov{\settrans{}}$.
The transition matrix $T$ corresponding to $\mathrm{P}$ then takes the role of a particular upper transition operator $\uptrans$.
In the remainder, we will assume that the reader has taken notice of this fact and we will simply apply our results from the previous sections in this more specific case without further ado.
On top of that, we will often use the basic relation that $T^k h \leq \uptrans{}^k h$ for all $T\in\settrans$, all $h\in\setofgambles{}(\statespace{})$ and all $k\in\natz{}$.
This intuitive inequality can be easily derived from the definition of $\uptrans$ and the monotonicity [\ref{transcoherence: monotonicity}] of $\uptrans{}$ and $T$.

For notational convenience, we use $T_f$ and $\smash{\prev{\mathrm{av},k}^{\mathrm{P}}}$ to denote the objects $\smash{\avuptrans{f}{}}$ and $\smash{\upprev[\mathrm{av},k]}$---or, equivalently, $\smash{\avriupprev{k}}$---that correspond to any homogeneous Markov chain $\mathrm{P} \in \rimarkov{\settrans{}}$; so we let $\smash{T_f h \coloneqq f + Th}$, with $T$ the transition matrix of $\mathrm{P}$, and $\smash{\prev{\mathrm{av},k}^{\mathrm{P}}(f \vert x)} \coloneqq \smash{\prev{\mathrm{P}}(f_{\mathrm{av}}(X_{1:k}) \vert X_1 = x)}$ for all $f,h \in \setofgambles{}(\statespace{})$, all $x \in \statespace{}$ and all $k \in \nats{}$.
Taking into account the previous considerations, we can then write that, for all $f \in \setofgambles{}(\statespace{})$, all $x \in \statespace{}$ and all $k \in \nats{}$,
\begin{align}\label{Eq: recursive formula rep. independence}
\avriupprev{k}(f \vert x) 
= \sup_{\mathrm{P} \in \rimarkov{\settrans{}}} \prev{\mathrm{av},k}^{\mathrm{P}}(f \vert x)
= \sup_{T \in \settrans{}} \tfrac{1}{k} [T_f^{k}(0)](x), 
\end{align}
where the last equality holds because of Equation~\eqref{Eq: recursive expression 2} and the fact that $\mathrm{P} \in \rimarkov{\settrans{}}$ if and only if $T\in\settrans{}$ (for the transition matrix $T$ correponding to $\mathrm{P}$).
The expression above will now serve as the main starting point in our further study of $\avriupprev{k}(f\vert x)$.

As a first step, we show that an imprecise Markov chain $\smash{\rimarkov{\settrans}}$ is weakly ergodic if the upper transition operator $\uptrans{}$ corresponding to $\settrans$ has a graph $\mathscr{G}(\uptrans{})$ that is strongly connected.
Moreover, in that case, we also have that $\avriupprev{\infty}(f) = \upprev[\mathrm{av},\infty](f)$ for all $f\in\setofgambles{}(\statespace{})$.
The result can be obtained in a rather straightforward fashion from the following lemma, which essentially states that, if $\avuptrans{f}{}$ has an eigenvalue, there is a homogeneous Markov chain $\mathrm{P}\in\rimarkov{\settrans}$ that behaves `approximately' weakly ergodic.
Its proof makes use of the supremum norm $\supnorm{\cdot}$, defined by $\supnorm{h} \coloneqq \max_{x \in \statespace{}} \vert h(x) \vert$ for all $h \in \setofgambles{}(\statespace{})$.

\begin{lemma}\label{lemma: if T_f has eigenvalue then repetition independence weakly ergodic}
Consider any $f\in\setofgambles{}(\statespace{})$.
If the map $\avuptrans{f}{}$ has an eigenvalue $\mu$, then, for any $\epsilon>0$, there is some $T\in\settrans$ such that
\begin{align*}
\mu-\epsilon \leq \liminf_{k\to+\infty} \tfrac{1}{k} [T_f^{k}(0)] \leq \limsup_{k\to+\infty} \tfrac{1}{k} [T_f^{k}(0)] \leq \mu.
\end{align*}
\end{lemma}
\begin{proof}
Suppose that the map $\avuptrans{f}{}$ has an eigenvalue $\mu$.
Then $\avuptrans{f}{}$ has at least one eigenvector $h\in\setofgambles{}(\statespace{})$ for which it holds that $\avuptrans{f}{k} h = h + k \mu$ for all $k\in\nats$.
Fix any $\epsilon>0$.
Then, since $\settrans$ is assumed to be separately specified, we are allowed to use \ref{sep specified}, which implies that there is some $T\in\settrans$ such that $\avuptrans{f}{} h -\epsilon = f + \uptrans{}h -\epsilon \leq f + Th = T_f h \leq \avuptrans{f}{} h$.
The function $h$ is an eigenvector of $\avuptrans{f}{}$ with eigenvalue $\mu$, so it follows that $h + \mu - \epsilon \leq T_f h \leq h + \mu$.
Since $T_f$ is monotone [\ref{topical: monotonicity}] and satisfies constant additivity [\ref{topical: constant addivity}]---because $T$ is monotone [\ref{transcoherence: monotonicity}] and satisfies constant additivity [\ref{transcoherence: constant addivity}]---this implies that
\begin{align*}
h + 2\mu - 2\epsilon \leq T_f h + \mu - \epsilon = T_f (h + \mu - \epsilon) \leq T_f^2 h \leq T_f(h + \mu) = T_fh + \mu  \leq h + 2\mu.
\end{align*}
Using the same properties of $T_f$, this on its turn implies that
\begin{align*}
h + 3\mu - 3\epsilon 
\leq T_f^2 h + \mu - \epsilon 
= T_f (T_f h + \mu - \epsilon)
&= T_f^2 (h + \mu - \epsilon) \\
&\leq T_f^3 h 
\leq T_f^2(h + \mu) 
= T_f(T_f h + \mu)
= T_f^2 h + \mu  
\leq h + 3\mu.
\end{align*}
Repeating the argument above allows us to conclude that
$h + k\mu - k\epsilon \leq T_f^k h \leq h + k\mu$ for all $k\in\nats$, and therefore, that $\supnorm{h + k\mu - T_f^k h} \leq k\epsilon$ for all $k\in\nats$.
Furthermore, $T_f$ is topical [since, as we already mentioned above, $T_f$ satisfies \ref{topical: constant addivity} and \ref{topical: monotonicity}], so $T_f$ is non-expansive with respect to the supremum norm \cite[Proposition~1.1]{GUNAWARDENA2003141}. 
This implies that $\supnorm{T_f^k h - T_f^k(0)} \leq \supnorm{T_f^{k-1} h - T_f^{k-1}(0)} \leq \cdots \leq \supnorm{h - 0} = \supnorm{h}$ for all $k\in\nats$.
Then, recalling that $\supnorm{h + k\mu - T_f^k h} \leq k\epsilon$ for all $k\in\nats$, we can use the triangle inequality to infer that $\supnorm{h + k\mu - T_f^k(0)} \leq k \epsilon + \supnorm{h}$ or, equivalently, that $\supnorm{\tfrac{1}{k} h + \mu - \tfrac{1}{k}T_f^k(0)} \leq \epsilon + \tfrac{1}{k}\supnorm{h}$ for all $k\in\nats$.
Hence, we have that 
$(\tfrac{1}{k} h + \mu) - (\epsilon + \tfrac{1}{k} \supnorm{h})
\leq \tfrac{1}{k}T_f^k(0)
\text{ for all } k\in\nats$,
which implies that
\begin{align}\label{Eq: lemma: if T_f has eigenvalue then repetition independence weakly ergodic}
\mu - \epsilon
\leq \liminf_{k\to+\infty} \tfrac{1}{k}T_f^k(0).
\end{align}
On the other hand, we have that $\limsup_{k\to+\infty}\tfrac{1}{k}T_f^{k}(0) \leq \limsup_{k\to+\infty}\tfrac{1}{k}\avuptrans{f}{k} (0)$ because $T_f^k(0) \leq \avuptrans{f}{k} (0)$ for all $k\in\nats$.
Furthermore, since $h$ is an eigenvector of $\smash{\avuptrans{f}{}}$ with eigenvalue $\mu$, we have that $\lim_{k\to+\infty}\smash{\tfrac{1}{k}\avuptrans{f}{k} h} = \mu$, which by Lemma~\ref{lemma: cycle time exists independently of starting point} implies that $\lim_{k\to+\infty}\tfrac{1}{k}\avuptrans{f}{k}(0) = \mu$.
Hence, we have that $\limsup_{k\to+\infty}\tfrac{1}{k}T_f^{k}(0) \leq \mu$, which, together with Equation~\eqref{Eq: lemma: if T_f has eigenvalue then repetition independence weakly ergodic} and the fact that obviously $\smash{\liminf_{k\to+\infty}\tfrac{1}{k}T_f^{k}(0) \leq \limsup_{k\to+\infty}\tfrac{1}{k}T_f^{k}(0)}$, implies the desired statement.
\end{proof}

\begin{proposition}\label{prop: rep independence converges to epi irrelevance for strongly connected}
If $\mathscr{G}(\uptrans{})$ is strongly connected, then $\rimarkov{\settrans{}}$ is weakly ergodic and we have that\/ $\smash{\avriupprev{\infty}(f) = \upprev[\mathrm{av},\infty](f)}$ for all $f\in\setofgambles{}(\statespace{})$.
\end{proposition}
\begin{proof}
Fix any $f\in\setofgambles{}(\statespace{})$.
Since $\mathscr{G}(\uptrans{})$ is strongly connected, Proposition~\ref{proposition: strongly connected then eigenvector} implies that $\uptrans{}$ is weakly ergodic and that $\upprev{}_{\mathrm{av},\infty}(f) = \mu$ where $\mu$ is the unique eigenvalue of the map $\uptrans{}_f$.
Then it follows from Lemma~\ref{lemma: if T_f has eigenvalue then repetition independence weakly ergodic} that, for any $\epsilon>0$, there is a $T\in\settrans$ such that
$\mu-\epsilon \leq \liminf_{k\to+\infty} \tfrac{1}{k} [T_f^{k}(0)]$.
So, for any $x\in\statespace{}$, we have that
\begin{align*}
\mu \leq \sup_{T\in\settrans} \liminf_{k\to+\infty} \tfrac{1}{k}[T_f^k(0)](x)
 \leq  \liminf_{k\to+\infty} \sup_{T\in\settrans} \tfrac{1}{k}[T_f^k(0)](x)
  = \liminf_{k\to+\infty} \avriupprev{k}(f\vert x),
\end{align*}
where we used Equation~\eqref{Eq: recursive formula rep. independence} in the last step.
On the other hand, since $\avriupprev{k}(f\vert x) \leq \upprev{}_{\mathrm{av},k}(f\vert x)$ for all $k\in\nats$ (because $\rimarkov{\settrans} \subseteq \cimarkov{\settrans} \subseteq \eimarkov{\settrans}$), we also have that 
\begin{align*}
\liminf_{k\to+\infty} \avriupprev{k}(f\vert x) 
\leq \liminf_{k\to+\infty} \upprev{}_{\mathrm{av},k}(f\vert x)
= \lim_{k\to+\infty} \upprev{}_{\mathrm{av},k}(f\vert x)
= \upprev{}_{\mathrm{av},\infty}(f) = \mu.
\end{align*}
Hence, we conclude that $\lim_{k\to+\infty}\avriupprev{k}(f\vert x)$ exists and is equal to $\mu = \upprev{}_{\mathrm{av},\infty}(f)$ for all $x\in\statespace{}$ and therefore, that $\rimarkov{\settrans}$ is weakly ergodic and that $\avriupprev{\infty}(f) = \upprev{}_{\mathrm{av},\infty}(f)$.
\end{proof}

Proposition~\ref{prop: rep independence converges to epi irrelevance for strongly connected} provides a sufficient condition---having a graph $\mathscr{G}(\uptrans{})$ that is strongly connected---for an imprecise Markov chain $\rimarkov{\settrans}$ to be weakly ergodic.
However, as was the case for the imprecise Markov chains $\eimarkov{\settrans}$ and $\cimarkov{\settrans}$, this condition can actually be weakened to the condition that $\uptrans{}$ should be top class absorbing (TCA).
It will moreover turn out that this weaker accessibility condition is not only sufficient, but also necessary.

Both the sufficiency and necessity of (TCA) can be made intuitive using arguments that are somewhat similar to the ones we have used in Section~\ref{section: Result}, where we discussed the sufficiency and necessity of (TCA) for imprecise Markov chains under epistemic irrelevance and complete independence.
Before we explain why (TCA) is sufficient here as well, first observe that, for any closed class $\mathcal{S}$ in the graph $\mathscr{G}(\uptrans{})$---associated with $\settrans$ through $\uptrans{}$---the class $\mathcal{S}$ is also closed in the graph $\mathscr{G}(T)$, for all $T\in\settrans$:

\begin{samepage}
\begin{lemma}\label{lemma: S is maximal in G(uptrans) then S is maximal in G(T)}
If a class $\mathcal{S}$ is closed in $\mathscr{G}(\uptrans{})$, then, for all $T\in\settrans{}$, the class $\mathcal{S}$ is also closed in $\mathscr{G}(T)$.
\end{lemma}
\begin{proof}
Let $\mathcal{S}$ be any closed class in $\mathscr{G}(\uptrans{})$ and consider any $x\in\mathcal{S}$.
By definition, we have that $x \not\to y$ in $\mathscr{G}(\uptrans{})$ for any $y \in \mathcal{S}^c$, which by Lemma~\ref{lemma: directed path} implies that $\uptrans{}^k \indica{y} (x) \leq 0$ for all $k \in \nats{}$.
Furthermore, for any $T\in\settrans{}$ and all $k\in\nats$, we have that $T^k \indica{y} (x) \leq \uptrans{}^k \indica{y} (x)$, so it follows that $T^k \indica{y} (x) \leq 0$.
Hence, by Lemma~\ref{lemma: directed path}, $x \not\to y$ in $\mathscr{G}(T)$.
This holds for all $x\in\mathcal{S}$ and all $y \in \mathcal{S}^c$, so $\mathcal{S}$ is also closed in $\mathscr{G}(T)$ for any $T\in\settrans{}$.
\end{proof}
\end{samepage}

Now suppose that the upper transition operator $\uptrans{}$ corresponding to $\settrans$ satisfies (TCA) with top class~$\mathcal{R}$, and that the process' initial state $x$ is in $\mathcal{R}$.
Because of the lemma above, $\mathcal{R}$ is also closed in $\mathscr{G}(T)$ for any $T\in\settrans$.
Hence, according to any (homogeneous) Markov chain $\mathrm{P}$ with transition matrix $T$,
the process' state can never leave $\mathcal{R}$.
So it is to be expected that the inference $\smash{\prev{\mathrm{av},k}^{\mathrm{P}}({f} \vert x)}$ will not be influenced by the behaviour of the Markov chain $\mathrm{P}$ outside of $\mathcal{R}$.
Since this is true for any $T\in\settrans$, and since $\smash{\avriupprev{k}(f \vert x)}$ is simply an upper envelope of $\prev{\mathrm{av},k}^{\mathrm{P}}({f} \vert x)$ over all compatible Markov chains $\mathrm{P}$, it seems that, in our study of the inference $\smash{\avriupprev{k}(f \vert x)}$, we can limit ourselves to that part of $\settrans$ and $\rimarkov{\settrans}$ that describes the process' dynamics for states in $\mathcal{R}$. 
Moreover, it can be shown (see Lemma~\ref{lemma: settrans_S is set of transition matrices} in \ref{Sect: appendix: repetition independence}) that the upper transition operator $\uptrans{}'$ and the graph $\mathscr{G}(\uptrans{}')$ associated with such a restricted version $\settrans'$ of $\settrans$, are itself versions of the original $\uptrans$ and $\mathscr{G}(\uptrans)$ restricted to the states in $\mathcal{R}$.
Since $\mathcal{R}$ is a strongly connected component in $\mathscr{G}(\uptrans{})$, the graph $\mathscr{G}(\uptrans{}')$ will therefore be strongly connected as a whole and we are led to believe, due to Proposition~\ref{prop: rep independence converges to epi irrelevance for strongly connected}, that the inference $\avriupprev{k}(f \vert x)$ will converge to a constant that does not depend on the state $x\in\mathcal{R}$.
So, one would expect that, if we limit ourselves to initial states in $\mathcal{R}$, the condition of (TCA) is sufficient for weakly ergodic behaviour.

In order to see why this can also be expected if we allow the process to start in $\mathcal{R}^c$, first consider the following straightforward lemma. 
\begin{lemma}\label{lemma: TCA then all T have absorbing maximal class}
If \/ $\uptrans{}$ satisfies (TCA) with top class~$\mathcal{R}$, then, for all $T\in\settrans{}$, the class $\mathcal{R}$ is an absorbing closed class in $\mathscr{G}(T)$.
\end{lemma}
\begin{proof}
Since $\mathcal{R}$ is the top class in $\mathscr{G}(\uptrans{})$, it is closed in $\mathscr{G}(\uptrans{})$ [Lemma~\ref{lemma: if single communication class then top class}, \ref{Sect: app: accessibility relations}] and therefore, by Lemma~\ref{lemma: S is maximal in G(uptrans) then S is maximal in G(T)}, it is also closed in $\mathscr{G}(T)$ for all $T\in\settrans$.
So it remains to prove that $\mathcal{R}$ is moreover absorbing in $\mathscr{G}(T)$ for all $T\in\settrans$.
To do so, recall that $\mathcal{R}$ is an absorbing (and closed) class in $\mathscr{G}(\uptrans{})$, meaning that $\uptrans{}^{k_x} \indica{\mathcal{R}^c} (x) < 1$ for any $x\in\mathcal{R}^c$ and some $k_x\in\nats{}$.
Then, for any $T\in\settrans$, since $T^{k_x} \indica{\mathcal{R}^c} \leq \uptrans{}^{k_x} \indica{\mathcal{R}^c}$, we also have that $T^{k_x} \indica{\mathcal{R}^c} (x) < 1$ which implies that $\mathcal{R}$ is indeed an absorbing closed class in $\mathscr{G}(T)$.
\end{proof}

So if we consider any Markov chain $\mathrm{P}\in\rimarkov{\settrans}$ and let $T\in\settrans$ be the corresponding transition matrix, the class $\mathcal{R}$ is closed and absorbing in $\mathscr{G}(T)$.
We can then reason in a similar way as in the paragraph that preceded Proposition~\ref{proposition: if top class absorbing then weakly ergodic}; 
due to the fact that $\mathcal{R}$ is absorbing and closed, it is to be expected that, according to the model $\mathrm{P}$, the process is eventually in $\mathcal{R}$ with practical certainty.
Unlike before, however, $\mathcal{R}$ is not necessarily a communication class in $\mathscr{G}(T)$, so we cannot assert that $\mathrm{P}$ behaves in a weakly ergodic way for initial states in $\mathcal{R}$.
In other words, for any $x\in\mathcal{R}$, the expected time average $\prev{\mathrm{av},k}^\mathrm{P}(f \vert x)$ may not converge or, if it does converge, it may possibly depend on the initial state $x\in\mathcal{R}$.
Nonetheless, it can be inferred from our previous discussion and Lemma~\ref{lemma: if T_f has eigenvalue then repetition independence weakly ergodic} that the model $\mathrm{P}\in\rimarkov{\settrans}$ can be chosen in such a way that its behaviour for initial states $x\in\mathcal{R}$ is `approximately' weakly ergodic, in the sense that $\prev{\mathrm{av},k}^\mathrm{P}(f \vert x)$ will eventually lie at $\epsilon$-distance from a constant $\mu$ that does not depend on $x\in\mathcal{R}$.
Recalling that $\mathcal{R}$ is closed and absorbing---for all $\mathrm{P}\in\rimarkov{\settrans}$ and therefore, also for any specific choice of $\mathrm{P}$---the same can be expected for initial states $y\in\mathcal{R}^c$.
As a result, we would then get that $\mu -\epsilon \leq \avriupprev{k}(f \vert x)$ for any $x\in\statespace{}$ and for $k$ large enough.
Moreover, it can easily be shown that the constant $\mu$ is the limit value of the inference $\upprev[\mathrm{av},k](f \vert x)$ (see, for instance, Proposition~\ref{proposition: strongly connected then eigenvector}), so $\mu$ is also an upper bound for the limit values of $\avriupprev{k}(f \vert x)$.
Both observations taken together, one would expect that the inference $\avriupprev{k}(f \vert x)$ converges to the constant $\mu$ for all $x\in\statespace{}$. 
Our next result confirms this.

\begin{proposition}\label{Prop: TCA is sufficient for weak ergodicity of IMC under rep. ind.}
If\/ $\uptrans{}$ satisfies (TCA), then $\rimarkov{\settrans{}}$ is weakly ergodic and\/ $\smash{\avriupprev{\infty}(f) = \upprev[\mathrm{av},\infty](f)}$ for all $f\in\setofgambles{}(\statespace{})$.
\end{proposition}

Conversely, consider any $\settrans$ such that the corresponding $\uptrans{}$ does not satisfy (TCA).
Then either $\mathscr{G}(\uptrans{})$ has no top class, in which case it must necessarily have two (or more) different closed communication classes $\mathcal{S}_1$ and $\mathcal{S}_2$, or $\mathscr{G}(\uptrans{})$ has a top class but it is not absorbing.
Consider the first case.
Since $\mathcal{S}_1$ and $\mathcal{S}_2$ are both closed in $\mathscr{G}(\uptrans{})$, they are also both closed (but not necessarily a communication class) in $\mathscr{G}(T)$ for any $T\in\settrans$, because of Lemma~\ref{lemma: S is maximal in G(uptrans) then S is maximal in G(T)}.
So, according to any Markov chain $\mathrm{P}\in\rimarkov{\settrans}$ with transition matrix $T$, the process surely remains in $\mathcal{S}_i$ with $i\in\{1,2\}$ once it has reached $\mathcal{S}_i$.
Hence, if we let $c_1,c_2 \in\reals$ such that $c_1 \not= c_2$ and assume that $f(x) = c_i$ for all $x\in\mathcal{S}_i$---which is possible because $\mathcal{S}_1$ and $\mathcal{S}_2$ are two different communication classes in $\mathscr{G}(\uptrans{})$ and are therefore disjoint---we would expect that $\prev{\mathrm{av},k}^\mathrm{P}(f \vert x)$ is simply equal to $c_i$ for all $k \in \nats{}$ and all $x\in\mathcal{S}_i$.
Since this is the case for all $\mathrm{P}\in\smash{\rimarkov{\settrans}}$, we are inclined to conclude that $\smash{\avriupprev{k}(f \vert x)} = c_i$ for all $k \in \nats{}$ and all $x\in\mathcal{S}_i$, and therefore that $\lim_{k\to+\infty} \smash{\avriupprev{k}(f \vert x_1)} = c_1 \not= c_2 = \lim_{k\to+\infty} \smash{\avriupprev{k}(f \vert x_2)}$ for any $x_1\in\mathcal{S}_1$ and any $x_2\in\mathcal{S}_2$.
This would preclude $\rimarkov{\settrans}$ from being weakly ergodic.
As a result, if we assume that $\rimarkov{\settrans}$ is weakly ergodic, then $\mathscr{G}(\uptrans{})$ must have a top class.

\begin{proposition}\label{Prop: if no top class then IMC under ri is not weakly ergodic}
If \/ $\rimarkov{\settrans}$ is weakly ergodic, then the graph $\mathscr{G}(\uptrans{})$ has a top class.
\end{proposition}

Finally, consider the case that $\mathscr{G}(\uptrans{})$ has a top class $\mathcal{R}$ that is not absorbing and recall the discussion that lead to Proposition~\ref{prop: TCA if weak ergodicity}.
There, we relied on the fact that, since $\mathcal{R}$ is not absorbing, there is some precise model $\mathrm{P}\in\eimarkov{\settrans}$ for which the process is guaranteed to remain in $\mathcal{R}^c$ given that it started in some state $x\in\mathcal{R}^c$.\footnote{Once more, this is only valid under the assumption that $\settrans$ is closed.
The formal proof of Proposition~\ref{Prop: if not absorbing then IMC under ri is not weakly ergodic}, however, does not rely on $\settrans$ being closed.}
As a matter of fact, this compatible model $\mathrm{P}$ can always be chosen in such a way that it is a homogeneous Markov chain, and therefore such that $\mathrm{P}\in\rimarkov{\settrans}$.
Hence, if we again let $f$ be the indicator $\indica{\mathcal{R}^c}$, then we would obtain that $\smash{\prev{\mathrm{av},k}^{\mathrm{P}}(f \vert x)} = 1$.
Since no other $\mathrm{P}'\in\smash{\rimarkov{\settrans}}$ can then yield a higher expected time average for this $f$, we would also have that $\smash{\avriupprev{k}(f \vert x)} = 1$ for all $k\in\nats$.
On the other hand, since $\mathcal{R}$ is closed and $f(y) = 0$ for any $y\in\mathcal{R}$, it is to be expected---due to the same reasons as we have come to explain in the paragraph above, where both $\mathcal{S}_1$ and $\mathcal{S}_2$ were closed---that $\avriupprev{k}(f \vert y) = 0$ for all $k \in \nats{}$.
So in conclusion, we would have that $\lim_{k\to+\infty}\smash{\avriupprev{k}(f \vert x)}=1\neq 0=\lim_{k\to+\infty}\smash{\avriupprev{k}(f \vert y)}$, again precluding $\rimarkov{\settrans}$ from being weakly ergodic.
\vspace*{2pt}

\begin{proposition}\label{Prop: if not absorbing then IMC under ri is not weakly ergodic}
If \/ $\rimarkov{\settrans}$ is weakly ergodic and the graph $\mathscr{G}(\uptrans{})$ has a top class, then $\uptrans{}$ satisfies (TCA).
\end{proposition}

It now only remains to combine Propositions~\ref{Prop: TCA is sufficient for weak ergodicity of IMC under rep. ind.}, \ref{Prop: if no top class then IMC under ri is not weakly ergodic} and~\ref{Prop: if not absorbing then IMC under ri is not weakly ergodic} to establish our second main result.
\begin{theorem}\label{theorem: IMC under rep indep. weakly ergodic iff TCA}
An imprecise Markov chain under repetition independence $\rimarkov{\settrans{}}$ is weakly ergodic if and only if the upper transition operator $\uptrans{}$ corresponding with the separately specified set $\settrans{}$ satisfies (TCA).
Furthermore, in that case, we have that\/ $\smash{\avriupprev{\infty}(f) = \upprev[\mathrm{av},\infty](f)}$ for all $f\in\setofgambles{}(\statespace{})$. 
\end{theorem}
\begin{proof}
Sufficiency follows from Proposition~\ref{Prop: TCA is sufficient for weak ergodicity of IMC under rep. ind.}; necessity follows from Proposition~\ref{Prop: if no top class then IMC under ri is not weakly ergodic} and Proposition~\ref{Prop: if not absorbing then IMC under ri is not weakly ergodic}.
The last statement---that $\smash{\avriupprev{\infty}(f) = \upprev[\mathrm{av},\infty](f)}$ for all $f\in\setofgambles{}(\statespace{})$ in the case that $\uptrans{}$ satisfies (TCA)---also follows from Proposition~\ref{Prop: TCA is sufficient for weak ergodicity of IMC under rep. ind.}.
\end{proof}

\section{Conclusion}

The most important conclusion of our study of upper and lower expected time averages is the following (see Theorems~\ref{theorem: weakly ergodic iff top class absorbing} and~\ref{theorem: IMC under rep indep. weakly ergodic iff TCA}):
The condition of being top class absorbing is necessary and sufficient for weakly ergodic behaviour of an imprecise Markov chain, irrespectively of the imposed independence assumption.
In that case, upper (and lower) expected time averages converge to limit values that are constant, not only for all possible initial states (or distributions) of the process, but also for all possible types of independence assumptions.

Now, if we compare our notion of weak ergodicity with that of conventional ergodicity---which guarantees the existence of a limit upper and lower expectation---we believe that weak ergodicity, and the associated (limits of) upper and lower expected time averages, should become the new objects of interest when it comes to characterising long-term average behaviour.
Our conviction is based on the following three arguments:
\begin{enumerate}[leftmargin=*,ref={(\roman*)},label={\roman*.},itemsep=0pt]

\item
Weak ergodicity requires less stringent conditions to be satisfied than conventional ergodicity, which additionally requires top class regularity. 
We illustrated this difference in Example~\ref{example 1}, where we considered a(n imprecise) Markov chain that satisfies (TCA) but not (TCR).

\item
The inferences $\upprev[\mathrm{av},\infty](f)$ are able to provide us with more information about how time averages might behave, compared to limit expectations $\upprev[\infty](f)$.
To see why, recall Example~\ref{example 2}, where the inference $\upprev[\mathrm{av},\infty](\indica{b}) = 1/2$ significantly differed from $\upprev[\infty](\indica{b}) = 1$.
Clearly, the former was more representative for the limit behaviour of the time average of $\indica{b}$. 
As a consequence of \cite[Lemma~57]{8535240}, a similar statement holds for general functions.
In particular, it implies that $\upprev[\mathrm{av},\infty](f) \leq \upprev[\infty](f)$ for any function $f \in \setofgambles{}(\statespace{})$.
Since both inferences are upper bounds, $\upprev[\mathrm{av},\infty](f)$ is therefore at least as (and sometimes much more) informative as $\upprev[\infty](f)$.

\item
The characterisation of weak ergodicity, as well as the limit values $\upprev[\mathrm{av},\infty](f)$---or, equivalently, $\avriupprev{\infty}(f)$---of upper (and lower) expected time averages, do not depend on the type of independence assumption that we impose on the imprecise Markov chain.
As we have illustrated in Example~\ref{example 2}, conventional ergodicity does not exhibit this kind of robustness.
We perceive this as an advantage in favor of weak ergodicity.
On the one hand, it provides a clear practical benefit because one should not spend time and/or money in searching for the appropriate independence assumption; it simply does not matter.
On the other hand, it also opens doors for further theoretical research on this topic, because the limit values $\upprev[\mathrm{av},\infty](f)$ are now approximated by two different objects: the upper expectations $\upprev[\mathrm{av},k](f \vert x)$ and the upper expectations $\avriupprev{k}(f \vert x)$.
It is, for instance, not unreasonable to think that this feature might very well be a crucial step in developing efficient algorithms for the computation of $\upprev[\mathrm{av},\infty](f)$.

\end{enumerate}

That said, there is also one important feature that limit upper and lower expectations have, but that we did not consider yet for upper and lower expected time averages: an (imprecise) point-wise ergodic theorem~\cite[Theorem~32]{DECOOMAN201618}. For the limit upper and lower expectations of an ergodic imprecise Markov chain, this result states that
\begin{align*}
\smash{\lowprev{\infty}(f)} 
\leq \liminf_{k \to +\infty} f_{\mathrm{av}} (X_{1:k})
\leq \limsup_{k \to +\infty} f_{\mathrm{av}} (X_{1:k})
\leq \smash{\upprev[\infty](f)},
\end{align*} 
with lower probability one. 
In order for limit upper and lower expected time averages to be the undisputed quantities of interest when studying long-term time averages, a similar result would need to be obtained for weak ergodicity, where the role of $\upprev[\infty](f)$ and $\lowprev{\infty}(f) \coloneqq - \upprev[\infty](- f)$ is taken over by $\upprev[\mathrm{av},\infty](f)$ and $\smash{\lowprev{\mathrm{av},\infty}(f)} \coloneqq \smash{- \upprev[\mathrm{av},\infty](- f)}$, respectively.
If such a result would hold, it would provide us with (strictly almost sure) bounds on the limit values attained by time averages that are not only more informative than the current ones, but also guaranteed to exist under weaker conditions and equal for all types of independence assumptions.
In fact, we are happy to report that we already established such a result.
However, the proof necessitates a mathematical framework that strongly differs from the one used here, and we therefore intend to present it in future work.

Another topic that we would like to consider in the future, is the convergence of the inferences $\upprev[\mathrm{av},k](f \vert x)$ and $\avriupprev{k}(f\vert x)$ in general, without imposing that their limit values should be constant for all states $x$ in $\statespace{}$. 
We suspect that this kind of convergence will require no conditions at all.

\section*{Acknowledgments}
The research of Jasper De Bock was partially funded by the European Commission's H2020 programme, through the UTOPIAE Marie Curie Innovative Training Network, H2020-MSCA-ITN-2016, Grant Agreement number 722734, as well as by FWO (Research Foundation - Flanders),
through project number 3GO28919, entitled ``Efficient inference in large-scale queueing models using imprecise continuous-time Markov chains''.

\bibliographystyle{plain}
\bibliography{references-weak-ergodicity}


\appendix

\section{Accessibility relations}\label{Sect: app: accessibility relations}


\begin{lemma}\label{lemma: -> induces partial order}
The relation $\rightarrow$ induces a partial order on the partition $\mathscr{C}$ of all communication classes.
\end{lemma}
\begin{proof}
That $\rightarrow$ is reflexive and transitive on $\mathscr{C}$ follows immediately from the reflexivity and transitivity of the relation $\to$ on the singletons.
To see that it is also antisymmetric, consider any two sets $A$ and $B$ in $\mathscr{C}$.
Then if $A \to B$ and $B \to A$, it should be clear that any two vertices in $A \cup B$ communicate and therefore that $A \cup B$ is a communication class.
Since $A$ and $B$ are two sets in the partition $\mathscr{C}$ of communication classes, $A \cup B$ can only be a communication class as well if $A=B$.
Hence, the relation $\to$ induces a partial order on the partition $\mathscr{C}$ of all communication classes.
\end{proof}

The following lemma establishes our claim about the equivalence---for communication classes---between the notions of maximality and closedness.

\begin{lemma}\label{lemma: inescapable is maximal for communication classes}
Consider any communication class $\mathcal{S}\in\mathscr{C}$.
Then $\mathcal{S}$ is closed if and only if\/ $\mathcal{S}$ is maximal.
\end{lemma}
\begin{proof}
Suppose that $\mathcal{S}$ is closed; so $x\not\to y$ for all $x\in\mathcal{S}$ and all $y\in\mathcal{S}^c$.
Consider any $\mathcal{S}'\in\mathscr{C}$ such that $\mathcal{S}'\not=\mathcal{S}$, and note that $\mathcal{S}\cap\mathcal{S}' = \emptyset$ because $\mathscr{C}$ is a partition.
Then, since $\mathcal{S}$ is closed, there is no $x\in\mathcal{S}$ and $y\in\mathcal{S}'$ such that $x\to y$, which implies that $\mathcal{S}'$ is not accessible from (or does not dominate) $\mathcal{S}$.
Because this is the case for all $\mathcal{S}'\in\mathscr{C}$ such that $\mathcal{S}'\not=\mathcal{S}$, we conclude that $\mathcal{S}$ is maximal.
Conversely, suppose that $\mathcal{S}$ is maximal, and consider any $x\in\mathcal{S}$ and any $y\in\mathcal{S}^c$.
Let $\mathcal{S}'\in\mathscr{C}$ be such that $y\in\mathcal{S}'$; there is exactly one such $\mathcal{S}'$ because $\mathscr{C}$ forms a partition of $\statespace{}$.
Moreover, $\mathcal{S}'\not=\mathcal{S}$ because $y\in\mathcal{S}'\cap\mathcal{S}^c$.
Then, since $\mathcal{S}$ is maximal, we have that $\mathcal{S}\not\to\mathcal{S}'$ and therefore, in particular, $x\not\to y$.
This is true for all $x\in\mathcal{S}$ and all $y\in\mathcal{S}^c$, so we conclude that $\mathcal{S}$ is closed.
\end{proof}

The following result is well-known in order theory. However, since we could not immediately find an appropriate reference, we have chosen to provide a proof of our own.
\begin{lemma}\label{lemma: if single communication class then top class}
A communication class $\mathcal{S}$ is the top class if and only if\/ $\mathcal{S}$ is the only maximal---or, equivalently, closed---communication class.
\end{lemma}
\begin{proof}
First note that maximality and closedness can indeed be used interchangeably here; see~Lemma~\ref{lemma: inescapable is maximal for communication classes}.
To see that the direct implication holds, suppose that $\mathcal{S}$ is the top class.
Then observe that $\mathcal{S}$ is maximal because otherwise there would be some $C\in\mathscr{C}\setminus\{\mathcal{S}\}$ such that $\mathcal{S} \to C$ and (since $\mathcal{S}$ is the top class) $C \to \mathcal{S}$, and therefore, by the antisymmetry of $\to$, that $\mathcal{S}=C$, contradicting that $C\in\mathscr{C}\setminus\{\mathcal{S}\}$.
It is also the only maximal communication class because each $C\in\mathscr{C}$ is dominated by $\mathcal{S}$.

Now suppose that $\mathcal{S}$ is the only maximal communication class.
To prove the converse implication, we will rely on the following observation: for any finite sequence $C_1,\cdots,C_n$ of different communication classes such that $C_i \to C_j$ and $C_i\not=\mathcal{S}$ for all $i,j\in\{1,\cdots,n\}$ such that $i\leq j$, 
there is a $C_{n+1}\in\mathscr{C}$ such that $C_i \not= C_{n+1}$ and $C_i \to C_{n+1}$ for all $i\in\{1,\cdots,n\}$.
Indeed, $C_n \not=\mathcal{S}$, which implies that $C_n$ is not maximal---$\mathcal{S}$ is the only maximal communication class---and therefore that there is some $C_{n+1}\in\mathscr{C}$ such that $C_n \to C_{n+1}$ (and obviously $C_n \not= C_{n+1}$).
Since $C_i \to C_n$ for any $i\in\{1,\cdots,n\}$, the transitivity of $\to$ implies that also $C_i \to C_{n+1}$.
Moreover, $C_{n+1}$ differs from any $C_i$ with $i\in\{1,\cdots,n-1\}$, because we would otherwise have that $C_i\to C_n$ and $C_n \to C_{n+1} = C_i$, and therefore, by the antisymmetry of $\to$, that $C_i = C_n$.
This would contradict our assumptions about $C_1,\cdots,C_n$ and since we already established that $C_n \not= C_{n+1}$, we indeed conclude that $C_i \not= C_{n+1}$ and $C_i \to C_{n+1}$ for all $i\in\{1,\cdots,n\}$.

Now fix any $C_1$ such that $C_1 \not= \mathcal{S}$.
Then we can use the rule above to show that $C_1 \to \mathcal{S}$ and therefore---since $C_1\not=\mathcal{S}$ is arbitrary---that $\mathcal{S}$ is the top class.
Indeed, since $C_1 \not= \mathcal{S}$, it follows from this rule that there is a $C_2\in\mathscr{C}$ such that $C_1\not=C_2$ and $C_1\to C_2$.
If also $C_2 \not= \mathcal{S}$, there is a third $C_3\in\mathscr{C}$ such that $C_i \not= C_3$ and $C_i \to C_3$ for all $i\in\{1,2\}$.
If also $C_3 \not= \mathcal{S}$, then there is a fourth $C_4\in\mathscr{C}$ such that $C_i \not= C_4$ and $C_i \to C_4$ for all $i\in\{1,\cdots,3\}$, and so on, always continuing to extend this sequence in the same way.
Then, since $\statespace{}$---and therefore also $\mathscr{C}$---is finite, and since $\mathcal{S}\in\mathscr{C}$, we will eventually find that, at some point, the next element $C_n$ of this sequence is such that $C_n=\mathcal{S}$.
Then, due to the fact that $C_i \to C_n$ for all $i\in\{1,\cdots,n-1\}$, we have in particular that $C_n=\mathcal{S}$ is accessible from $C_1$.
\end{proof}

\begin{corollary}\label{corollary: no top class}
If \/ $\uptrans{}$ does not have a top class, then it has at least two (disjoint) closed communication classes.
\end{corollary}
\begin{proof}
Once more, let $\mathscr{C}$ be the partition of all communication classes in $\statespace{}$.
Since we know from Lemma~\ref{lemma: -> induces partial order} that $\to$ induces a partial order relation on $\mathscr{C}$ and because $\mathscr{C}$ is finite, there is at least one maximal communication class \cite[Theorem~3]{birkhoff1940lattice}.
The case that there is exactly one is impossible, because by Lemma~\ref{lemma: if single communication class then top class} that would mean that there is a top class.
Hence, there are at least two maximal---or, by Lemma~\ref{lemma: inescapable is maximal for communication classes}, closed---communication classes in $\statespace{}$.
Furthermore, both communication classes are necessarily disjoint because $\mathcal{C}$ is a partition of $\statespace{}$.
\end{proof}

\begin{lemma}\label{lemma: formula top class}
$\uptrans{}$ has a top class $\mathcal{R}$ if and only if \/ $\mathcal{R}' \coloneqq \{x \in \statespace{} \colon y \to x \text{ for all } y \in \statespace{}\} \not= \emptyset$ and, in that case, $\mathcal{R} = \mathcal{R}'$.
\end{lemma}
\begin{proof}
Suppose that $\uptrans{}$ has a top class $\mathcal{R}$.
Then it follows from the fact that $\mathcal{R}$ dominates (or, in other words, is accessible from) each other communication class $\mathcal{S}\in\mathscr{C}$, that any $x\in\statespace{}$ is in $\mathcal{R}$ if and only if $y\to x$ for all $y \in \statespace{}$. Indeed, on the one hand, if $x$ is in $\mathcal{R}$, then for any $y\in\statespace{}$, if we let $\mathcal{Y}$ be the unique communication class that contains $y$, we find that $y\to x$ because $\mathcal{Y}\to\mathcal{R}$. And on the other hand, if $y\to x$ for all $y\in\statespace{}$, then for any communication class $\mathcal{Y}$, if we let $y$ be any element of $\mathcal{Y}$, we find that $\mathcal{Y}\to\mathcal{R}$ because $y\to x$.
Hence, since $\mathcal{R}$ is implicitly assumed to be non-empty if it exists, we have that $\mathcal{R}' = \mathcal{R} \not= \emptyset$.

Conversely, suppose that $\mathcal{R}' = \{x \in \statespace{} \colon y \to x \text{ for all } y \in \statespace{}\} \not= \emptyset$.
Then first observe that $\mathcal{R}'$ is a communication class. Indeed, on the one hand, for any $x,z\in\mathcal{R}'$, the definition of $\mathcal{R}'$ trivially implies that $x\leftrightarrow z$. And on the other hand, for any $x\in\mathcal{R}'$ and $z\in\statespace{}$ such that $x\leftrightarrow z$, we know that $z\in\mathcal{R}'$ because, for all $y\in\statespace{}$, we have that $y\to x\to z$ and hence $y\to z$. So $\mathcal{R}'$ is a communication class. Since for any other communication class $\mathcal{S}\in\mathscr{C}$, we have that $y\to x$ for any $x\in\mathcal{R}'$ and any $y\in\mathcal{S}$, it follows that $\mathcal{S}\to\mathcal{R}'$ for all $\mathcal{S}\in\mathscr{C}$.
Hence, $\mathcal{R}'$ is the (non-empty) top class $\mathcal{R}$.
\end{proof}

\begin{lemma}\label{lemma: top class regularity formula}
$\uptrans{}$ satisfies (TCR) if and only if \/ $\mathcal{R}' \coloneqq \{x \in \statespace{} \colon (\exists k^\ast \in \nats{})\,(\forall k \geq k^\ast) \ \min \uptrans{}^k \indica{x} > 0\} \not= \emptyset$\/ and, in that case, $\mathcal{R}'$ is the top class $\mathcal{R}$ of \/ $\uptrans{}$.
\end{lemma}
\begin{proof}
Suppose that $\uptrans{}$ satisfies (TCR) and let $\mathcal{R}\not=\emptyset$ be the corresponding top class.
Then, since $\mathcal{R}$ is regular [\ref{property: regularity}], we clearly have that $\mathcal{R}\subseteq\mathcal{R}'$.
Moreover, since $\min\uptrans{}^{k_x}\indica{x} > 0$ for any $x\in\mathcal{R}'$ and some $k_x\in\nats$, it follows from Lemma~\ref{lemma: directed path} that any $x\in\mathcal{R}'$ is accessible from anywhere in $\mathscr{G}(\uptrans{})$.
Hence, due to Lemma~\ref{lemma: formula top class}, we also have that $\mathcal{R}' \subseteq \mathcal{R} = \{x \in \statespace{} \colon y \to x \text{ for all } y \in \statespace{}\}$.
As a conclusion, we have that $\mathcal{R}' = \mathcal{R} \not=\emptyset$.

Conversely, suppose that $\mathcal{R}' = \{x \in \statespace{} \colon (\exists k^\ast \in \nats{})\,(\forall k \geq k^\ast) \ \min \uptrans{}^k \indica{x} > 0\} \not= \emptyset$.
Again, since $\min\uptrans{}^{k_x}\indica{x} > 0$ for any $x\in\mathcal{R}'$ and some $k_x\in\nats$, Lemma~\ref{lemma: directed path} implies that any $x\in\mathcal{R}'$ is accessible from anywhere in $\mathscr{G}(\uptrans{})$.
So, we have that $\mathcal{R}' \subseteq \{x \in \statespace{} \colon y \to x \text{ for all } y \in \statespace{}\}$, which by Lemma~\ref{lemma: formula top class} and the fact that $\mathcal{R}'\not=\emptyset$, and therefore $\{x\in\statespace{}\colon y \to x \text{ for all } y \in \statespace{}\}\not=\emptyset$, implies that the top class $\mathcal{R}$ exists, that it is non-empty, and that $\mathcal{R}' \subseteq \mathcal{R}$.
To show that also $\mathcal{R} \subseteq \mathcal{R}'$, consider any $x\in\mathcal{R}$ and any $y\in\mathcal{R}'$.
Due to Lemma~\ref{lemma: formula top class}, $x$ is accessible from anywhere in $\mathscr{G}(\uptrans)$ and hence definitely from $y$, so there is a directed path from $y$ to $x$.
Let $k\in\nats{}$ be the length of this path.
Furthermore, since $y\in\mathcal{R}'$, there is some $k^\ast\in\nats$ such that $\min \uptrans{}^{k'}\indica{y}>0$ for all $k'\geq k^\ast$, and therefore that $\uptrans{}^{k'}\indica{y}(z) > 0$ for all $k'\geq k^\ast$ and all $z\in\statespace{}$.
Fix any such $k'\geq k^\ast$ and any such $z\in\statespace{}$.
Then, according to Lemma~\ref{lemma: directed path}, there is a directed path of length $k'$ from $z$ to $y$.
Hence, recalling that there is a directed path of length $k$ from $y$ to $x$, we infer that there is a directed path of length $k+k'$ from $z$ to $x$, and therefore, again by Lemma~\ref{lemma: directed path}, that $\uptrans{}^{k+k'}\indica{x}(z) > 0$.
Since this holds for any $k'\geq k^\ast$ and any $z\in\statespace{}$, we have that $\min \uptrans{}^{k+k'}\indica{x}>0$ for all $k'\geq k^\ast$, or equivalently, that $\min \uptrans{}^{\ell}\indica{x}>0$ for all $\ell\geq k+ k^\ast$.
As a result, $x$ is an element of $\mathcal{R}'$. Since $x\in\mathcal{R}$ was chosen arbitrarily, it follows that $\mathcal{R}\subseteq\mathcal{R}'$.
\end{proof}

\section{Proof of Proposition~\ref{proposition: if top class then eigenvector in top class}}

In the following, we will often use the fact that, since $\uptrans{}$ is an upper transition operator, the iterates of $\uptrans{}$ will also be upper transition operators.
This can easily be derived using the properties \ref{transcoherence: upper bound}--\ref{transcoherence: mixed additivity} and an induction argument in $k$.
For an illustration of how to do so, we refer to \cite[Lemma~23]{extended8627473}.

\begin{lemma}\label{lemma: T^k is coherent}
If \/ $\uptrans{}$ is an upper transition operator then, for any $k \in \nats{}$, $\uptrans{}^k$ is an upper transition operator as well. 
\end{lemma}
\noindent
The properties \ref{transcoherence: subadditivity}--\ref{transcoherence: mixed additivity} therefore also apply to $\uptrans{}^k$:
\begin{enumerate}[leftmargin=*,ref={\upshape{}U\arabic*$^\prime$},label={\upshape{}U\arabic*$^\prime$}.,itemsep=3pt, series=iteratedcoherence, start=2]
\item\label{iteratedcoherence: subadditivity} $\uptrans{}^k (h+g) \leq \uptrans{}^k h + \uptrans{}^k g$ \hfill [sub-additivity];
\item\label{iteratedcoherence: homogeneity} $\uptrans{}^k (\lambda h) = \lambda \uptrans{}^k h$ \hfill [non-negative homogeneity];
\item\label{iteratedcoherence: bounds} $\min h  \leq \uptrans{}^k h \leq \max h$ \hfill [boundedness];
\item\label{iteratedcoherence: constant addivity} $\uptrans{}^k (\mu + h) = \mu + \uptrans{}^k h$ \hfill [constant additivity];
\item\label{iteratedcoherence: monotonicity} if $h \leq g$ then $\uptrans{}^k h \leq \uptrans{}^k g$ \hfill [monotonicity];
\item\label{iteratedcoherence: mixed additivity} $\uptrans{}^k h - \uptrans{}^k g \leq \uptrans{}^k  (h-g)$ \hfill [mixed sub-additivity],
\end{enumerate}
for all $k \in \natz{}$, all $h,g \in \setofgambles{}(\statespace)$, all real $\mu$ and all real $\lambda \geq 0$.

Many of the results in this appendix will make use of the graph-theoretic concepts and notations that were defined in Section~\ref{Sect: accessibility and topical maps}.
Unless mentioned otherwise, we will always implicitly assume that they correspond to the graph $\mathscr{G}(\uptrans{})$ of $\uptrans{}$.
Note however that, due to Corollary~\ref{corollary: graphs are identical}, we could also equivalently consider the graphs $\mathscr{G}'(\uptrans{})$ or $\mathscr{G}'(\avuptrans{f}{})$.

\begin{lemma}\label{lemma: not from S to Sc}
For any $\uptrans{}$ with a closed class $\mathcal{S}$, we have that $\uptrans{}^k \indica{\mathcal{S}^c} (x) = 0$ for all $x \in \mathcal{S}$ and all $k \in \nats{}$. 
\end{lemma}
\begin{proof}
Consider any $x \in \mathcal{S}$.
Then, since $\mathcal{S}$ is closed, we have that $x \not\to y$ for any $y \in \mathcal{S}^c$, which by Lemma~\ref{lemma: directed path} implies that $\uptrans{}^k \indica{y} (x) \leq 0$ for all $k \in \nats{}$.
Hence,
\begin{align*}
0 \leq \uptrans{}^k \indica{\mathcal{S}^c} (x) 
 = \Big[ \uptrans{}^k \Big( \sum\nolimits_{y \in \mathcal{S}^c} \indica{y} \Big) \Big] (x) \leq \sum\nolimits_{y \in \mathcal{S}^c} \uptrans{}^k \indica{y} (x) \leq 0 \text{ for all } k \in \nats{},
\end{align*}
where the first step uses \ref{iteratedcoherence: bounds} and the third uses \ref{iteratedcoherence: subadditivity}.
\end{proof}

\begin{lemma}\label{lemma: time average in maximal class only depends on value of f in maximal class}
For any $f\in\setofgambles{}(\statespace{})$ and any $\uptrans{}$ with a closed class $\mathcal{S}$, we have that $\avuptrans{f}{}  h(x)  =  \avuptrans{f}{}  ( h \indica{\mathcal{S}} )(x)$ for all $h \in \setofgambles{}(\statespace{})$ and all $x \in \mathcal{S}$.
\end{lemma}
\begin{proof}
Fix any $h \in \setofgambles{}(\statespace{})$ and any $x \in \mathcal{S}$.
By sub-additivity [\ref{transcoherence: subadditivity}], we have that $\uptrans{}h(x) \leq \uptrans{}(h \indica{\mathcal{S}})(x) + \uptrans{}(h \indica{\mathcal{S}^c})(x)$.
Since $h \indica{\mathcal{S}^c} \leq \supnorm{h} \indica{\mathcal{S}^c}$, monotonicity [\ref{transcoherence: monotonicity}] therefore implies that  
\begin{align*}
\uptrans{}h(x) \leq \uptrans{}(h \indica{\mathcal{S}})(x) + \uptrans{}(\supnorm{h} \indica{\mathcal{S}^c})(x) 
&= \uptrans{}(h \indica{\mathcal{S}})(x) + \supnorm{h} \uptrans{}\indica{\mathcal{S}^c}(x) \\
&= \uptrans{}(h \indica{\mathcal{S}})(x),
\end{align*}
where the first equality follows from non-negative homogeneity [\ref{transcoherence: homogeneity}] and the second from Lemma~\ref{lemma: not from S to Sc}.
Hence, we obtain that $\avuptrans{f}{} h (x) \leq \avuptrans{f}{} (h \indica{\mathcal{S}}) (x)$.
To prove the converse inequality, observe that 
\begin{align*}
\uptrans{}h(x) \geq \uptrans{}(h \indica{\mathcal{S}})(x) - \uptrans{}( - h \indica{\mathcal{S}^c})(x)
&\geq \uptrans{}(h \indica{\mathcal{S}})(x) - \uptrans{}(\supnorm{h} \indica{\mathcal{S}^c})(x) \\
&= \uptrans{}(h \indica{\mathcal{S}})(x) - \supnorm{h} \uptrans{}\indica{\mathcal{S}^c}(x) \\
&= \uptrans{}(h \indica{\mathcal{S}})(x),
\end{align*}
where the first step follows from \ref{transcoherence: mixed additivity}, the second follows from $- h \indica{\mathcal{S}^c} \leq \supnorm{h} \indica{\mathcal{S}^c}$ and monotonicity [\ref{transcoherence: monotonicity}], the third follows from non-negative homogeneity [\ref{transcoherence: homogeneity}] and the last from Lemma~\ref{lemma: not from S to Sc}.
So, we have that $\uptrans{} h (x) \geq \uptrans{} (h \indica{\mathcal{S}}) (x)$ and therefore also that $\avuptrans{f}{} h (x) \geq \avuptrans{f}{} (h \indica{\mathcal{S}}) (x)$.
Hence, $\avuptrans{f}{} h (x) = \avuptrans{f}{} (h \indica{\mathcal{S}}) (x)$ for all $h \in \setofgambles{}(\statespace{})$ and all $x \in \mathcal{S}$.
\end{proof}


To prove Proposition~\ref{proposition: if top class then eigenvector in top class}, we will use the following notations that allow us to confine the dynamics of the process to a closed class.
Let $\uptrans{}$ be any upper transition operator and let $\mathcal{S}$ be any non-empty subset of $\statespace{}$.
For any $h \in \setofgambles{}(\statespace{})$, let $h \vert_{\mathcal{S}} \in \setofgambles{}(\mathcal{S})$ denote the restriction of $h$ to the domain $\mathcal{S}$.
Additionally, for any $h \in \setofgambles{}(\mathcal{S})$, we let $h^\uparrow \in \setofgambles{}(\statespace{})$ denote the zero-extension of $h$ into $\setofgambles{}(\statespace{})$, which takes the value $h(x)$ for $x \in \mathcal{S}$ and $0$ elsewhere.
Then note that $( h \vert_\mathcal{S} )^\uparrow = h \indica{\mathcal{S}}$ for any $h \in \setofgambles{}(\statespace{})$ and $( g^\uparrow )\vert_\mathcal{S} = g$ for any $g \in \setofgambles{}(\mathcal{S})$.
Let $\avuptrans{f,\mathcal{S}}{} \colon \setofgambles{}(\mathcal{S}) \to \setofgambles{}(\mathcal{S})$ be defined by $\avuptrans{f,\mathcal{S}}{} h \coloneqq (\avuptrans{f}{} h^\uparrow)\vert_{\mathcal{S}}$ for all $h \in \setofgambles{}(\mathcal{S})$.

\begin{lemma}\label{lemma: G^k is equal to G_R^k}
For any $f\in\setofgambles{}(\statespace{})$ and any $\uptrans{}$ with a closed class $\mathcal{S}$, we have that $(\avuptrans{f}{k} h)\vert_{\mathcal{S}} = \avuptrans{f,\mathcal{S}}{k} (h\vert_{\mathcal{S}})$ for all $h \in \setofgambles{}(\statespace{})$ and all $k \in \nats{}$.
\end{lemma}
\begin{proof}
We use an induction argument in $k \in \nats{}$.
That the statement holds for $k = 1$ follows immediately from Lemma~\ref{lemma: time average in maximal class only depends on value of f in maximal class}.
Indeed, for any $h \in \setofgambles{}(\statespace{})$, Lemma~\ref{lemma: time average in maximal class only depends on value of f in maximal class} says that $\avuptrans{f}{}h (x) = \avuptrans{f}{}(h \indica{\mathcal{S}}) (x)$ for all $x \in \mathcal{S}$ or, equivalently, that $(\avuptrans{f}{}h)\vert_{\mathcal{S}} = \big( \avuptrans{f}{}(h \indica{\mathcal{S}}) \big)\vert_{\mathcal{S}}$.
This implies, by the definition of $\avuptrans{f,\mathcal{S}}{}$ and the fact that $h \indica{\mathcal{S}} = (h \vert_{\mathcal{S}})^\uparrow$, that $(\avuptrans{f}{} h)\vert_{\mathcal{S}} = \avuptrans{f,\mathcal{S}}{} (h\vert_{\mathcal{S}})$ for any $h \in \setofgambles{}(\statespace{})$, which provides an induction base.

Now assume that the statement holds for all $i \in \{1,\cdots,k\}$, with $k \in \nats{}$.
Then, for any $h \in \setofgambles{}(\statespace{})$, we have that
\begin{align*}
(\avuptrans{f}{k+1} h)\vert_{\mathcal{S}} 
= \big( \avuptrans{f}{} (\avuptrans{f}{k} h) \big)\vert_{\mathcal{S}}
= \avuptrans{f,\mathcal{S}}{} \big((\avuptrans{f}{k} h)\vert_\mathcal{S} \big)
= \avuptrans{f,\mathcal{S}}{} \big( \avuptrans{f,\mathcal{S}}{k} (h\vert_\mathcal{S}) \big)
= \avuptrans{f,\mathcal{S}}{k+1} (h\vert_\mathcal{S}),
\end{align*}
where the second equality follows from the fact that the statement holds for $i=1$ and the third equality follows from the assumption that the statement holds for~$i=k$.
Combined with the induction base, this concludes the proof.
\end{proof}

\begin{lemma}\label{lemma: G_R is topical}
For any $f\in\setofgambles{}(\statespace{})$ and any $\uptrans{}$ with a closed class $\mathcal{S}$, the map $\avuptrans{f,\mathcal{S}}{}$ is topical.
\end{lemma}
\begin{proof}
To prove \ref{topical: constant addivity}, consider any $\mu \in \reals{}$ and any $h \in \setofgambles{}(\mathcal{S})$.
Since $\avuptrans{f}{}$ satisfies \ref{topical: constant addivity}, we have that $\avuptrans{f}{}(\mu + h^\uparrow) = \mu + \avuptrans{f}{}(h^\uparrow)$ and therefore, also that $\big( \avuptrans{f}{}(\mu + h^\uparrow) \big)\vert_\mathcal{S} = \mu + (\avuptrans{f}{} h^\uparrow )\vert_\mathcal{S} = \mu + \avuptrans{f,\mathcal{S}}{} h$.
Moreover, by Lemma~\ref{lemma: G^k is equal to G_R^k}, we have that $\big( \avuptrans{f}{}(\mu + h^\uparrow) \big)\vert_\mathcal{S} = \avuptrans{f,\mathcal{S}}{}\big((\mu + h^\uparrow)\vert_\mathcal{S} \big) = \avuptrans{f,\mathcal{S}}{}(\mu + h)$, implying that \ref{topical: constant addivity} holds.
Finally, that monotonicity [\ref{topical: monotonicity}] holds for $\avuptrans{f,\mathcal{S}}{}$ follows directly from its definition and the fact $\avuptrans{f}{}$ is monotone.
\end{proof}


\begin{lemma}\label{lemma: T_f,S has an eigenvector}
For any $f\in\setofgambles{}(\statespace{})$ and any $\uptrans{}$ with a closed communication class $\mathcal{S}$, the map $\avuptrans{f,\mathcal{S}}{}$ has exactly one (additive) eigenvalue.
\end{lemma}
\begin{proof}
Consider any two states $x$ and $y$ in $\mathcal{S}$.
Then, by the definition of $\mathscr{G}'(\avuptrans{f,\mathcal{S}}{})$, there is an edge from $x$ to $y$ in $\mathscr{G}'(\avuptrans{f,\mathcal{S}}{})$ if $\lim_{\alpha \to +\infty} \avuptrans{f,\mathcal{S}}{}(\alpha \indica{y}) (x) = +\infty$.
Moreover, by the definition of $\avuptrans{f,\mathcal{S}}{}$, we have that $\avuptrans{f,\mathcal{S}}{}(\alpha \indica{y}) = \big(\avuptrans{f}{}(\alpha \indica{y})^\uparrow\big)\vert_\mathcal{S} = \big(\avuptrans{f}{}(\alpha \indica{y})\big)\vert_\mathcal{S}$ for all $\alpha \in \reals{}$, where we used $\indica{y}$ to denote the indicator of $y$ in both $\setofgambles{}(\mathcal{S})$ and $\setofgambles{}(\statespace{})$ depending on the domain of the considered map. 
Hence, there is an edge from $x$ to $y$ in the graph $\mathscr{G}'(\avuptrans{f,\mathcal{S}}{})$ if and only if $\lim_{\alpha \to +\infty} \avuptrans{f}{}(\alpha \indica{y}) (x) = +\infty$ or, equivalently, if and only if there is an edge from $x$ to $y$ in the graph $\mathscr{G}'(\avuptrans{f}{})$.
So $\mathscr{G}'(\avuptrans{f,\mathcal{S}}{})$ is identical to the restriction of the graph $\mathscr{G}'(\avuptrans{f}{})$ to the vertices in $\mathcal{S}$.
Now, $x$ and $y$ are two states in the closed communication class $\mathcal{S}$ of $\mathscr{G}(\uptrans{})$, so we have that $x \to y$ in $\mathscr{G}(\uptrans{})$.
Moreover, the directed path from $x$ to $y$ remains within the closed class $\mathcal{S}$, because $x \not\to z$ for any $x \in \mathcal{S}$ and any $z \in \mathcal{S}^c$.
Then, since $\mathscr{G}(\uptrans{})$ is identical to $\mathscr{G}'(\avuptrans{f}{})$ because of Corollary~\ref{corollary: graphs are identical}, and since $\mathscr{G}'(\avuptrans{f,\mathcal{S}}{})$ is the restriction of $\mathscr{G}'(\avuptrans{f}{})$ to $\mathcal{S}$, we find that $x \to y$ in $\mathscr{G}'(\avuptrans{f,\mathcal{S}}{})$.
Since this holds for any two vertices in $\mathscr{G}'(\avuptrans{f,\mathcal{S}}{})$, it follows that $\mathscr{G}'(\avuptrans{f,\mathcal{S}}{})$ is strongly connected.
Because $\avuptrans{f,\mathcal{S}}{}$ is topical by Lemma~\ref{lemma: G_R is topical}, Theorem~\ref{theorem: strongly connected then eigenvector} then guarantees the existence of an (additive) eigenvector $h \in \setofgambles{}(\mathcal{S})$.
Let $\mu\in\reals$ be the corresponding eigenvalue.
By Corollary~\ref{corollary:if eigenvalue then it is the only eigenvalue}, this is the only eigenvalue of~$\avuptrans{f,\mathcal{S}}{}$.  
\end{proof}

\begin{lemma}\label{lemma: if top class then eigenvector in top class}
For any $\uptrans{}$ with a closed communication class $\mathcal{S}$, we have that 
\begin{align*}
\lim_{k\to+\infty}\tfrac{1}{k}[\avuptrans{f}{k}(0)](x) = \mu,
\end{align*}
for all $f\in\setofgambles{}(\statespace{})$ and all $x\in\mathcal{S}$, where $\mu$ is the unique eigenvalue of the map $\avuptrans{f,\mathcal{S}}{}$.
\end{lemma}
\begin{proof}
Consider any $\uptrans{}$ with a closed communication class $\mathcal{S}$ and fix any $f\in\setofgambles{}(\statespace{})$.
Lemma~\ref{lemma: T_f,S has an eigenvector} guarantees that $\avuptrans{f,\mathcal{S}}{}$ has a unique eigenvalue $\mu \in \reals$. 
Let $h \in \setofgambles{}(\mathcal{S})$ be a corresponding (not necessarily unique) eigenvector.
Then, in a similar way as argued below Lemma~\ref{lemma: cycle time exists independently of starting point}, we deduce that 
$\lim_{k \to +\infty} \avuptrans{f,\mathcal{S}}{k} (h) / k 
= \mu$.
Since $\avuptrans{f,\mathcal{S}}{}$ is topical due to Lemma~\ref{lemma: G_R is topical}, Lemma~\ref{lemma: cycle time exists independently of starting point} then also implies that $\smash{\lim_{k \to +\infty} \avuptrans{f,\mathcal{S}}{k} (0 \vert_{\mathcal{S}}) / k = \mu}$, with $0$ the zero vector in $\setofgambles{}(\statespace{})$.
Moreover, for all $x \in \mathcal{S}$ and all $k \in \natz{}$, due to Lemma~\ref{lemma: G^k is equal to G_R^k}, we have that
$\tfrac{1}{k} \big[ \, \avuptrans{f}{k} (0) \big](x) 
= \tfrac{1}{k} \big[ \, \avuptrans{f,\mathcal{S}}{k} (0\vert_{\mathcal{S}}) \big] (x)$.
This allows us to conclude that $\smash{\lim_{k \to +\infty} \tfrac{1}{k} \big[ \, \avuptrans{f}{k} (0) \big](x) = \mu}$ for all $x \in \mathcal{S}$, where $\mu$ is the unique eigenvalue of the map $\avuptrans{f,\mathcal{S}}{}$.
\end{proof}

\begin{lemma}\label{lemma: average does not depend on f outside S}
For any $\uptrans{}$ with a closed class $\mathcal{S}$, we have that\/ $[\avuptrans{f}{k}(0)] \, \indica{\mathcal{S}} = [\avuptrans{g}{k}(0)] \, \indica{\mathcal{S}}$ for any two $f,g \in \setofgambles{}(\statespace{})$ such that $g \indica{\mathcal{S}} = f \indica{\mathcal{S}}$ and all $k \in \nats{}$.
\end{lemma}
\begin{proof}
Let $\mathcal{S}$ be a closed class of $\uptrans{}$.
Fix any two $f,g \in \setofgambles{}(\statespace{})$ such that $g \indica{\mathcal{S}} = f \indica{\mathcal{S}}$ and let $\smash{\avuptrans{f}{}(\cdot)} \coloneqq f + \uptrans{}(\cdot)$ and $\smash{\avuptrans{g}{}(\cdot)} \coloneqq g + \uptrans{}(\cdot)$ as before.
To prove the statement, we will use an induction argument in $k \in \nats{}$.
That the statement holds for $k = 1$ is trivial because, due to \ref{transcoherence: bounds}, we have that $\smash{[\avuptrans{f}{}(0)]} = f$ and $\smash{[\avuptrans{g}{}(0)]} = g$.
Now suppose that the statement holds for $k = i$ with $i \in\nats$.
Then, by assumption and by Equation~\eqref{Eq: recursive expression 2}, we have that $\tilde{m}_{f,i-1} \indica{\mathcal{S}}
= [\avuptrans{f}{i}(0)] \, \indica{\mathcal{S}} 
= [\avuptrans{g}{i}(0)] \, \indica{\mathcal{S}}
 = \tilde{m}_{g,i-1} \indica{\mathcal{S}}$.
This allows us to write that, for any $x \in \mathcal{S}$, 
\begin{align}\label{Eq: lemma: average does not depend on f outside S}
\tilde{m}_{f,i}(x) 
= \avuptrans{f}{} \tilde{m}_{f,i-1} (x)
= \avuptrans{f}{} (\tilde{m}_{f,i-1} \indica{\mathcal{S}}) (x)
&= \avuptrans{f}{} (\tilde{m}_{g,i-1} \indica{\mathcal{S}}) (x) \nonumber \\
&= \avuptrans{f}{} \tilde{m}_{g,i-1}(x), 
\end{align}
where the second and last step follow from Lemma~\ref{lemma: time average in maximal class only depends on value of f in maximal class}.
Moreover, note that, for any $x \in \mathcal{S}$,
\begin{align*}
\avuptrans{f}{} \tilde{m}_{g,i-1}(x)
= (f +\uptrans{} \tilde{m}_{g,i-1})(x)
= f(x) +\uptrans{} \tilde{m}_{g,i-1}(x)
&= g(x) +\uptrans{} \tilde{m}_{g,i-1}(x) \\
&= \avuptrans{g}{} \tilde{m}_{g,i-1}(x) \\
&= \tilde{m}_{g,i}(x)
\end{align*}
where the third step follows from $f \indica{\mathcal{S}} = g \indica{\mathcal{S}}$.
Hence, recalling Equation~\eqref{Eq: lemma: average does not depend on f outside S}, we have that $\tilde{m}_{f,i}(x) = \tilde{m}_{g,i}(x)$ for all $x \in \mathcal{S}$, which, due to Equation~\eqref{Eq: recursive expression 2}, implies that 
$[\avuptrans{f}{i+1}(0)] \, \indica{\mathcal{S}} 
= \tilde{m}_{f,i} \indica{\mathcal{S}} = \tilde{m}_{g,i} \indica{\mathcal{S}}
= [\avuptrans{g}{i+1}(0)] \, \indica{\mathcal{S}}$ and therefore that the statement holds for $k=i+1$.
\end{proof}

\begin{proofof}{Proposition~\ref{proposition: if top class then eigenvector in top class}.}
Consider any closed communication class $\mathcal{S}$ of $\uptrans{}$, any $f\in\setofgambles{}(\statespace{})$ and any $x \in \mathcal{S}$.
It follows from Lemma~\ref{lemma: average does not depend on f outside S} and Equation~\eqref{Eq: recursive expression 2} that, for any $k \in \nats{}$, 
\begin{align*}
\upprev[\mathrm{av},k](f \vert x) 
= \tfrac{1}{k} [\avuptrans{f}{k}(0)] (x) 
= \tfrac{1}{k} [\avuptrans{f \indica{\mathcal{S}}}{k}(0)]  (x) 
= \upprev[\mathrm{av},k](f \indica{\mathcal{S}} \vert x).
\end{align*}
Moreover, by Lemma~\ref{lemma: if top class then eigenvector in top class}, we have that $\lim_{k\to+\infty} \tfrac{1}{k} [\avuptrans{f}{k}(0)](x) = \mu$ for some $\mu\in\reals{}$ and all $x\in\mathcal{S}$, which implies that the upper expectation $\upprev[\mathrm{av},k](f \vert x)$ converges to a constant $\mu$ that does not depend on the specific state $x \in \mathcal{S}$.
\end{proofof}

\section{Proof of Proposition~\ref{proposition: if top class absorbing then weakly ergodic}}

\begin{lemma}\label{lemma: Rc decreases}
For any $\uptrans{}$ with a closed class $\mathcal{S}$, the function $\uptrans{}^k \indica{\mathcal{S}^c}$ is non-increasing in $k \in \nats{}$. 
\end{lemma}
\begin{proof}
From Lemma~\ref{lemma: not from S to Sc}, it follows that $\uptrans{} \indica{\mathcal{S}^c} (x) = 0$ for all $x \in \mathcal{S}$.
Since moreover $0 \leq \uptrans{} \indica{\mathcal{S}^c} \leq 1$ by \ref{transcoherence: bounds}, it follows that $\indica{\mathcal{S}^c}\geq \uptrans{} \indica{\mathcal{S}^c}$.
Then, using \ref{iteratedcoherence: monotonicity}, we deduce that 
$\uptrans{}^{k} \indica{\mathcal{S}^c} \geq \uptrans{}^{k+1} \indica{\mathcal{S}^c}$ for all $k \in \nats{}$.
\end{proof}

\begin{lemma}\label{lemma: limit of Rc converges to zero}
For any $\uptrans{}$ with an absorbing closed class $\mathcal{S}$, we have that \/ $\lim_{k \to +\infty} \uptrans{}^k \indica{\mathcal{S}^c} = 0$.
\end{lemma}
\begin{proof}
The statement holds if $\mathcal{S}^c = \emptyset$ because then $\indica{\mathcal{S}^c} = 0$, which by \ref{iteratedcoherence: bounds} implies that $\uptrans{}^k \indica{\mathcal{S}^c} = 0$ for all $k \in \natz{}$.
So assume that $\mathcal{S}^c$ is non-empty.
Since $\mathcal{S}$ is absorbing [\ref{property: absorbing}], there is, for any $x \in \mathcal{S}^c$, an index $k_x \in \nats{}$ such that $\uptrans{}^{k_x} \indica{\mathcal{S}^c} (x) < 1$.
By Lemma~\ref{lemma: Rc decreases} and the fact that $\mathcal{S}$ is closed, we then also have that $\uptrans{}^{k} \indica{\mathcal{S}^c} (x) < 1$ for all $x \in \mathcal{S}^c$ and all $k \geq k_x$.
Hence, for $k \coloneqq \max_{x \in \mathcal{S}^c} k_x \in \nats{}$, 
we have that $\uptrans{}^{k} \indica{\mathcal{S}^c} (x) < 1$ for all $x \in \mathcal{S}^c$.
Now let $\alpha \coloneqq \max_{x \in \mathcal{S}^c} \uptrans{}^{k} \indica{\mathcal{S}^c} (x) < 1$.
The set $\mathcal{S}$ is a closed class, so it follows from Lemma~\ref{lemma: not from S to Sc} that $\uptrans{}^{k} \indica{\mathcal{S}^c}(x) = 0$ for all $x \in \mathcal{S}$.
Since moreover $\uptrans{}^{k} \indica{\mathcal{S}^c}(x) \leq \alpha$ for all $x \in \mathcal{S}^c$, we infer that $\uptrans{}^{k} \indica{\mathcal{S}^c} \leq \alpha \indica{\mathcal{S}^c}$.
Then, using \ref{iteratedcoherence: monotonicity}, \ref{iteratedcoherence: homogeneity} and the non-negativity [\ref{iteratedcoherence: bounds}] of $\alpha$, it follows that $\uptrans{}^{2 k} \indica{\mathcal{S}^c} \leq \alpha^2 \indica{\mathcal{S}^c}$.
Repeating this argument leads us to conclude that $\uptrans{}^{\ell k} \indica{\mathcal{S}^c} \leq \alpha^\ell \indica{\mathcal{S}^c}$ for all $\ell \in \nats{}$.
Since $\alpha$ is a non-negative real such that $\alpha < 1$, 
this implies that $\lim_{\ell \to +\infty} \uptrans{}^{\ell k} \indica{\mathcal{S}^c} \leq 0$ and therefore by \ref{iteratedcoherence: bounds} that $\lim_{\ell \to +\infty} \uptrans{}^{\ell k} \indica{\mathcal{S}^c} = 0$.
Then also $\lim_{k \to +\infty} \uptrans{}^{k} \indica{\mathcal{S}^c} = 0$ because $\uptrans{}^{k} \indica{\mathcal{S}^c}$ is non-increasing according to Lemma~\ref{lemma: Rc decreases}.
\end{proof}

\begin{lemma}\label{lemma: absorbing implies that T^k h only depends on h in R}
Consider any $\uptrans{}$ with an absorbing closed class $\mathcal{S}$.
Then, for any $\epsilon > 0$, there is a $k_1 \in \natz{}$ such that $\supnorm{\uptrans{}^k h -  \uptrans{}^k (h \indica{\mathcal{S}})} \leq \supnorm{h} \epsilon$ for all $k \geq k_1$ and all $h \in \setofgambles{}(\statespace{})$.
\end{lemma}
\begin{proof}
Fix any $\epsilon > 0$.
Because of Lemma~\ref{lemma: limit of Rc converges to zero}, we have that $\lim_{k \to +\infty} \uptrans{}^k \indica{\mathcal{S}^c} = 0$.
Since $\statespace{}$ is finite, this implies that there is an index $k_1 \in \natz{}$ such that $0 \leq \uptrans{}^k \indica{\mathcal{S}^c} \leq \epsilon$ for all $k \geq k_1$, where we also used the non-negativity [\ref{iteratedcoherence: bounds}] of $\uptrans{}^k \indica{\mathcal{S}^c}$.
Hence, multiplying by $\supnorm{h}$ for any $h \in \setofgambles{}(\statespace{})$ and using non-negative homogeneity [\ref{iteratedcoherence: homogeneity}], allows us to write that $0 \leq \uptrans{}^k (\supnorm{h} \indica{\mathcal{S}^c}) \leq \supnorm{h} \epsilon$ for all $k \geq k_1$ and all $h \in \setofgambles{}(\statespace{})$.
Moreover, note that
\begin{align*}
h \indica{\mathcal{S}} - \supnorm{h} \indica{\mathcal{S}^c} \leq h \leq h \indica{\mathcal{S}} + \supnorm{h} \indica{\mathcal{S}^c} \text{ for all } h \in \setofgambles{}(\statespace{}).
\end{align*}
Then, by subsequently applying \ref{iteratedcoherence: mixed additivity}, \ref{iteratedcoherence: monotonicity} and \ref{iteratedcoherence: subadditivity}, we find that 
\begin{align}\label{Eq: proposition: if top class absorbing then weakly ergodic 2}
\uptrans{}^k (h \indica{\mathcal{S}}) - \uptrans{}^k (\supnorm{h} \indica{\mathcal{S}^c}) 
&\leq \uptrans{}^k (h \indica{\mathcal{S}} - \supnorm{h} \indica{\mathcal{S}^c}) \nonumber \\
&\leq \uptrans{}^k h 
\leq \uptrans{}^k (h \indica{\mathcal{S}}) + \uptrans{}^k (\supnorm{h} \indica{\mathcal{S}^c}),
\end{align}
for all $h \in \setofgambles{}(\statespace{})$ and all $k \in \natz{}$.
Hence, recalling that $0 \leq \uptrans{}^k (\supnorm{h} \indica{\mathcal{S}^c}) \leq \supnorm{h} \epsilon$ for all $k \geq k_1$ and all $h \in \setofgambles{}(\statespace{})$, we indeed find that $\supnorm{\uptrans{}^k h - \uptrans{}^k (h \indica{\mathcal{S}})} \leq \supnorm{h} \epsilon$ for all $k \geq k_1$ and all $h \in \setofgambles{}(\statespace{})$.
\end{proof}

For notational convenience, we will henceforth use $\overline{m}_{f,k} \coloneqq \tfrac{1}{k} \tilde{m}_{f,k}$ for any $f\in\setofgambles{}(\statespace{})$ and any $k \in \nats{}$, to denote the function in $\setofgambles{}(\statespace{})$ that takes the value $\overline{m}_{f,k}(x) = \tfrac{1}{k} \tilde{m}_{f,k}(x) = \upprev[\mathrm{av},k](f \vert x)$ in $x \in \statespace{}$. 

\begin{lemma}\label{lemma: bounds average}
$\min f \leq \overline{m}_{f,k}(x) = \upprev[\mathrm{av},k](f \vert x) \leq \max f$ for all $f\in\setofgambles{}(\statespace{})$, all $k \in \nats{}$ and all $x\in\statespace{}$.
\end{lemma}
\begin{proof}
Fix any $f\in\setofgambles{}(\statespace{})$.
Due to the definition of $\overline{m}_{f,k}$, it clearly suffices to prove that $k \min f \leq \tilde{m}_{f,k} \leq k \max f$ for all $k \in \nats{}$.
We do this by induction.
For $k=1$, the statement holds trivially because $\tilde{m}_{f,1} = f$ and therefore $\min f \leq \tilde{m}_{f,1} \leq \max f$.
Now suppose that the statement holds for $k = i$ with $i \in \nats{}$.
Then $i \min f \leq \tilde{m}_{f,i} \leq i \max f$.
It then follows from \ref{transcoherence: bounds} and \ref{transcoherence: monotonicity} that 
\begin{align*}
i \min f = \uptrans{}(i \min f) 
\leq \uptrans{} \tilde{m}_{f,i} 
\leq \uptrans{}(i \max f)
= i \max f.
\end{align*}
By adding $f$ to all the terms, we find that 
$(i+1) \min f 
\leq f + \uptrans{} \tilde{m}_{f,i}  
\leq (i+1) \max f$.
Since $f + \uptrans{} \tilde{m}_{f,i} = \avuptrans{f}{} \tilde{m}_{f,i}
= \tilde{m}_{f,i+1}$ by Equation~\eqref{Eq: recursive expression}, it follows that the statement holds for $k=i+1$ as well.
\end{proof}

\begin{lemma}\label{lemma: T^kh - T_f^k h}
$\supnorm{\uptrans{}^k h - \avuptrans{f}{k} h } \leq k \supnorm{f}$ for all $f\in\setofgambles{}(\statespace{})$, all $h \in \setofgambles{}(\statespace{})$ and all $k \in \natz{}$.
\end{lemma}
\begin{proof}
Fix any $f\in\setofgambles{}(\statespace{})$.
It suffices to show that 
\begin{align}\label{Eq: lemma: average is equal to T^k average}
- k \supnorm{f} + \uptrans{}^k h \leq \avuptrans{f}{k} h \leq k \supnorm{f} + \uptrans{}^k h \text{ for all } h \in \setofgambles{}(\statespace{}) \text{ and all } k \in \natz{}.
\end{align}
To do so, we will use an induction argument in $k$.
The inequalities above clearly hold for all $h \in \setofgambles{}(\statespace{})$ and $k = 0$.
Now suppose that they hold for all $h \in \setofgambles{}(\statespace{})$ and all $k \in \{0,\cdots,i\}$, with $i \in \natz{}$.
Then we have that 
\begin{align*}
\avuptrans{f}{i+1} h 
\leq \avuptrans{f}{}\big( i \supnorm{f} + \uptrans{}^i h \big)
= i \supnorm{f} + \avuptrans{f}{}\big(  \uptrans{}^i h \big)
\leq (i+1) \supnorm{f} +  \uptrans{}^{(i+1)} h,
\end{align*}
for all $h \in \setofgambles{}(\statespace{})$, where the first step follows from the induction hypothesis for $k=i$ and the monotonicity [\ref{topical: monotonicity}] of $\avuptrans{f}{}$, the second from the constant additivity [\ref{topical: constant addivity}] of $\smash{\avuptrans{f}{}}$, and the third from the definition of $\avuptrans{f}{}$.
In an analogous way, we find that
\begin{align*}
\avuptrans{f}{i+1} h 
\geq \avuptrans{f}{}\big( - i \supnorm{f} + \uptrans{}^i h \big)
&= - i \supnorm{f} + \avuptrans{f}{}\big(  \uptrans{}^i h \big) \\
&\geq - (i+1) \supnorm{f} +  \uptrans{}^{(i+1)} h,
\end{align*}
for all $h \in \setofgambles{}(\statespace{})$, where the first step follows once more from the induction hypothesis for $k=i$ and the monotonicity [\ref{topical: monotonicity}] of $\avuptrans{f}{}$, the second from the constant additivity [\ref{topical: constant addivity}] of $\avuptrans{f}{}$, and the third from the definition of $\avuptrans{f}{}$.
Both inequalities together establish that the statement holds for $k=i+1$, thereby concluding the induction step.
\end{proof}

\begin{lemma}\label{lemma: average is equal to T^k average}
$\lim_{k \to +\infty} \supnorm{\uptrans{}^{\ell} \overline{m}_{f,k} - \overline{m}_{f,k+\ell}} = 0$ \/ for all $f\in\setofgambles{}(\statespace{})$ and all\/ $\ell \in \natz{}$.
\end{lemma}
\begin{proof}
Fix any $f\in\setofgambles{}(\statespace{})$, any $\ell \in \natz{}$ and any $\epsilon>0$.
Let $k \in \nats{}$ be such that $k \geq \ell \supnorm{f} / \epsilon$ and let $h \coloneqq \avuptrans{f}{k}(0) = k \overline{m}_{f,k}$.
Then
\begin{align}\label{Eq: lemma: average is equal to T^k average 2}
\epsilon 
\geq \tfrac{\ell}{k} \supnorm{f} 
\geq \tfrac{1}{k} \supnorm[\big]{\uptrans{}^\ell h - \avuptrans{f}{\ell} h }
&=  \supnorm[\big]{\uptrans{}^\ell \tfrac{1}{k} h - \tfrac{1}{k} \avuptrans{f}{\ell} h } \nonumber \\
&= \supnorm[\big]{\uptrans{}^\ell \overline{m}_{f,k} - \tfrac{1}{k} \avuptrans{f}{\ell} h },
\end{align}
where the second step follows from Lemma~\ref{lemma: T^kh - T_f^k h} and the third follows from the non-negative homogeneity [\ref{iteratedcoherence: homogeneity}] of $\uptrans{}^\ell$.
Moreover, we also have that
\begin{align*}
\supnorm[\big]{\tfrac{1}{k} \avuptrans{f}{\ell} h - \overline{m}_{f,k+\ell}}
= \supnorm[\big]{\tfrac{k+\ell}{k} \, \overline{m}_{f,k+\ell} - \overline{m}_{f,k+\ell}}
&= \tfrac{\ell}{k} \supnorm{\overline{m}_{f,k+\ell}} \\
&\leq \tfrac{\ell}{k} \supnorm{f} \leq \epsilon,
\end{align*}
where the second to last step follows from Lemma~\ref{lemma: bounds average}.
Combining this with Equation~\eqref{Eq: lemma: average is equal to T^k average 2} and using the triangle inequality, we get that
$\supnorm{\uptrans{}^\ell \overline{m}_{f,k} - \overline{m}_{f,k+\ell}}
\leq 2 \epsilon$.
Since this holds for any $\epsilon > 0$ and all $k \geq \ell \supnorm{f} / \epsilon $, we indeed have that $\lim_{k \to +\infty} \supnorm{\uptrans{}^\ell \overline{m}_{f,k} - \overline{m}_{f,k+\ell}} = 0$.
\end{proof}


\begin{proposition}\label{prop: if T is absorbing, then limit of m_f,k does not depend on S^c}
Consider any $\uptrans{}$ with a closed class $\mathcal{S}$ that is absorbing.
Then, for any $f\in\setofgambles{}(\statespace{})$ and any $y \in \statespace{}$, we have that
\begin{align*}
\min_{x \in \mathcal{S}} \liminf_{k \to +\infty} \tfrac{1}{k} [\avuptrans{f}{k}(0)](x) 
\leq \liminf_{k \to +\infty} \tfrac{1}{k} [\avuptrans{f}{k}(0)](y) 
\leq \limsup_{k \to +\infty} \tfrac{1}{k} [\avuptrans{f}{k}(0)](y) 
\leq \max_{x \in \mathcal{S}} \limsup_{k \to +\infty} \tfrac{1}{k} [\avuptrans{f}{k}(0)](x). 
\end{align*}
\end{proposition}
\begin{proof}
First note that the proposition is trivially satisfied if $f=0$ because Lemma~\ref{lemma: bounds average} and Equation~\eqref{Eq: recursive expression 2} then imply that $\tfrac{1}{k} [\avuptrans{f}{k}(0)](x) = \upprev[\mathrm{av},k](f\vert x) = 0$ for all $x \in \statespace{}$ and $k \in \nats{}$.
So suppose that $f \not= 0$, fix any $\epsilon > 0$ and let $\epsilon_1 \coloneqq (\sfrac{1}{\supnorm{f}}) \epsilon$.
Choose $\ell_1$ such that Lemma~\ref{lemma: absorbing implies that T^k h only depends on h in R} holds with $\epsilon_1$; that is, in such a way that 
\begin{align}\label{Eq: proposition: if top class absorbing then weakly ergodic 5}
\supnorm{\uptrans{}^\ell (h \indica{\mathcal{S}}) - \uptrans{}^\ell h} \leq \epsilon_1 \supnorm{h} \text{ for all } \ell \geq \ell_1 \text{ and all } h \in \setofgambles{}(\statespace{}).
\end{align}
Then, in particular, for any $\ell \geq \ell_1$, we have that
$\supnorm[\big]{\uptrans{}^\ell (\overline{m}_{f,k} \indica{\mathcal{S}}) - \uptrans{}^\ell \overline{m}_{f,k}} 
\leq \epsilon_1 \supnorm{\overline{m}_{f,k}} 
\leq \epsilon_1 \supnorm{f}
= \epsilon \text{ for all } k \in \nats{}$,
where the second inequality follows from Lemma~\ref{lemma: bounds average}.
Fix any such $\ell \geq \ell_1$.
Furthermore, recall Lemma~\ref{lemma: average is equal to T^k average}, which guarantees that there is some $k_1 \in \nats{}$ such that $\supnorm{\uptrans{}^\ell \overline{m}_{f,k} - \overline{m}_{f,k+\ell}} \leq \epsilon$ for all $k \geq k_1$.
Combining this with the inequality above, we get that $\supnorm{\uptrans{}^\ell (\overline{m}_{f,k} \indica{\mathcal{S}}) - \overline{m}_{f,k+\ell}} \leq 2\epsilon$ for all $k \geq k_1$ or, equivalently, that 
\begin{align}\label{Eq: proposition: if top class absorbing then weakly ergodic 3}
\uptrans{}^\ell (\overline{m}_{f,k} \indica{\mathcal{S}}) - 2\epsilon 
\leq \overline{m}_{f,k+\ell} 
\leq \uptrans{}^\ell (\overline{m}_{f,k} \indica{\mathcal{S}}) + 2\epsilon \text{ for all } k \geq k_1.
\end{align}

Now, as a consequence of the definitions of the limit inferior and the limit superior operators and the fact that $\overline{m}_{f,k}$ is bounded due to Lemma~\ref{lemma: bounds average}, there is, for any $x\in\mathcal{S}$, an index $k_{x}\in\nats$ such that 
\begin{align*}
\liminf_{k'\to+\infty} \overline{m}_{f,k'}(x) - \epsilon
\leq \overline{m}_{f,k}(x)
\leq \limsup_{k'\to+\infty} \overline{m}_{f,k'}(x) + \epsilon \text{ for all } k\geq k_x.
\end{align*} 
Then, if we let $a\coloneqq \min_{x\in\mathcal{S}}\liminf_{k'\to+\infty} \overline{m}_{f,k'}(x)$ and $b\coloneqq \max_{x\in\mathcal{S}}\limsup_{k'\to+\infty} \overline{m}_{f,k'}(x)$, we certainly have that $a - \epsilon
 \leq \overline{m}_{f,k}(x)
 \leq b + \epsilon$ for any $x\in\mathcal{S}$ and all $k\geq k_x$.
Hence,
\begin{align*}
a \indica{\mathcal{S}} - \epsilon
 \leq \overline{m}_{f,k} \indica{\mathcal{S}}
 \leq b \indica{\mathcal{S}} + \epsilon 
 \text{ for all } k\geq k_2 \coloneqq \max_{x\in\mathcal{S}} k_x \in \nats,
\end{align*}
which by \ref{iteratedcoherence: monotonicity} and \ref{iteratedcoherence: constant addivity} implies that 
\begin{align*}
\uptrans{}^\ell (a \indica{\mathcal{S}}) - \epsilon
 \leq \uptrans{}^\ell (\overline{m}_{f,k} \indica{\mathcal{S}})
 \leq \uptrans{}^\ell (b \indica{\mathcal{S}}) + \epsilon 
 \text{ for all } k\geq k_2.
\end{align*}
Together with Equation~\eqref{Eq: proposition: if top class absorbing then weakly ergodic 3}, this implies that
\begin{align}\label{Eq: proposition: if top class absorbing then weakly ergodic 4}
\uptrans{}^\ell (a \indica{\mathcal{S}}) - 3\epsilon
 \leq \overline{m}_{f,k+\ell}
 \leq \uptrans{}^\ell (b \indica{\mathcal{S}}) + 3\epsilon 
 \text{ for all } k \geq \max\{k_1,k_2\}.
\end{align}
Now, recall Equation~\eqref{Eq: proposition: if top class absorbing then weakly ergodic 5} and observe that therefore 
\begin{align*}
\supnorm[\big]{\uptrans{}^\ell (a \indica{\mathcal{S}}) - a}
= \supnorm[\big]{\uptrans{}^\ell (a \indica{\mathcal{S}}) - \uptrans{}^\ell a} 
\leq \epsilon_1 \supnorm{a},
\end{align*}
where the first step follows from \ref{iteratedcoherence: bounds}.
Moreover, we have that $\supnorm{a} = \vert a \vert \leq \supnorm{f}$ due to Lemma~\ref{lemma: bounds average}, so we conclude that $\supnorm{\uptrans{}^\ell (a \indica{\mathcal{S}}) - a} \leq \epsilon_1 \supnorm{f} = \epsilon$.
In a completely analogous way, we can deduce that $\supnorm{\uptrans{}^\ell (b \indica{\mathcal{S}}) - b} \leq \epsilon$.
So, in particular, we have that $a- \epsilon \leq \uptrans{}^\ell (a \indica{\mathcal{S}})$ and that $\uptrans{}^\ell (b \indica{\mathcal{S}}) \leq b+\epsilon$.
Applying these inequalities to Equation~\eqref{Eq: proposition: if top class absorbing then weakly ergodic 4}, we obtain that $a - 4\epsilon
\leq \overline{m}_{f,k+\ell}
\leq b + 4\epsilon$ for all $k \geq \max\{k_1,k_2\}$.
Hence, 
\begin{align*}
a - 4\epsilon
\leq \liminf_{k\to+\infty} \overline{m}_{f,k}
\leq \limsup_{k\to+\infty} \overline{m}_{f,k}
\leq b + 4\epsilon
\end{align*}
Since this holds for any $\epsilon>0$, we conclude that $a \leq \liminf_{k\to+\infty} \overline{m}_{f,k} \leq \limsup_{k\to+\infty} \overline{m}_{f,k}
\leq b$, which by the definition of $a$ and $b$, and the fact that, due to Equation~\eqref{Eq: recursive expression 2}, $\overline{m}_{f,k} = \tfrac{1}{k} [\avuptrans{f}{k}(0)]$ for all $k\in\natz$, implies the desired statement.
\end{proof}

\begin{proofof}{Proposition~\ref{proposition: if top class absorbing then weakly ergodic}.}
Assume that $\uptrans{}$ satisfies (TCA) and let $\mathcal{R}$ be the corresponding top class.
Fix any $f\in\setofgambles{}(\statespace{})$.
According to Lemma~\ref{lemma: if single communication class then top class}, $\mathcal{R}$ is a closed communication class, which by Proposition~\ref{proposition: if top class then eigenvector in top class} implies that $\lim_{k\to+\infty}\upprev[\mathrm{av},k](f \vert x) = \lim_{k\to+\infty}\overline{m}_{f,k}(x) = \mu$ for some $\mu\in\reals{}$ and all $x\in\mathcal{R}$.
Since $\mathcal{R}$ is closed and absorbing [\ref{property: absorbing}], we can then apply Proposition~\ref{prop: if T is absorbing, then limit of m_f,k does not depend on S^c}---taking into account that $\tfrac{1}{k} [\avuptrans{f}{k}(0)]=\overline{m}_{f,k}$ for all $k\in\nats$ due to Equation~\eqref{Eq: recursive expression 2}---to find that, for all $y\in\statespace{}$,
$\mu 
\leq \liminf_{k \to +\infty} \overline{m}_{f,k} (y) 
\leq \limsup_{k \to +\infty} \overline{m}_{f,k} (y) 
\leq \mu$. 
Hence, we have that $\lim_{k \to +\infty} \overline{m}_{f,k} = \mu$ and conclude that $\uptrans{}$ is weakly ergodic.
\end{proofof}

\section{Proof of Propositions~\ref{prop: top class if weak ergodicity} and~\ref{prop: TCA if weak ergodicity}}\label{Sect: app: weakly ergodic iff top class absorbing}

In the following proofs, as well as in the proofs of \ref{Sect: appendix: repetition independence}, we will often implicitly use the fact that any given top class $\mathcal{R}$ is necessarily closed; see Lemma~\ref{lemma: if single communication class then top class}.

\begin{proofof}{Proposition~\ref{prop: top class if weak ergodicity}.}
Suppose that $\uptrans{}$ does not have a top class.
Then due to Corollary~\ref{corollary: no top class}, there are (at least) two closed communication classes $\mathcal{S}_1$ and $\mathcal{S}_2$.
Consider any two $c_1, c_2 \in \reals{}$ such that $c_1 \not= c_2$, and let $f \coloneqq c_1 \indica{\mathcal{S}_1} + c_2 \indica{\mathcal{S}_2}$.
Since $f \indica{\mathcal{S}_1} = c_1 \indica{\mathcal{S}_1}$, Lemma~\ref{lemma: average does not depend on f outside S} implies that 
$\tfrac{1}{k} [\avuptrans{f}{k}(0)]\, \indica{\mathcal{S}_1} 
= \tfrac{1}{k} [\avuptrans{c_1}{k}(0)]\, \indica{\mathcal{S}_1}$ for all $k \in \nats{}$ or, by Equation~\eqref{Eq: recursive expression 2}, that $\upprev[\mathrm{av},k](f \vert x) = \upprev[\mathrm{av},k](c_1 \vert x)$ for all $x \in \mathcal{S}_1$ and all $k \in \nats{}$.
By Lemma~\ref{lemma: bounds average}, we know that for any $x \in \mathcal{S}_1$ and any $k \in \nats{}$, $\upprev[\mathrm{av},k](f \vert x) = \upprev[\mathrm{av},k](c_1 \vert x) = c_1$.
Hence, $\lim_{k \to +\infty} \upprev[\mathrm{av},k](f \vert x) = c_1$ for all $x \in \mathcal{S}_1$.
In a completely analogous way, we can deduce that $\lim_{k \to +\infty} \upprev[\mathrm{av},k](f \vert x) = c_2$ for all $x \in \mathcal{S}_2$.
By assumption, $c_1 \not= c_2$, so we can conclude that the upper expectation $\upprev[\mathrm{av},k](f \vert x)$ with $f = c_1 \indica{\mathcal{S}_1} + c_2 \indica{\mathcal{S}_2}$, does not converge to a constant that is equal for all $x \in \statespace{}$.
Hence, the upper transition operator $\uptrans{}$ is not weakly ergodic.
\end{proofof}

\begin{lemma}\label{lemma: not TCA then TI_A equals one}
Consider any $\uptrans{}$ that has a top class $\mathcal{R}$ but that does not satisfy (TCA).
Then there is a non-empty subset $A \subseteq \mathcal{R}^c$ such that $\indica{A} \leq \uptrans{}\indica{A}$.
\end{lemma}
\begin{proof}
If $\uptrans{}$ has a top class $\mathcal{R}$ that is not absorbing, then there is at least one $x \in \mathcal{R}^c$ such that $\uptrans{}^k \indica{\mathcal{R}^c}(x) = 1$ for all $k \in \nats{}$.
Let $A \subseteq \mathcal{R}^c$ be the set of all such states $x \in \mathcal{R}^c$.
If $A = \mathcal{R}^c$ then, since $\uptrans{} \indica{\mathcal{R}^c}(x) = 1$ for all $x \in A = \mathcal{R}^c$ and since $\uptrans{} \indica{\mathcal{R}^c}(x) = 0$ for all $x \in \mathcal{R}$ by Lemma~\ref{lemma: not from S to Sc}, we have that $\uptrans{}\indica{\mathcal{R}^c} = \indica{\mathcal{R}^c}$.
Hence, in that case, the statement holds. In the remainder of the proof, we can therefore assume that $A \subset \mathcal{R}^c$, implying that $\mathcal{R}^c \setminus A$ is non-empty.

Observe that by the definition of $A$ there is for any $x \in \mathcal{R}^c \setminus A$, an index $k_x \in \nats{}$ such that $\uptrans{}^{k_x} \indica{\mathcal{R}^c} (x) \not= 1$, and therefore, due to \ref{iteratedcoherence: bounds}, also that $\uptrans{}^{k_x} \indica{\mathcal{R}^c} (x) < 1$.
By Lemma~\ref{lemma: Rc decreases}, it follows that then also $\uptrans{}^{k} \indica{\mathcal{R}^c} (x) < 1$ for all $x \in \mathcal{R}^c \setminus A$ and all $k \geq k_x$.
Hence, for $k \coloneqq \smash{\max_{x \in \mathcal{R}^c \setminus A} k_x} \in \nats{}$, we have that $\uptrans{}^{k} \indica{\mathcal{R}^c} (x) < 1$ for all $x \in \mathcal{R}^c \setminus A$.
Let $\alpha \coloneqq \smash{\max_{x \in \mathcal{R}^c \setminus A} \uptrans{}^{k} \indica{\mathcal{R}^c} (x)} < 1$.
Since $\uptrans{}^{k} \indica{\mathcal{R}^c}(x) = 0$ for all $x \in \mathcal{R}$ due to Lemma~\ref{lemma: not from S to Sc} and $0 \leq \alpha$ due to \ref{iteratedcoherence: bounds}, we infer that $\uptrans{}^{k} \indica{\mathcal{R}^c}(x) \leq \alpha$ for all $x \in \mathcal{R} \cup (\mathcal{R}^c \setminus A) = A^c$ or, equivalently, that $\indica{A^c} \uptrans{}^{k} \indica{\mathcal{R}^c} \leq \alpha \indica{A^c}$. Hence,
\begin{equation*}
\uptrans{}^k \indica{\mathcal{R}^c}
=\indica{A} \uptrans{}^k \indica{\mathcal{R}^c}+\indica{A^c}\uptrans{}^k \indica{\mathcal{R}^c}
\leq\indica{A} \uptrans{}^k \indica{\mathcal{R}^c}+ \alpha \indica{A^c}
=\indica{A}+ \alpha \indica{A^c},
\end{equation*}
using the definition of $A$ for the last equality. It follows that
\begin{align*}
\uptrans{}^{k+1} \indica{\mathcal{R}^c} 
=
\uptrans{}\,\uptrans{}^{k} \indica{\mathcal{R}^c} 
\leq \uptrans{} \big( \indica{A} + \alpha \indica{A^c} \big) 
= \uptrans{} \big( \alpha + (1 - \alpha) \indica{A} \big) 
= \alpha + (1 - \alpha) \uptrans{} \indica{A}, 
\end{align*}
where we used monotonicity [\ref{transcoherence: monotonicity}] for the inequality and \ref{transcoherence: constant addivity} and \ref{transcoherence: homogeneity} for the last equality. Multiplying with $\indica{A}$ yields $\indica{A} \uptrans{}^{k+1} \indica{\mathcal{R}^c} \leq \alpha\indica{A} + (1 - \alpha) \indica{A} \uptrans{} \indica{A}$ and therefore, since the definition of $A$ implies that $\indica{A}=\indica{A}\uptrans{}^{k+1} \indica{\mathcal{R}^c}$, we find that $\indica{A} \leq \alpha\indica{A} + (1 - \alpha) \indica{A} \uptrans{} \indica{A}$, or equivalently, that $(1 - \alpha) \indica{A} \leq (1 - \alpha) \indica{A} \uptrans{} \indica{A}$.
Since $1-\alpha > 0$, it follows that $\indica{A} \leq \indica{A} \uptrans{} \indica{A}$, which implies that $\indica{A} \leq \uptrans{} \indica{A}$ because $\uptrans{} \indica{A}$ is non-negative [\ref{transcoherence: bounds}].
\end{proof}

\medskip 

\begin{proofof}{Proposition~\ref{prop: TCA if weak ergodicity}.}
Suppose that $\uptrans{}$ has a top class $\mathcal{R}$, but that $\mathcal{R}$ is not absorbing.
Then Lemma~\ref{lemma: not TCA then TI_A equals one} guarantees that there is a non-empty subset $A \subseteq \mathcal{R}^c$ such that $\indica{A} \leq \uptrans{}\indica{A}$.
Now consider any $x \in A$ and any $y \in \mathcal{R}$.
We will show that $\lim_{k \to +\infty} \upprev[\mathrm{av},k](\indica{A} \vert x) = 1$ and that $\lim_{k \to +\infty} \upprev[\mathrm{av},k](\indica{A} \vert y) = 0$, implying that $\uptrans{}$ cannot be weakly ergodic.

To prove that $\lim_{k \to +\infty} \upprev[\mathrm{av},k](\indica{A} \vert x) = 1$, we first show by induction that $\tilde{m}_{\indica{A},k} \geq k \indica{A}$ for all $k \in \nats{}$.
By definition, we have that $\tilde{m}_{\indica{A},1} = \indica{A}$, which establishes our induction base.
To prove the induction step, assume that the inequality holds for $k = i$, with $i \in \nats{}$, so $\tilde{m}_{\indica{A},i} \geq i \indica{A}$.
Then according to the recursive expression \eqref{Eq: recursive expression},
\begin{align*}
\tilde{m}_{\indica{A},i+1} 
= \indica{A} + \uptrans{} \tilde{m}_{\indica{A},i}
\geq \indica{A} + \uptrans{} (i \indica{A})
\geq (i+1) \indica{A},
\end{align*}
where the second step follows from the induction hypothesis and the monotonicity [\ref{transcoherence: monotonicity}] of $\uptrans{}$, and the last from \ref{transcoherence: homogeneity} together with the fact that $\indica{A} \leq \uptrans{}\indica{A}$.
This implies that the inequality holds for $k=i+1$ as well, hence finalising our induction argument.
Taking into account Equation~\eqref{Eq: recursive expression 2}, we conclude that $\upprev[\mathrm{av},k](\indica{A} \vert x) = \tfrac{1}{k} \tilde{m}_{\,\indica{A},k}(x) \geq \indica{A}(x)$ for all $k \in \nats{}$.
Due to Lemma~\ref{lemma: bounds average} and since $x \in A$, this implies that $\smash{\upprev[\mathrm{av},k](\indica{A} \vert x)} = 1$ for all $k \in \nats{}$.
Hence, $\smash{\lim_{k \to +\infty} \upprev[\mathrm{av},k](\indica{A} \vert x)} = 1$.

It remains to prove that $\lim_{k \to +\infty} \upprev[\mathrm{av},k](\indica{A} \vert y) = 0$.
Because $A \subseteq \mathcal{R}^c$, we have that $\indica{A} \indica{\mathcal{R}} = 0 = 0 \, \indica{\mathcal{R}}$, and since $\mathcal{R}$ is a closed communication class, Lemma~\ref{lemma: average does not depend on f outside S} implies that $\tfrac{1}{k}[\avuptrans{\indica{A}}{k}(0)]\,\indica{\mathcal{R}} = \tfrac{1}{k}[\avuptrans{0}{k}(0)]\,\indica{\mathcal{R}}$ for all $k \in \nats{}$.
Hence, by Equation~\eqref{Eq: recursive expression 2}, we have that $\upprev[\mathrm{av},k](\indica{A} \vert y) = \upprev[\mathrm{av},k](0\, \vert y)$ for all $k \in \nats{}$, and therefore, by Lemma~\ref{lemma: bounds average}, that $\upprev[\mathrm{av},k](\indica{A} \vert y) = 0$ for all $k \in \nats{}$.
As a consequence, $\lim_{k \to +\infty} \upprev[\mathrm{av},k](\indica{A} \vert y) = 0$.
\end{proofof}

\section{Proof of the results in Section~\ref{sect: weak ergo in repetition independence}}\label{Sect: appendix: repetition independence}

For any upper transition operator $\uptrans{}$, any transition matrix $T$, any $f\in\setofgambles{}(\statespace{})$ and any non-empty $\mathcal{S}\subseteq\statespace{}$, we define the maps $T_{f,\mathcal{S}}$, $\avuptrans{\mathcal{S}}{}$ and $T_\mathcal{S}$ on $\setofgambles{}(\mathcal{S})$ similarly to how we previously defined the map $\avuptrans{f,\mathcal{S}}{}$;
so $T_{f,\mathcal{S}}h\coloneqq(T_f h^\uparrow)\vert_\mathcal{S}$, $\avuptrans{\mathcal{S}}{}h\coloneqq(\uptrans{} h^\uparrow)\vert_\mathcal{S}$ and $T_{\mathcal{S}}h\coloneqq(T h^\uparrow)\vert_\mathcal{S}$ for all $h\in\setofgambles{}(\mathcal{S})$.
In the following statements, we implicitly assume that $\settrans$ is a separately specified set of transition matrices, that $\uptrans{}$ is the associated upper transition operator and that $f$ is any function in $\setofgambles{}(\statespace{})$.

\begin{lemma}\label{lemma: settrans_S is set of transition matrices}
For any closed class $\mathcal{S}$ in $\mathscr{G}(\uptrans{})$, we have that $\{T_\mathcal{S}\colon T\in\settrans{}\}$ is a separately specified set of \/ $\vert\mathcal{S}\vert \times \vert\mathcal{S}\vert$ transition matrices and that $\avuptrans{\mathcal{S}}{}$ is the corresponding upper transition operator.
\end{lemma}
\begin{proof}
It follows trivially from its definition that, for any $T\in\settrans$, the map $T_\mathcal{S}$ is represented by a $\vert\mathcal{S}\vert \times \vert\mathcal{S}\vert$ matrix where the $(i,j)$-th element is equal to the $(i,j)$-th element in $T$ for all $i,j \in \{1,\cdots,\vert\mathcal{S}\vert\}$.
So, to ensure that $T_\mathcal{S}$ is row-stochastic---and can therefore be called a transition matrix---we only have to check whether the elements in each row sum to $1$.
To see that this is indeed the case, note that, due to Lemma~\ref{lemma: S is maximal in G(uptrans) then S is maximal in G(T)} and the fact that $\mathcal{S}$ is closed in $\mathscr{G}(\uptrans{})$, the class $\mathcal{S}$ is closed in $\mathscr{G}(T)$.
Then Lemma~\ref{lemma: not from S to Sc} implies that $T \indica{\mathcal{S}^c}(x) = 0$ for all $x\in\mathcal{S}$ and therefore, since $T$ is a transition matrix (which is linear), that $T \indica{\mathcal{S}}(x) = \sum_{y\in\mathcal{S}} T(x,y) = 1$ for all $x\in\mathcal{S}$.
Hence, recalling our previous considerations, we also have that $\sum_{y\in\mathcal{S}} T_\mathcal{S}(x,y) = 1$ for all $x\in\mathcal{S}$, which allows us to call $T_\mathcal{S}$ a transition matrix.
Furthermore, that $\{T_\mathcal{S}\colon T\in\settrans{}\}$ is separately specified follows immediately from the fact that $\settrans$ is separately specified and that, as explained above, for any $T\in\settrans$, the matrix $T_\mathcal{S}$ is simply a restriction of $T$ to the first $\vert\mathcal{S}\vert$ dimensions.
Finally, to see that $\avuptrans{\mathcal{S}}{}$ is the upper expectation operator corresponding to the set $\settrans' \coloneqq \{T_\mathcal{S}\colon T\in\settrans\}$---meaning that $[\avuptrans{\mathcal{S}}{} h] (x) = \sup_{T'\in\settrans'} T'h (x)$ for all $h\in\setofgambles{}(\mathcal{S})$ and all $x\in\mathcal{S}$---it suffices to observe that, for any $h\in\setofgambles{}(\mathcal{S})$ and any $x\in\mathcal{S}$,
\begin{align*}
[\avuptrans{\mathcal{S}}{} h] (x) 
= [\uptrans{} h^\uparrow] (x)
= \sup_{T\in\settrans} T h^\uparrow (x)
= \sup_{T\in\settrans} T_\mathcal{S} h (x)
= \sup_{T'\in\settrans'} T'h (x). 
\tag*{\qedhere}
\end{align*}
\end{proof}

\begin{lemma}\label{Lemma: if max commun class S then IMC_ri weakly ergodic for states in S}
Consider any closed communication class $\mathcal{S}$ in $\mathscr{G}(\uptrans{})$ and let $\mu$ be the unique eigenvalue of\/ $\avuptrans{f,\mathcal{S}}{}$.
Then, for any $\epsilon>0$, there is a $T\in\settrans$, such that
\begin{align*}
\mu -\epsilon 
\leq \liminf_{k\to+\infty} \tfrac{1}{k} [T_f^{k}(0)]\vert_\mathcal{S} \leq \mu.
\end{align*}
\end{lemma}
\begin{proof}
Due to Lemma~\ref{lemma: settrans_S is set of transition matrices} and the fact that $\mathcal{S}$ is closed in $\mathscr{G}(\uptrans{})$, we have that $\settrans'\coloneqq\{T_\mathcal{S}\colon T\in\settrans\}$ is a separately specified set of transition matrices and that $\uptrans{}' \coloneqq \avuptrans{\mathcal{S}}{}$ is the corresponding upper transition operator.
Furthermore, if we let $g\coloneqq f\vert_\mathcal{S}$, then we can write that $\avuptrans{f,\mathcal{S}}{} h = f\vert_\mathcal{S} + (\uptrans{} h^\uparrow)\vert_\mathcal{S} = g + \uptrans{}' h = \avuptrans{g}{\prime} h$ for all $h\in\setofgambles{}(\mathcal{S})$.
Now, because of Lemma~\ref{lemma: T_f,S has an eigenvector} and the fact that $\mathcal{S}$ is a closed communication class in $\mathscr{G}(\uptrans{})$, the map $\avuptrans{f,\mathcal{S}}{}$ has a unique eigenvalue $\mu$, which implies that $\mu$ is also the unique eigenvalue of the map $\avuptrans{g}{\prime}$.
Recalling that $\uptrans{}'$ is the upper transition operator corresponding to the separately specified set $\settrans'$ of transition matrices, we obtain from Lemma~\ref{lemma: if T_f has eigenvalue then repetition independence weakly ergodic} that, for any $\epsilon>0$, there is a $T'\in\settrans'$ such that
\begin{align}\label{Eq: Lemma: if max commun class S then IMC_ri weakly ergodic for states in S}
\mu-\epsilon \leq \liminf_{k\to+\infty}\tfrac{1}{k}[(T'_{g})^k(0)] \leq \mu.
\end{align}
Furthermore, let $T\in\settrans$ be such that $T_\mathcal{S}=T'$; there is at least one such a $T$ because $T'\in\settrans'=\{T_\mathcal{S}\colon T\in\settrans\}$. 
Then, similarly as before, we have that $T_{f,\mathcal{S}} h = f\vert_\mathcal{S} + (T h^\uparrow)\vert_\mathcal{S} = g + T'h = T'_{g} h$ for all $h\in\setofgambles{}(\mathcal{S})$.
In particular, we infer that $[T_{f,\mathcal{S}}^k(0)] = [(T'_{g})^k(0)]$ for all $k\in\nats$.
Moreover, since $\mathcal{S}$ is closed in $\mathscr{G}(\uptrans{})$, the class $\mathcal{S}$ is also closed in $\mathscr{G}(T)$ due to Lemma~\ref{lemma: S is maximal in G(uptrans) then S is maximal in G(T)}, so Lemma~\ref{lemma: G^k is equal to G_R^k}
implies that $[T_{f,\mathcal{S}}^{k}(0)]= [T_f^{k}(0)]\vert_\mathcal{S}$ for all $k\in\nats$.
As a result, we obtain that $[(T'_{g})^k(0)] = [T_f^{k}(0)]\vert_\mathcal{S}$ for all $k\in\nats$, which, by plugging this back into Equation~\eqref{Eq: Lemma: if max commun class S then IMC_ri weakly ergodic for states in S}, implies the desired statement.
\end{proof}

\begin{proofof}{Proposition~\ref{Prop: TCA is sufficient for weak ergodicity of IMC under rep. ind.}.}
Fix any $f\in\setofgambles{}(\statespace{})$, any $\epsilon>0$ and suppose that $\uptrans{}$ satisfies (TCA) with top class $\mathcal{R}$.
Due to Lemma~\ref{Lemma: if max commun class S then IMC_ri weakly ergodic for states in S} and the fact that $\mathcal{R}$ is a closed communication class in $\mathscr{G}(\uptrans{})$, there is a $T\in\settrans$ such that 
\begin{align}\label{Eq: Prop: TCA is sufficient for weak ergodicity of IMC under rep. ind. 1}
\mu-\epsilon \leq \liminf_{k\to+\infty} \tfrac{1}{k} [T_f^{k}(0)](x) \text{ for all } x\in\mathcal{R},
\end{align}
where $\mu$ is the unique eigenvalue of\/ $\avuptrans{f, \, \mathcal{R}}{}$.
Furthermore, note that $\mathcal{R}$ is also an absorbing closed class in $\mathscr{G}(T)$ because of Lemma~\ref{lemma: TCA then all T have absorbing maximal class}, 
so Proposition~\ref{prop: if T is absorbing, then limit of m_f,k does not depend on S^c} implies that
\begin{align*}
\min_{x \in \mathcal{R}} \liminf_{k \to +\infty} \tfrac{1}{k} [T_f^k(0)](x)
\leq \liminf_{k \to +\infty} \tfrac{1}{k} [T_f^k(0)](y) \text{ for all } y\in\statespace{}. 
\end{align*}
Combining this with Equation~\eqref{Eq: Prop: TCA is sufficient for weak ergodicity of IMC under rep. ind. 1}, we infer that 
$\mu - \epsilon \leq \liminf_{k\to+\infty} \tfrac{1}{k} [T_f^{k}(0)](x)$ for all $x\in\statespace{}$, and therefore certainly that $\mu - \epsilon \leq \liminf_{k\to+\infty} \sup_{T\in\settrans} \tfrac{1}{k} [T_f^{k}(0)](x)$ for all $x\in\statespace{}$.
This holds for any $\epsilon>0$, so we conclude that  
\begin{align*}
\mu 
\leq \liminf_{k\to+\infty} \sup_{T\in\settrans} \tfrac{1}{k} [T_f^{k}(0)](x) 
= \liminf_{k \to +\infty} \avriupprev{k}(f\vert x) \text{ for all } x\in\statespace{},
\end{align*}
where the last equality follows from Equation~\eqref{Eq: recursive formula rep. independence}. 
Now recall that $\mu$ is the unique eigenvalue of $\avuptrans{f,\mathcal{R}}{}$ and therefore, that $\lim_{k\to+\infty} \upprev[\mathrm{av},k](f \vert x) = \mu$ for all $x\in\statespace{}$.
Indeed, this can easily be deduced from our earlier results and the fact that $\mathcal{R}$ is an absorbing top class in $\mathscr{G}(\uptrans)$; Lemma~\ref{lemma: if top class then eigenvector in top class} and Equation~\eqref{Eq: recursive expression 2} imply that $\lim_{k\to+\infty} \upprev[\mathrm{av},k](f \vert x) = \mu$ for all $x\in\mathcal{R}$, which then also implies that $\lim_{k\to+\infty} \upprev[\mathrm{av},k](f \vert y) = \mu$ for all $y\in\mathcal{R}^c$ because $\uptrans{}$ is weakly ergodic due to Proposition~\ref{proposition: if top class absorbing then weakly ergodic}. 
Since moreover, for all $x\in\statespace{}$ and all $k\in\nats$, $\avriupprev{k}(f\vert x) \leq \upprev[\mathrm{av},k](f\vert x)$ because $\rimarkov{\settrans} \subseteq \cimarkov{\settrans} \subseteq \eimarkov{\settrans}$, we obtain that 
\begin{align*}
\mu 
\leq \liminf_{k \to +\infty} \, \avriupprev{k}(f\vert x)
\leq \limsup_{k \to +\infty} \, \avriupprev{k}(f\vert x)
\leq \limsup_{k\to+\infty} \, \upprev[\mathrm{av},k](f \vert x)
= \lim_{k\to+\infty} \upprev[\mathrm{av},k](f \vert x) 
= \mu \text{ for all }x\in\statespace{}.
\end{align*}
This implies that $\rimarkov{\settrans{}}$ is weakly ergodic and that $\avriupprev{\infty}(f) = \upprev[\mathrm{av},\infty](f)$.
\end{proofof}

\begin{proofof}{Proposition~\ref{Prop: if no top class then IMC under ri is not weakly ergodic}.}
Suppose that $\mathscr{G}(\uptrans{})$ does not have a top class.
Due to Corollary~\ref{corollary: no top class}, there are (at least) two, disjoint, closed communication classes $\mathcal{S}_1$ and $\mathcal{S}_2$.
Consider any two $c_1, c_2 \in \reals{}$ such that $c_1 \not= c_2$, and let $f \coloneqq c_1 \indica{\mathcal{S}_1} + c_2 \indica{\mathcal{S}_2}$.
Furthermore, fix any $T\in\settrans$ and observe that, because of Lemma~\ref{lemma: S is maximal in G(uptrans) then S is maximal in G(T)}, the classes $\mathcal{S}_1$ and $\mathcal{S}_2$ are also closed (and obviously disjoint) in $\mathscr{G}(T)$.
Since $f \indica{\mathcal{S}_1} = c_1 \indica{\mathcal{S}_1}$, Lemma~\ref{lemma: average does not depend on f outside S} implies that $\tfrac{1}{k}[T_f^{k}(0)]\,\indica{\mathcal{S}_1} = \tfrac{1}{k}[T^k_{c_1}(0)]\, \indica{\mathcal{S}_1}$ for all $k \in \nats{}$.
Moreover, it follows trivially from the definition of the map $T_{c_1}$ and the fact that $T$ is a transition matrix (or, using \ref{transcoherence: bounds}), that $T_{c_1}^k(0) = k c_1$ for all $k\in\nats$, so we have that $\tfrac{1}{k}[T^k_{c_1}(0)] \, \indica{\mathcal{S}_1} = c_1 \indica{\mathcal{S}_1}$ for all $k \in \nats{}$.
Taking our earlier considerations into account, we obtain that $\smash{\tfrac{1}{k}[T_f^{k}(0)]}\,\indica{\mathcal{S}_1} = c_1 \indica{\mathcal{S}_1}$ for all $k \in \nats{}$.
In a completely analogous way, we can deduce that $\tfrac{1}{k}[T_f^{k}(0)]\,\indica{\mathcal{S}_2} = c_2 \indica{\mathcal{S}_2}$ for all $k \in \nats{}$.
Hence, for any $x\in\mathcal{S}_1$ and any $y\in\mathcal{S}_2$, we have that $\tfrac{1}{k}[T_f^{k}(0)](x) = c_1$ and $\smash{\tfrac{1}{k}[T_f^{k}(0)](y)} = c_2$ for all $k \in \nats{}$.
This holds for any $T\in\settrans$, so we infer that $\sup_{T\in\settrans}\tfrac{1}{k}[T_f^{k}(0)](x) = c_1$ and $\sup_{T\in\settrans}\tfrac{1}{k}[T_f^{k}(0)](y) = c_2$ for all $k \in \nats{}$.
By Equation~\eqref{Eq: recursive formula rep. independence}, this implies that $\avriupprev{k}(f \vert x) = c_1$ and $\avriupprev{k}(f \vert y) = c_2$ for all $k \in \nats{}$, and since $c_1 \not= c_2$ by assumption, we conclude that $\avriupprev{k}(f \vert x')$ with $f = c_1 \indica{\mathcal{S}_1} + c_2 \indica{\mathcal{S}_2}$ does not converge to a constant that is equal for all $x' \in \statespace{}$.
\end{proofof}

\begin{proofof}{Proposition~\ref{Prop: if not absorbing then IMC under ri is not weakly ergodic}.}
Suppose that $\mathscr{G}(\uptrans{})$ has a top class $\mathcal{R}$, but that $\uptrans{}$ does not satisfy (TCA).
We show that $\rimarkov{\settrans}$ is not weakly ergodic.

Since $\uptrans{}$ does not satisfy (TCA), there is, according to Lemma~\ref{lemma: not TCA then TI_A equals one}, a non-empty set $A\subseteq\mathcal{R}^c$ such that $\indica{A} \leq \uptrans{} \indica{A}$.
Since $\settrans$ is separately specified, \ref{sep specified} ensures that, for any $\epsilon>0$, there is a $T\in\settrans$ such that 
\begin{align}\label{Eq: Prop: if not absorbing then IMC under ri is not weakly ergodic}
\indica{A}-\epsilon \leq \uptrans{}\indica{A}-\epsilon \leq T\indica{A}.
\end{align}
Then it is easy to see, using an induction argument, that $[T_{\indica{A}}^k(0)] \geq k\indica{A}- \tfrac{k(k-1)}{2}\epsilon$ for all $k\in\nats$.
Indeed, the inequality trivially holds for $k=1$ because $[T_{\,\indica{A}}(0)] = \indica{A}$ due to \ref{transcoherence: bounds}.
To prove the induction step, suppose that $[T_{\indica{A}}^k(0)] \geq k\indica{A}- \tfrac{k(k-1)}{2}\epsilon$ for some $k\in\nats$.
Then, using the linearity, monotonicity and constant additivity of $T$ and the fact that $\tfrac{k(k-1)}{2} = \sum_{\ell=1}^{k-1}\ell$, we find that 
\begin{multline*}
[T_{\indica{A}}^{k+1}(0)]
= \indica{A} + T [T_{\indica{A}}^{k}(0)]
\geq \indica{A} + T [k\indica{A} - \tfrac{k(k-1)}{2}\epsilon]
= \indica{A} + k T\indica{A} - \tfrac{k(k-1)}{2}\epsilon
\geq \indica{A} + k (\indica{A}-\epsilon) - \tfrac{k(k-1)}{2}\epsilon \\
= (k+1)\indica{A} - k\epsilon - \sum_{\ell=1}^{k-1}\ell\epsilon 
= (k+1)\indica{A} - \tfrac{k(k+1)}{2}\epsilon,
\end{multline*}
which establishes the induction step and therefore, the fact that $[T_{\indica{A}}^k(0)] \geq k\indica{A}- \tfrac{k(k-1)}{2}\epsilon$ for all $k\in\nats$.
Hence, for any $x\in A$, we have that $\sup_{T'\in\settrans} \tfrac{1}{k}[(T'_{\indica{A}})^k(0)](x) \geq \tfrac{1}{k}[T_{\indica{A}}^k(0)](x) \geq 1 - \tfrac{(k-1)}{2}\epsilon$ for all $k\in\nats$.
Since this is true for any $\epsilon>0$,
we infer that $\sup_{T\in\settrans} \tfrac{1}{k}[T_{\indica{A}}^k(0)](x) \geq 1$ for all $k\in\nats$ and therefore, by Equation~\eqref{Eq: recursive formula rep. independence}, that 
\begin{align*}
\liminf_{k\to+\infty} \avriupprev{k}(\indica{A}\vert x) \geq 1 \text{ for all } x\in A.
\end{align*}

On the other hand, we have that $\liminf_{k \to +\infty} \avriupprev{k}(\indica{A}\vert y) \leq 0$ for all $y\in\mathcal{R}$.
Indeed, because $A \subseteq \mathcal{R}^c$, we have that $\indica{A} \indica{\mathcal{R}} = 0 = 0 \, \indica{\mathcal{R}}$, and since $\mathcal{R}$ is a closed communication class in $\mathscr{G}(\uptrans{})$, Lemma~\ref{lemma: average does not depend on f outside S} implies that $\tfrac{1}{k}[\avuptrans{\indica{A}}{k}(0)] \indica{\mathcal{R}} = \tfrac{1}{k}[\avuptrans{0}{k}(0)] \indica{\mathcal{R}}$ for all $k \in \nats{}$.
Since $\tfrac{1}{k}[\avuptrans{0}{k}(0)]=0$ due to \ref{transcoherence: bounds}, we therefore have that $\tfrac{1}{k}[\avuptrans{\indica{A}}{k}(0)](y) = 0$ for all $y\in\mathcal{R}$ and all $k\in\nats$.
Moreover, for any $T\in\settrans$, note that $Th \leq \uptrans{} h$ and therefore $T_{\,\indica{A}} h \leq \avuptrans{\indica{A}}{} h$ for all $h\in\setofgambles{}(\statespace{})$, which implies by the monotonicity [\ref{topical: monotonicity}] of $T_{\,\indica{A}}$ that $[T_{\,\indica{A}}^{k}(0)]\leq[\avuptrans{\indica{A}}{k}(0)]$ for all $k\in\nats$.
Together with our previous considerations, we find that $\tfrac{1}{k}[T_{\indica{A}}^{k}(0)](y)\leq\tfrac{1}{k}[\avuptrans{\indica{A}}{k}(0)](y) = 0$ for any $T\in\settrans$, all $y\in\mathcal{R}$ and all $k\in\nats$.
As a result, 
\begin{align*}
\liminf_{k \to +\infty} \avriupprev{k}(\indica{A}\vert y)
\overset{\text{\eqref{Eq: recursive formula rep. independence}}}{=} \liminf_{k \to +\infty} \sup_{T\in\settrans} \tfrac{1}{k}[T_{\indica{A}}^{k}(0)](y) \leq  0 \text{ for all } y\in\mathcal{R}.
\end{align*}
So, we have that $\liminf_{k\to+\infty} \avriupprev{k}(\indica{A}\vert x) \geq 1$ and $\liminf_{k \to +\infty} \avriupprev{k}(\indica{A}\vert y) \leq 0$ for all $x\in A$ and all $y\in\mathcal{R}$, implying that $\rimarkov{\settrans}$ cannot be weakly ergodic.
\end{proofof}

\end{document}